\documentclass{amsbook}
\usepackage{amsthm,amsmath,amsfonts,latexsym,amssymb,mathrsfs,color,changepage,graphicx}

\makeatletter
\newcommand* \bigcdot{\mathpalette \bigcdot@{.5}}
\newcommand* \bigcdot@[2]{\mathbin{\vcenter{\hbox{\scalebox{#2}{$\m@th#1\bullet$}}}}}
\makeatother

\theoremstyle{definition}
\newtheorem{Def}{Definition}[chapter]
\newtheorem{df}[Def]{Definition}
\newtheorem{thm}[Def]{Theorem}
\newtheorem{cor}[Def]{Corollary}
\newtheorem{lem}[Def]{Lemma}
\newtheorem{prop}[Def]{Proposition}

\newtheorem{rem}[Def]{Remark}
\newtheorem{ex}[Def]{Example}
\newtheorem{q}[Def]{Question}

\newtheorem{prob}[Def]{Problem}
\newtheorem{question}[Def]{Question}

\renewcommand{\thefootnote}{\fnsymbol{footnote}}

\numberwithin{section}{chapter}
\numberwithin{equation}{chapter}

\begin{document}
\frontmatter

\title[The Beurling-Lax-Halmos Theorem for Infinite Multiplicity]
  {\bigskip
  The Beurling-Lax-Halmos Theorem \\ for Infinite Multiplicity\\ {\rm\sl \Large October 14, 2019}}

\author{\Large Ra{\'u}l\ E.\ Curto}

\author{\Large In Sung Hwang}

\author{\Large Woo Young Lee}

\maketitle

%
%
%
%

\setcounter{page}{4}\tableofcontents

\chapter*{Abstract}

In this paper, we consider several questions emerging from the
Beurling-Lax-Halmos Theorem, which characterizes the shift-invariant
subspaces of vector-valued Hardy spaces. \ The Beurling-Lax-Halmos Theorem states that a backward shift-invariant subspace is a model space $\mathcal{H}(\Delta) \equiv H_E^2 \ominus \Delta H_{E}^2$, for some inner function $\Delta$. \ Our first question calls for a description of the set $F$ in $H_E^2$ such that $\mathcal{H}(\Delta)=E_F^*$, where $E_F^*$ denotes the smallest backward shift-invariant subspace containing the set $F$.

In our pursuit of a general solution to this question, we are naturally led to take into account a canonical decomposition of operator-valued strong $L^2$-functions. \ This decomposition reduces to the
Douglas-Shapiro-Shields factorization if the flip of the strong
$L^2$-function is of bounded type. \ (Given a strong $L^2$-function $\Phi$, we define its {\it flip} by $\breve\Phi(z):=\Phi(\overline z)$.) 

Next, we ask: Is every shift-invariant subspace the kernel of a (possibly unbounded) 
Hankel operator\,? \ As we know, the kernel of a Hankel operator is shift-invariant, so the above question is equivalent to seeking a solution to the equation $\ker H_{\Phi}^*=\Delta H_{E^{\prime}}^2$, where $\Delta$ is an inner function satisfying $\Delta^* \Delta=I_{E^{\prime}}$ almost everywhere on the unit circle $\mathbb{T}$ and $H_{\Phi}$ denotes the Hankel operator with symbol $\Phi$.

\ Consideration of the above question on the structure of shift-invariant subspaces leads us to study and coin a new notion of ``Beurling degree" for an inner function. \ We then establish a deep connection between the
spectral multiplicity of the model operator, i.e., the truncated
backward shift on the corresponding  model space, and the Beurling
degree of the corresponding characteristic function. \ 

At the same time, we consider the notion of meromorphic pseudo-continuations of bounded type for operator-valued functions, and then use this notion to study
the spectral multiplicity of model operators (truncated backward
shifts) between separable complex Hilbert spaces. \ In particular,
we consider the case of multiplicity-free: more precisely, for which
characteristic function $\Delta$ of the  model operator
 $T$ does it follow that $T$ is multiplicity-free, i.e.,
$T$ has multiplicity 1 ? \ We show that if $\Delta$ has a meromorphic
pseudo-continuation of bounded type in the complement of the closed
unit disk and the adjoint of the flip of $\Delta$ is an outer
function then $T$ is multiplicity-free. \ 

In the case when the
characteristic function $\Delta$ of the model operator $T$ has a
finite-dimensional domain (in particular, when $\Delta$ is an inner
matrix function) admitting a meromorphic pseudo-continuation of bounded type in the
complement of the closed unit disk, we prove that the spectral multiplicity of $T$ can be
computed from that of the induced $C_0$-contraction, and as a result the characteristic function is two-sided inner. \ Finally, by using the
preceding results we analyze left and right coprimeness, the
model operator, and an interpolation problem for operator-valued
functions.


\renewcommand{\thefootnote}{}
\footnote{ 2010 \textit{Mathematics Subject Classification.} Primary
46E40, 47B35, 30H10, 30J05; Secondary 43A15, 47A15
\\
\smallskip
\indent\textit{Key words.} The Beurling-Lax-Halmos Theorem, strong
$L^2$-functions, a canonical decomposition, the
Douglas-Shapiro-Shields factorization, a complementary factor of an
inner function, the degree of non-cyclicity, functions of bounded
type, the Beurling degree, the spectral multiplicity, the model operator,
characteristic inner functions, meromorphic
pseudo-continuation of bounded type, multiplicity-free.
\\
\smallskip
\indent
The work of the
second named author was supported by NRF(Korea) grant No.
2019R1A2C1005182.  The work of the third named author was supported
by NRF(Korea) grant No. 2018R1A2B6004116.
\\
\smallskip
\indent
Affiliations:\\
\indent Ra{\'u}l\ E.\ Curto: Department of Mathematics, University
of Iowa, Iowa City, IA 52242, U.S.A.
email: raul-curto@uiowa.edu\\
\indent In Sung Hwang: Department of Mathematics, Sungkyunkwan
University, Suwon 440-746, Korea,
email: ihwang@skku.edu\\
\indent Woo Young Lee: Department of Mathematics, Seoul National
University, Seoul 151-742, Korea, email: wylee@snu.ac.kr
}

\bigskip

%
%
%
%
%
%

\mainmatter


\maketitle

%
%
%
%

\chapter{Introduction}

\medskip

The celebrated Beurling Theorem \cite{Be} characterizes the
shift-invariant subspaces of the Hardy space. \ P.D. Lax \cite{La}
extended the Beurling Theorem to the case of finite
multiplicity, and proved the so-called Beurling-Lax Theorem. \ Subsequently, P.R.
Halmos \cite{Ha} gave a beautiful proof for the case of
infinite multiplicity, and thus established the so-called Beurling-Lax-Halmos
Theorem. \ Since then, the Beurling-Lax-Halmos Theorem has been
extended to various settings and extensively applied in
connection with model theory, system theory and the
interpolation problem by many authors (cf. \cite{ADR},
\cite{AS}, \cite{BH1}, \cite{BH2}, \cite{BH3}, \cite{Ca}, \cite{dR},
\cite{He}, \cite{Po}, \cite{Ri}, \cite{SFBK}). \ 

In this paper,
we will focus on a detailed analysis of the Beurling-Lax-Halmos Theorem for infinite
multiplicity. \ We obtain answers to several questions emerging
from the classical Beurling-Lax-Halmos Theorem and establish some new and exciting results, including: (i) a
canonical decomposition for operator-valued $L^2$-functions (in
fact, for a much bigger class of functions), (ii) the introduction of the Beurling degree of an inner
function, and (iii) the study of the spectral multiplicity of a model operator. \ 

Let
$\mathbb T$ be the unit circle in the
complex plane $\mathbb C$. \ Throughout this paper, whenever we deal
with operator-valued functions $\Phi$ on $\mathbb T$, we assume that
$\Phi(z)$ is a bounded linear operator between separable complex
Hilbert spaces for almost all $z \in \mathbb T$. \
For a separable complex Hilbert space
$E$, let $S_{E}$ be the shift operator on the $E$-valued Hardy space
$H^2_{E}$, i.e.,
$$
(S_{E}f)(z):=zf(z) \quad \hbox{for each} \ f \in H^2_E.
$$
The Beurling-Lax-Halmos Theorem states that every subspace $M$ invariant under $S_{E}$ (i.e., a closed subspace of $H^2_{E}$
such that $S_{E} f\in M$ for all $f \in M$) is of the form $\Delta
H^2_{E^\prime}$, where $E^\prime$ is a closed subspace of $E$ and
$\Delta$ is an {\it inner} function. \ As usual, $\Delta$ is an inner function if $\Delta(z)$ is
an isometric operator from $E^\prime$ into $E$ for almost all
$z\in\mathbb T$, i.e., $\Delta^*\Delta=I_{E^\prime}$ a.e. on
$\mathbb T$. \ If, in addition, $\Delta\Delta^*= I_{E}$ a.e. on $\mathbb
T$, then $\Delta$ is called a {\it two-sided} inner function. \

There exists an equivalent description of a closed subspace $M$ of $H^2_{E}$ which is invariant under
the backward shift operator $S^*_{E}$; that is, $M=\mathcal H(\Delta):= H^2_{E}\ominus \Delta H^2_{E^\prime}$ for
some inner function $\Delta$. \ The space $\mathcal H(\Delta)$ is often
called a model space or a de Branges-Rovnyak space \cite{dR},
\cite{Sa}, \cite{SFBK}. \ Thus, for a subset $F$ of $H^2_{E}$, if
$E_F^*$ denotes the smallest $S^*_{E}$-invariant subspace containing
$F$, i.e.,
$$
E_F^*:=\bigvee\bigl\{S^{*n}_{E} F: n\ge 0\bigr\},
$$
(where $\bigvee$ denotes the closed linear span), then $E_F^*=\mathcal
H(\Delta)$ for some inner function $\Delta$.  \ 

Now, given a
backward shift-invariant subspace $\mathcal H(\Delta)$, we may ask:

\begin{question} \label{mainq}
(i) \ What is
the smallest number of vectors in $F$ satisfying $\mathcal
H(\Delta)=E_F^*$ ? \newline
(ii) \ More generally, we are interested in the problem of describing the
set $F$ in $H^2_E$ such that $\mathcal H(\Delta)=E_F^*$. \label{mainmain}  
\end{question}

To examine Question \ref{mainq} we need to consider (possibly unbounded) linear operators
(defined on the unit circle) constructed by arranging the vectors in $F$ as column vectors. \ In other words, in what follows we will encounter bounded linear operators whose ``column" vectors are $L^2$-functions. \ (Since bounded linear operators
between separable Hilbert spaces can be represented as infinite matrices, considering the columns of such a matrix as column vectors of the operator seems well justified). \ This approach naturally leads to the notion of (operator-valued) strong $L^2$-function. \ This notion seems to have been introduced by V. Peller \cite[Appendix 2.3]{Pe} for the purpose of defining general symbols of vectorial Hankel operators. \ However, Peller’s book gives only the definition of a strong $L^2$-function, and does not describe the properties of such functions. \ Besides Peller’s book, we have not found any other references in the literature to strong $L^2$-functions. \ In Chapter 3 we study strong $L^2$-functions (including operator-valued $L^2$- and $L^{\infty}$-functions) and then derive some basic properties. 

Let $\mathcal B(D,E)$ denote the
set of all bounded linear operators between separable complex
Hilbert spaces $D$ and $E$. \ A {\it strong $L^2$-function}
$\Phi$ is a $\mathcal B(D,E)$-valued function defined almost
everywhere on the unit circle $\mathbb T$ such that $\Phi(\cdot)x\in
L^2_E$ for each $x\in D$. \ We can easily see that
 every operator-valued $L^p$-function ($p\geq 2$) is a strong
$L^2$-function (cf.~p.\pageref{lps}). \ Following V. Peller
\cite{Pe}, we write $L^2_s(\mathcal B(D,E))$ for the set of  strong
$L^2$-functions with values in $\mathcal B(D,E)$. \ 

The set
$L^2_s(\mathcal B(D,E))$ constitutes a nice collection of general symbols of
vectorial Hankel operators (see \cite{Pe}). \ Similarly, we write
$H^2_s(\mathcal B(D,E))$ for the set of strong $L^2$-functions with
values in $\mathcal B(D,E)$ such that $\Phi(\cdot)x\in H^2_E$ for
each $x\in D$. \  Of course, $H^2_s(\mathcal B(D,E))$ contains all
$\mathcal B(D,E)$-valued $H^2$-functions. \ In Chapter 3, we study
operator-valued Hardy classes as well as strong $L^2$-functions as a
groundwork of this paper. \

Question \ref{mainq} is closely related to a canonical
decomposition of strong $L^2$-functions. \ We first observe that if
$\Phi$ is an operator-valued $L^\infty$-function, then the kernel of
the Hankel operator $H_{\Phi^*}$ is shift-invariant. \ Thus by the
Beurling-Lax-Halmos Theorem, the kernel of the Hankel operator
$H_{\Phi^*}$ is of the form $\Delta H^2_{E^\prime}$ for some inner
function $\Delta$. \ If the kernel of the Hankel operator $H_{\Phi^*}$
is trivial, take $E^{\prime}=\{0\}$. \ Of course, $\Delta$ need not
be a two-sided inner function. \ In fact, we can show that if $\Phi$
is an operator-valued $L^{\infty}$-function and $\Delta$ is a
two-sided inner function, then the kernel of the Hankel operator
$H_{\Phi^*}$ is $\Delta H^2_{E^\prime}$ if and only if $\Phi$ is
expressed in the form
\begin{equation}\label{DSSF}
\Phi=\Delta A^*,
\end{equation}
where $A$ is an operator-valued $H^\infty$-function such that
$\Delta$ and $A$ are right coprime (see Lemma \ref{rem2.4}). \ The
expression (\ref{DSSF}) is called the (canonical) {\it Douglas-Shapiro-Shields
factorization} of an operator-valued $L^\infty$-function $\Phi$ (see
\cite{DSS}, \cite{FB}, \cite{Fu}; in particular, \cite{Fu} contains
many important applications of the Douglas-Shapiro-Shields
factorization to linear system theory). \ 

Let $\mathbb D$ be the
open unit disk in the complex plane $\mathbb C$. \ We recall that a
meromorphic function $\varphi:\mathbb D \rightarrow \mathbb C$ is
said to be {\it of bounded type} (or {\it in the Nevanlinna class})
if it is a quotient of two bounded analytic functions. \  A matrix
function of bounded type is defined by a matrix-valued function
whose entries are all of bounded type.\label{matrixbt} \  Very recently,
a systematic study on matrix-valued functions of bounded type was
undertaken in the research monograph \cite{CHL3}. \ It is also known that every
matrix-valued $L^\infty$-function whose adjoint is of bounded type
satisfies (\ref{DSSF}) (cf. \cite{GHR}). \ In fact, if
we extend the notion of ``bounded type'' for operator-valued
$L^\infty$-functions (as we will do in Definition \ref{dfbundd}
for a bigger class), then we may say that the expression
(\ref{DSSF}) characterizes the class of $L^\infty$-functions whose flips are
of bounded type, where the flip $\breve\Phi$ of $\Phi$ is defined by
$\breve\Phi(z):=\Phi(\overline z)$. \ From this viewpoint, we may
ask whether there exists an appropriate decomposition corresponding
to general $L^\infty$-functions, more generally, to strong
$L^2$-functions. \ The following problem is the first objective of this
paper,
\begin{prob}
Find a canonical decomposition of strong $L^2$-functions.
\end{prob}
To establish  a canonical decomposition of strong $L^2$-functions,
we need to introduce new notions; this will be done in Chapter 4. \
First of all, we coin the notion of ``complementary factor",
denoted by $\Delta_c$, of an inner function $\Delta$ with values in
$\mathcal B(D, E)$. \ This notion is defined by using the kernel of
$\Delta^*$, denoted by $\ker \Delta^*$, which is defined by the set
of vectors $f$ in $H^2_{E}$ such that $\Delta^*f=0$ a.e. on $\mathbb
T$. \ Moreover, the kernel of $H_{\Delta^*}$ can be represented by
orthogonally adding the complementary factor $\Delta_c$ to $\Delta$ (see
Lemma \ref{thm2.9}). \ We also employ a notion of ``degree of
non-cyclicity" on the set of all subsets (or vectors) of $H^2_E$,
which is a complementary notion of ``degree of cyclicity" due to
V.I. ~Vasyunin and N.K. ~Nikolskii \cite{VN}. \ The {\it degree of
non-cyclicity}, denoted by $\hbox{nc} (F)$, of subsets $F \subseteq
H^2_{E}$, is defined by the number
\begin{equation}\label{defncnumber}
\hbox{nc}(F):=\sup_{\zeta\in\mathbb D}\dim\bigl\{g(\zeta): g\in H^2_E
\ominus E_F^*\bigr\}.
\end{equation}
Thus, in comparison with the degree of cyclicity, the degree of
non-cyclicity admits $\infty$, which is often beneficial when trying to understand the Beurling-Lax-Halmos Theorem. \ Now, for a canonical
decomposition of strong $L^2$-functions $\Phi$, we are tempted to
guess that $\Phi$ can be factored as $\Delta A^*$ (where
$\Delta$ is a possibly one-sided inner function) as in the Douglas-Shapiro-Shields factorization, in which $\Delta$ is
two-sided inner. \ But this is not the case. In fact, we can see
that a canonical decomposition is actually affected by the kernel of
$\Delta^*$ through some examples (see p.~\pageref{7.ex}). \ Upon reflection,
we recognize that this is not an accident. \ This is accomplished in
Chapter 5. \ 

Theorem \ref{vectormaintheorem} realizes the idea
inside those examples: if $\Phi$ is a strong $L^2$-function with
values in $\mathcal B(D, E)$, then $\Phi$ can be expressed in the
form
\begin{equation}\label{maincanonical}
\Phi=\Delta A^*+B,
\end{equation}
where  $\Delta$ is an inner function with values in $\mathcal
B(E^{\prime}, E)$, $\Delta$ and $A$ are right coprime, $\Delta^*B
=0$, and $\hbox{\rm nc}\{\Phi_+\}\le \hbox{\rm dim}\, E^\prime$. \ ($\{\Phi_+\}$ denotes the set of all ``column" vectors of the
analytic part of $\Phi$). \ In particular, if $\dim E^\prime<\infty$ (for instance, if $\dim E<\infty${\rm )}, then the expression
{\rm (\ref{maincanonical})} is unique (up to a unitary constant
right factor) (see Theorem \ref{vectormaintheorem},
p.~\pageref{pfa}). \ The expression (\ref{maincanonical})  will be
called {\it a canonical decomposition} of a strong $L^2$-function
$\Phi$. \ The proof of Theorem \ref{vectormaintheorem} shows that
the inner function $\Delta$ in the canonical decomposition
(\ref{maincanonical}) of a strong $L^2$-function $\Phi$ can be
obtained from the equation
$$
\ker H_{\breve{\Phi}}^*=\Delta H^2_{E^{\prime}}
$$
which is guaranteed by the Beurling-Lax-Halmos Theorem (see
Corollary \ref{kerhadjoint}). \ In this case, the expression
(\ref{maincanonical}) will be called the {\it BLH-canonical
decomposition} of $\Phi$, recalling that $\Delta$ comes from
the Beurling-Lax-Halmos Theorem. \ However, if $\dim
E^\prime=\infty$ (even in the case when $\dim D<\infty$), then it is possible
to get another inner function $\Theta$ of a canonical decomposition
(\ref{maincanonical}) for the same function: in this case, $\ker
H_{\breve{\Phi}}^* \ne \Theta H^2_{E^{\prime\prime}}$. \ Therefore
the canonical decomposition of a strong $L^2$-function is not unique
in general (see Remark \ref{canoexam}). \ But the second assertion
of Theorem \ref{vectormaintheorem} says that if the codomain of
$\Phi(z)$ is finite-dimensional (in particular, if $\Phi$ is a
matrix-valued $L^2$-function), then the canonical decomposition
(\ref{maincanonical}) of $\Phi$ is unique; in other words, the inner
function $\Delta$ in (\ref{maincanonical}) should be obtained from
the equation $\ker H_{\breve{\Phi}}^*=\Delta H^2_{E^{\prime}}$. \
Thus the unique canonical decomposition (\ref{maincanonical}) of
matrix-valued $L^2$-functions is precisely the BLH-canonical
decomposition. \ 

Further, if the flip $\breve{\Phi}$ of $\Phi$ is of
bounded type then $B$ turns to be a zero function, so that the
decomposition (\ref{maincanonical}) reduces to the
Douglas-Shapiro-Shields factorization. \ In fact, the
Douglas-Shapiro-Shields factorization was given for
$L^\infty$-functions, but the case $B=0$ in (\ref{maincanonical}) is
available for strong $L^2$-functions. \ Moreover, the notion of
``bounded type" for matrix-valued functions is not appropriate for
operator-valued functions, i.e., the statement ``each entry of the
matrix is of bounded type" does not produce a natural extension to
operator-valued functions even though it has a meaning for infinite
matrices (remember that we deal with operators between separable Hilbert
spaces). \ 

Thus we need to introduce an appropriate notion of
``bounded type" for operator-valued functions. \ We will do this in
Section 4.4. \ Moreover, to guarantee the statement ``each entry is
of bounded type," we adopt the notion of ``meromorphic
pseudo-continuation of bounded type" in $\mathbb{D}^e := \{z:
1<|z|\le\infty\}$, which coincides with the notion of ``bounded
type" for matrix-valued functions (cf. \cite{Fu1}): This will be
done in Section 4.5. \

On the other hand, we recall that
the {\it spectral multiplicity} for a bounded linear operator $T$
acting on a separable complex Hilbert space $E$ is defined by the number
$\mu_T$:
$$
\mu_T :=\inf \dim F,
$$
where $F\subseteq E$, the infimum being taken over all generating
subspaces $F$, i.e., subspaces such that $M_F\equiv\bigvee
\{T^nF:n\ge 0\}=E$.  \ In the definition of the spectral
multiplicity, $F$ may be taken as a subset rather than a subspace. \
In this case, we may regard $\mu_T$ as the quantity $\inf\dim \,\bigvee \{f: f\in F\}$
such that $M_F=E$. \ Unless this leads to ambiguity, we will deal
with $M_F$ for subsets $F\subseteq E$. \ If $S_E$ is the shift
operator on $H^2_{E}$, then it is known that $\mu_{S_{E}}=\dim E$.
\ By contrast, if $S^*_{E}$ is the backward shift operator on
$H^2_E$, then $S^*_{E}$ has a cyclic vector, i.e.,
$\mu_{S^*_{E}}=1$. Moreover, the cyclic vectors of $S^*_{E}$ form a
dense subset of $H^2_{E}$ (see \cite{Ha2}, \cite{Ni1}, \cite{Wo}). \
We here observe that Question \ref{mainq}(i) is identical to the
problem of finding the spectral multiplicity of the truncated
backward shift operator $S_{E}^*\vert_{\mathcal H(\Delta)}$, i.e.,
the restriction of $S^*_{E}$ to its invariant subspace $\mathcal
H(\Delta)$. \ The second objective of this paper is to show that this
problem has a deep connection with a canonical decomposition of
strong $L^2$-functions involved with the inner function $\Delta$. \

To understand the smallest $S_{E}^*$-invariant subspace containing a
subset $F\subseteq H^2_{E}$, we need to consider the kernels of the
adjoints of unbounded Hankel operators with  strong $L^2$-symbols
involved with $F$. \ Thus we will deal with unbounded Hankel
operators $H_\Phi$ with strong $L^2$-symbols $\Phi$. \ However, the
adjoint of the unbounded Hankel operator need not be a Hankel
operator. \ But if $\Phi$ is an $L^\infty$-function then
$H_{\Phi^*}=H_{\breve\Phi}^*$, where $\breve\Phi$ is the flip of
$\Phi$. \ Thus for a bounded symbol $\Phi$, we may use the notations
$H_{\Phi^*}$ and $H_{\breve\Phi}^*$ interchangeably. \ By contrast,  for a strong $L^2$-function $\Phi$,
$H_{\Phi^*}$ may not be equal to $H_{\breve\Phi}^*$ even though
$\Phi^*$ is a strong $L^2$-function. \ In particular, the kernel of
an unbounded Hankel operator $H_{\Phi^*}$ is likely to be trivial
because it is defined on the dense subset of polynomials. \ From
this viewpoint, to avoid potential technical issues in our arguments, we will deal
with the operator $H_{\breve\Phi}^*$ in place of $H_{\Phi^*}$. \ In
spite of this, and since the kernel of the adjoint of an unbounded
operator is always closed, we can show that via the
Beurling-Lax-Halmos Theorem, the kernel of $H^*_{\breve{\Phi}}$ with
strong $L^2$-symbol $\Phi$ is still of the form $\Delta
H^2_{E^\prime}$ (see Corollary \ref{kerhadjoint}). \

We now consider several questions, which are of independent interest. \
This will be done in Chapter 4. The next question arises naturally
from the Beurling-Lax-Halmos Theorem. \

\begin{q}\label{q111}
Since the kernel of the Hankel operator $H_{\breve\Phi}^*$ is of the form
$\Theta H^2_{E^\prime}$, which property of $\Phi$ determines the dimension of the space
$E^\prime$\,? \ In particular, if $\Phi$ is an $n\times m$
matrix-valued $L^2$-function and $\dim E^\prime=r$, which property of $\Phi$ determines the number $r$ ?
\end{q}

To answer Question \ref{q111}, we employ the notion of degree of
non-cyclicity (\ref{defncnumber}). \ Indeed, we can show that if the
kernel of the adjoint of the Hankel operator $H_{\breve{\Phi}}$ is
$\Theta H^2_{E^\prime}$ for some inner function $\Theta$, then  the
dimension of $E^\prime$ can be computed by the degree of
non-cyclicity of $\{\Phi_+\}$ (see Theorem \ref{thm7566}). \ Here we
note that the definition of $\{\Phi_+\}$ depends on the orthonormal bases of the
domain $D$ of $\Phi(\cdot)$. \ However, the degree of non-cyclicity
of $\{\Phi_+\}$ is independent of the particular choice of orthonormal basis of $D$ (see Theorem \ref{thm7566}). \

When $\Delta$ is an inner function, we may ask when it is possible to complement $\Delta$ to a two-sided inner function by aid of an
inner function $\Omega$; in other words, when is $[\Delta, \Omega]$ a two-sided inner function, where $[\Delta(\cdot), \Omega(\cdot)]$ is
understood as an $1\times 2$ operator matrix defined on the
unit circle $\mathbb T$ ? \ (It turns out that this question
can be answered by using the Complementing Lemma; see \cite{VN} or
\cite{Ni1}). \ The following question refers to more general
cases.

\begin{q}\label{q222}
If $\Delta$ is an $n\times r$ inner matrix function, which condition
on $\Delta$ allows us to complement $\Delta$ to an $n\times (r+q)$
inner matrix function using an $n\times q$ inner matrix function
?
\end{q}

An answer to Question \ref{q222} is also subject to the degree of non-cyclicity
of $\{\Delta\}$ (see Corollary \ref{cor5.hh3333}). \

By the Beurling-Lax-Halmos Theorem, we saw that the kernel of the
adjoint of a Hankel operator with  a strong $L^2$-symbol is of the form
$\Delta H^2_{E^\prime}$ for some inner function $\Delta$. \ In view
of its converse, we may ask:

\begin{q}\label{q333}
Is every shift-invariant subspace $\Delta H^2_{E^\prime}$
represented by the kernel of
$H_{\breve{\Phi}}^*$ with some strong $L^2$-symbol $\Phi$ with values in
$\mathcal B(D,E)$ ?
\end{q}

Question \ref{q333} asks whether a strong $L^{2}$-solution
$\Phi$ always exists for the equation $\ker H_{\breve\Phi}^*=\Delta
H^2_{E^{\prime}}$ for a given inner function $\Delta$. \ In Theorem \ref{kkknvnv} we give an affirmative answer to Question \ref{q333}. \ The matrix-valued version of this result is as
follows (see Corollary \ref{existrem}): for a given $n\times r$
inner matrix function $\Delta$, there always exists a solution
$\Phi\in L^\infty_{M_{n\times m}}$ of the
equation $\ker H_{\breve\Phi}^*=\Delta H^2_{\mathbb C^r}$, for some $m\leq r+1$. \ In view of this, it is reasonable to ask whether such a solution $\Phi\in
L^2_{M_{n\times m}}$ exists for each $m=1,2,\cdots$. \ But the answer to this question is negative
(see Remark \ref{remark7.9}). \ 

It is then natural to ask how
to determine a possible dimension of $D$ for which there exists a
strong $L^2$-solution $\Phi$ (with values in $\mathcal B(D,E)$) of
the  equation $\ker H_{\breve\Phi}^*=\Delta H^2_{E^\prime}$. \ In
fact, we would like to ask what is the infimum of $\dim D$ that guarantees the existence of a strong $L^2$-solution $\Phi$. \ To find a way to determine such an infimum, we
introduce the notion of ``Beurling degree" for an inner function. \ We do this by employing the canonical decomposition of a strong $L^2$-function
induced by the given inner function: if $\Delta$ is an inner
function with values in $\mathcal B(E^\prime,E)$,  then the {\it
Beurling degree}, denoted by $\hbox{deg}_{B} (\Delta)$, of $\Delta$
is defined by the infimum of the dimension of the nonzero space $D$ for
which there exists a pair $(A,B)$ such that $\Phi\equiv \Delta A^*+B$
is a canonical decomposition of a strong $L^2$-function $\Phi$ with
values in $\mathcal B(D,E)$ (Definition \ref{df of degree}). \

We now recall that the Model Theorem\label{MT} (\cite{Ni1},
\cite{SFBK}) states that if a bounded operator $T$ acting on a Hilbert space $\mathcal{H}$ (in symbols, $T\in \mathcal {B(H)}$) is a contraction
(i.e., $||T||\le 1$) satisfying 
\begin{equation}\label{stronglimit}
\lim_{n\to\infty} T^n x=0\quad \hbox{for each $x\in \mathcal{H}$},
\end{equation}
then $T$ is unitarily equivalent to a truncated backward shift
$S_{E}^*\vert_{\mathcal H(\Delta)}$ for some inner function $\Delta$
with values in $\mathcal B(E^{\prime}, E)$, where
$E=\hbox{cl}\,\hbox{ran} (I-T^*T)$. \ In this case,
$S_{E}^*\vert_{\mathcal H(\Delta)}$ is called the the {\it model
operator} of $T$ and $\Delta$ is called the {\it characteristic
function} of $T$. \ \ We often write $T\in C_{0\, \bigcdot}$ for a contraction operator
$T\in\mathcal {B(H)}$ satisfying the condition (\ref{stronglimit}). \ 

We can now prove that if $\Delta$ is the
characteristic function of the model operator $T$ with values in
$\mathcal B(E^\prime,E)$, with $\dim E^\prime<\infty$ (in
particular, when $\Delta$ is an inner matrix function), then the spectral
multiplicity of the model operator is equal to the Beurling degree
of $\Delta$. \ Equivalently, given an inner function $\Delta$ with
values in $\mathcal B(E^\prime, E)$, with $\dim E^\prime<\infty$,
let $T:=S_{E}^*|_{\mathcal H(\Delta)}$. \ Then
\begin{equation}\label{main_222}
\mu_T=\hbox{\rm deg}_{B}(\Delta)
\end{equation}
(see Theorem \ref{maintheorem_sm}). \ The equality (\ref{main_222})
is the second objective of this paper. \ It is somewhat surprising that
the spectral multiplicity of the model operator can be computed by a
function-theoretic property of the corresponding characteristic
function. \

The third objective of this paper is to consider the case of $\mu_T=1$,
i.e., when the operator $T$ has a cyclic vector. \ In general, if
$T\in\mathcal {B(H)}$ is such that $\mu_T=1$, then $T$ is said to be
{\it multiplicity-free}. \ To avoid confusion, we regard $T$ to be
multiplicity-free if the operator $T$ acts on the zero space. \ Thus
we are interested in the following question on  the characteristic
function $\Delta$ of $T$.

\begin{q}\label{q444}
Let $T:=S_{E}^*|_{\mathcal H(\Delta)}$.
For which inner function $\Delta$ does it follow that $T$ is multiplicity-free?
\end{q}
To get an answer to Question \ref{q444}, we consider the notion of
``characteristic scalar" inner function, which
is a generalization of the case of two-sided inner matrix function
(and we often call it {\it square inner} matrix function) (cf.
\cite{Hel}, \cite{SFBK}, \cite{CHL3}). \ This will be done in
Section 7.1. \ If $\Delta$ is an inner function and $\Delta_c$ is
its complementary factor, we write $\Delta_{cc}\equiv
(\Delta_c)_c$, $\Delta_{ccc}\equiv (\Delta_{cc})_c,\cdots$, etc. for
the successive iterated complementary factors of $\Delta$. \ The key idea for
an answer to Question \ref{q444} is given in the following result. \ First, let 
$\widetilde\Delta
(z):=\Delta(\overline z)^*$. 

\medskip
\begin{adjustwidth}{0.7cm}{}
If an
inner function $\Delta$ has a meromorphic pseudo-continuation of
bounded type in $\mathbb{D}^e$ and if $\widetilde\Delta$ is an outer function,
then $\Delta_{cc}=\Delta$  (see Lemma \ref{ffbfbbfbfbfvvvbbfLll}).
\end{adjustwidth}

\medskip
\noindent We can then get an answer to Question \ref{q444}, as follows: 

\medskip
\begin{adjustwidth}{0.7cm}{}
If $T:=
S_{E}^*\vert_{\mathcal H(\Delta)}$, where $\Delta$ has a meromorphic
pseudo-continuation of bounded type in $\mathbb{D}^e$ and $\widetilde{\Delta}$
is an outer function, then $T$ is multiplicity-free (see Theorem
\ref{wwwkslgjho}).
\end{adjustwidth}  

\medskip
\noindent Recall that for an inner matrix function $\Delta$,
the condition ``$\Delta$ has a meromorphic pseudo-continuation of
bounded type'' in $\mathbb{D}^e$ is equivalent to the condition
``\,$\breve{\Delta}$ is of bounded type'' (see Corollary
\ref{cor512,222}). \ As a consequence, the matrix-valued version of Theorem
\ref{wwwkslgjho} can be rephrased as follows: If $\Delta$ is an inner matrix
function whose flip $\breve{\Delta}$ is of bounded type and if
$\Delta^t$, the transpose of $\Delta$, is an outer function, then
$T:=S_E^*|_{\mathcal H(\Delta)}$ is multiplicity-fee (see Corollary
\ref{cor5.16}). \ We may ask whether the converse of the key idea
(Lemma \ref{ffbfbbfbfbfvvvbbfLll}) for Theorem \ref{wwwkslgjho} is
true; i.e., if $\Delta$ is an inner function having a meromorphic
pseudo-continuation of bounded type in $\mathbb{D}^e$ and
$\Delta_{cc}=\Delta$, does it follow that $\widetilde\Delta$ is an
outer function? \ We can show that the answer to this question is
affirmative when $\Delta$ is an inner matrix function: i.e., if
$\Delta_{cc}=\Delta$, then $\widetilde\Delta$ is an outer function
when $\Delta$ is an inner matrix function whose flip $\breve \Delta$
is of bounded type (see Corollary \ref{cor33fbfbbfLll}). \

On the other hand, the theory of spectral multiplicity for
$C_0$-operators has been well developed in terms of their
characteristic functions (cf. \cite[Appendix 1]{Ni1}). \ However
this theory is not applied directly to $C_{0\, \bigcdot}$-operators, in
which cases their characteristic functions need not be two-sided
inner. \ The fourth objective of this paper is to show that if the
characteristic function of a $C_{0\, \bigcdot}$-operator $T$ has a
finite-dimensional domain and a
meromorphic pseudo-continuation of bounded type in $\mathbb{D}^e$, then its
spectral multiplicity can be computed by that of the $C_0$-operator
induced by $T$. \ This will be done in Section 7.3. \ The main
theorem of that section is as follows: Given an inner function
$\Delta$ with values in $\mathcal B(E^\prime, E)$, with $\dim\,
E^\prime<\infty$, let $T:=S_{E}^*\vert_{\mathcal H(\Delta)}$. If
$\Delta$ has a meromorphic pseudo-continuation of bounded  type in
$\mathbb{D}^e$, then
\begin{equation}\label{maineqqq}
\mu_T=\mu_{T_s},
\end{equation}
where $T_s$ is a $C_0$-contraction of the form $T_s:=
S^*_{E^{\prime}}|_{\mathcal H(\Delta_s)}$ with $\Delta_s:=
\widetilde{(\widetilde\Delta)^i}$. \
Hence in particular, $\mu_T\le
\dim\, E^{\prime}$. (Here $(\cdot)^i$ means the inner part of the
inner-outer factorization of the given $H^\infty$-function.) (see
Theorem \ref{hwysonghkkl}). \ 

In Theorem \ref{hwysonghkkl}, we note
that $\Delta_s\equiv \widetilde{(\widetilde\Delta)^i}$ is a
two-sided inner function (see Lemma \ref{lemdchgvfervbhyu}) (and
hence, $T_s$ belongs to the class $C_0$). \ Therefore
(\ref{maineqqq}) shows that the spectral multiplicity of a
$C_{0\, \bigcdot}$-operator can be determined by the induced
$C_0$-operator if its characteristic function has a meromorphic
pseudo-continuation of bounded  type in $\mathbb{D}^e$. \ On the other hand,
it was known (cf. \cite[p.~41]{Ni1}) that if
$T:=S_{E}^*\vert_{\mathcal H(\Delta)}$ for an inner function
$\Delta$ with values in $\mathcal B(E^\prime, E)$, with
$\dim\,E^\prime<\dim\,E$, then
\begin{equation}\label{muTTTT}
\mu_T\le \dim\, E^\prime +1;
\end{equation}
if further $\dim\,E^\prime=\dim\,E<\infty$, then
\begin{equation}\label{muTTTTe}
\mu_T\le \dim\, E^\prime.
\end{equation}
Thus, the equation (\ref{maineqqq}) shows that (\ref{muTTTTe}) still
holds without the assumption $\dim\, E^\prime=\dim\, E$. \

\bigskip

The organization of this paper is as follows. \ The main theorems of
this paper are Theorem \ref{vectormaintheorem} (a canonical
decomposition of strong $L^2$-functions), Theorem
\ref{maintheorem_sm} (the Beurling degree and the spectral
multiplicity), Theorem \ref{wwwkslgjho} (multiplicity-free model
operators), and Theorem \ref{hwysonghkkl} (the spectral multiplicity
of model operators). \ To prove those theorems, we need to consider
several questions emerging from the Beurling-Lax-Halmos Theorem. \ We also consider several auxiliary lemmas, and new
notions of complementary factors of inner functions, the degree
of non-cyclicity, bounded type strong $L^2$-functions, and the
Beurling degree of an inner function. \ 

In Chapter 2 we give the
notations and the basic definitions. \ In Chapter 3 we study
operator-valued strong $L^2$-functions and then prove some
properties which will be used in the sequel. \ In Section 4.1-4.3
we introduce notions of complementary factors of inner functions and
the degree of non-cyclicity, and then give answers to Question
\ref{q111} and Question \ref{q222}. \ In Section 4.4 we introduce
the notion of ``bounded type" strong $L^2$-functions, which correspond to the functions whose entries are of bounded type in the
matrix-valued case.  \ 

In Chapter 5 we establish a canonical
decomposition of a strong $L^2$-functions $\Phi$, which reduces to
the Douglas-Shapiro-Shields factorization of $\Phi$ if
$\breve{\Phi}$ is of bounded type. \ In Chapter 6 we give an answer
to Question \ref{q333} and then establish a connection between the
spectral multiplicity of the model operator and the Beurling degree
of the corresponding characteristic function. \ 

In Chapter 7 we
consider the spectral multiplicity of model operators by using the
notion of meromorphic pseudo-continuation of bounded type in the
complement of the closed unit disk and then give an answer to
Question \ref{q444}. \ In Chapter 8 by using the preceding results,
we analyze the left and right coprimeness, the model operator  and
an interpolation problem for operator-valued functions. \ In Chapter
9 we address some unsolved problems.

%
%
%
%
%

\chapter{Preliminaries}

In this chapter we provide notations and definitions, which will be
used in this paper.

We write $\mathbb D$ for the open unit disk in the complex plane
$\mathbb C$ and $\mathbb T$ for the unit circle in $\mathbb C$. \ To
avoid a confusion, we will write $z$ for points on $\mathbb T$ and
$\zeta$ for points in $\mathbb C \setminus \mathbb T$. \ For $\phi
\in L^2$, write
$$
\breve{\phi}(z):=\phi(\overline{z})\quad \hbox{and} \quad
\widetilde{\phi}(z):=\overline{\phi(\overline z)}.
$$
For
$\phi \in L^2$, write
$$
\phi_+:= P_+\phi \quad \hbox{and} \quad \breve{\phi}_- :=
P_-\phi,
$$
where $P_+$ and $P_-$ are the orthogonal projections from $L^2$ onto
$H^2$ and $L^2\ominus H^2$, respectively. \ Thus, we may write
$\phi=\breve\phi_- + \phi_+. $

Throughout the paper, we assume that
$$
\begin{aligned}
X\ \hbox{and}\ Y \ &\hbox{are complex Banach spaces;}\\
D\ \hbox{and}\ E \ &\hbox{are separable complex Hilbert spaces}.
\end{aligned}
$$
We write $\mathcal B(X, Y)$ for the set of all bounded linear operators
from $X$ to $Y$ and abbreviate $\mathcal B(X, X)$ to $\mathcal
B(X)$. \ For a complex Banach space $X$, we write $X^*$ for its dual. \
We write $M_{n\times m}$ for the set of $n\times m$ complex matrices,
and abbreviate $M_{n\times n}$ to $M_{n}$. \ We also write $\hbox{g.c.d.}(\cdot)$ and $\hbox{l.c.m.}(\cdot)$ denote the greatest
common inner divisor and the least common inner multiple,
respectively, while $\hbox{left-g.c.d.}(\cdot)$ and
$\hbox{left-l.c.m.}(\cdot)$ denote the greatest common left inner
divisor and the least common left inner multiple, respectively. \

If $A: D\to E$ is a linear operator whose domain is a subspace of
$D$, then $A$ is also a linear operator from the closure of the
domain of $A$ into $E$.  \ So we will only consider those $A$ such
that the domain of $A$ is dense in $D$. \ Such an operator $A$ is
said to be {\it densely defined}. \ If $A: D \rightarrow E $ is
densely defined, we write $\hbox{dom}\, A$, $\hbox{ker}\,A$, and
$\hbox{ran}\,A$ for the domain, the kernel, and the range of $A$,
respectively. \ If $A : D \rightarrow E $ is densely defined, write
$$
\hbox{dom}\,A^* = \bigl\{e \in E :  \langle Ad, e \rangle \ \hbox{is
a bounded linear functional for all} \ d \in \hbox{dom}\,A \bigr\}.
$$
Then there exists a unique $f\in E$ such that $\langle Ad, e
\rangle=\langle d,f \rangle$ for all $d\in \hbox{dom}\,A$.  \ Denote
this unique vector $f$ by $f\equiv A^*e$. \ Thus $\langle Ad,
e\rangle=\langle d, A^*e\rangle $ for all $d\in \hbox{dom}\,A$ and
$e\in\hbox{dom}\,A^*$. \ We call $A^*$ the adjoint of $A$. \ It is
well known from unbounded operator theory (cf. \cite{Go},
\cite{Con})
 that if $A$ is densely
defined, then $\ker A^*= (\hbox{ran}\, A)^\perp$, so that $\ker A^*$
is closed even though $\ker A$ may not be closed. \

We
recall (\cite{Ab}, \cite{Co}, \cite{GHR}, \cite{Ni1}) that a
meromorphic function $\phi:\mathbb D \rightarrow \mathbb C$ is said
to be of {\it bounded type} (or in the Nevanlinna class $\mathcal
N$) if there are functions $\psi_1,\psi_2\in H^\infty$ such that
$$
\phi(z)=\frac{\psi_1 (z)}{\psi_2 (z)}\quad\hbox{for almost all}\ z\in\mathbb{T}.
$$
It is well known that $\phi$ is of bounded type if and only if
$\phi=\frac{\psi_1}{\psi_2}$ for some $\psi_i \in H^p$ ($p>0, \
i=1,2$). \ If $\psi_2=\psi^i \psi^e$ is the inner-outer
factorization of $\psi_2$, then
$\phi=\overline{\psi^{i}}\frac{\psi_1}{\psi^e}$. \ Thus if $\phi\in
L^2$ is of bounded type, then $\phi$ can be written as
$$
\phi=\overline{\theta}a,
$$
where $\theta$ is inner, $a \in H^2$ and $\theta$ and $a$ are
coprime.

Write $\mathbb{D}^e:=\{z: 1<|z|\leq \infty\}$. \ For a function $g:\mathbb{D}^e \to
\mathbb C$, define a function $g_{\mathbb D}: \mathbb D \to \mathbb
C$ by\label{pddc}
$$
g_{\mathbb D}(\zeta):=\overline{g(1/\overline{\zeta})}
\quad(\zeta \in \mathbb D).
$$
For a function $g:\mathbb{D}^e \to \mathbb C$, we say that $g$ belongs to
$H^p(\mathbb{D}^e)$ if $g_{\mathbb D} \in H^p$ $(1\leq p\leq \infty)$. \ A
function $g:\mathbb{D}^e \to \mathbb C$ is said to be of {\it bounded type}
if $g_{\mathbb D}$ is of bounded type. \ If $f\in H^2$, then the
function $\hat f$ defined in $\mathbb{D}^e$ is called a {\it
pseudo-continuation} of $f$ if $\hat f$ is a function of bounded type
and  $\hat f(z)=f(z)$ for almost all $z\in\mathbb T$ (cf. \cite{BB},
\cite{Ni1}, \cite{Sh}). \ Then we can easily show that $\breve{f}$
is of bounded type if and only if $f$ has a pseudo-continuation $\hat
f$. \ In this case, $\hat{f}_{\mathbb D}(z)=\overline{f(z)}$ for
almost all $z\in \mathbb T$. \ In particular,
\begin{equation}\label{btp}
\phi\equiv\breve{\phi}_- +\phi_+\in L^2\ \hbox{is of bounded type}\
\Longleftrightarrow \ \hbox{$\phi_-$ has a pseudo-continuation.}
\end{equation}

We review here a few essential facts concerning vector-valued $L^p$-
and $H^p$-functions that we will used to begin with, using
\cite{DS}, \cite{Du}, \cite{FF}, \cite{HP}, \cite{Ho}, \cite{Ni1},
\cite{Ni2}, \cite{Pe}, \cite{Sa} as general references. \

Let $(\Omega, \mathfrak M, \mu)$ be a positive $\sigma$-finite
measure space and $X$ be a complex Banach space. \
 A function $f: \Omega \to X$ of the form $
f=\sum_{k= 1}^{\infty}x_k \chi_{\sigma_k} $ (where $x_k \in X,  \
\sigma_k \in \mathfrak M$ and $\sigma_k \cap \sigma_j=\emptyset$ for
$k \neq j$) is said to be {\it countable-valued}. \ A function
$f:\Omega \to X$ is called {\it weakly measurable} if the map
$s\mapsto \phi(f(s))$ is measurable for all $\phi \in X^*$ and is
called {\it strongly} {\it measurable} if there exist
countable-valued functions $f_n$ such that $ f(s)=\lim_n f_n(s)$ for
almost all $s\in\Omega$. \ It is known that when $X$ is separable,
\begin{itemize}
\item[(i)] if $f$ is weakly measurable, then $||f(\cdot)||$ is
measurable;
\item[(ii)] $f$ is strongly measurable if and only if it is weakly
measurable.
\end{itemize}
A countable-valued function
$f=\sum_{k=1}^{\infty} x_k \chi_{\sigma_k}$ is called ({\it Bochner})
{\it integrable}
if
$$\int_{\Omega}||f(s)||d\mu(s)<\infty$$
and its
integral is defined by
$$
\int_{\Omega}f d\mu:=\sum_{k= 1}^{\infty}x_k \mu(\sigma_k).
$$
A function $g:\Omega \to X$ is called {\it integrable} if there
exist countable-valued integrable functions $g_n$ such that
$g(s)=\lim_{n}g_n(s)$ for almost all $s\in\Omega$ and $
\lim_{n}\int_{\Omega}||g-g_n||d \mu=0. $ \ Then $\int_{\Omega}g d
\mu \equiv \lim_{n}\int_{\Omega}g_nd\mu$ exists and $\int_{\Omega}g
d \mu$  is called the  ({\it Bochner}) {\it integral} of $g$.  \ If
$f:\Omega \rightarrow X$ is integrable, then we can see that
\begin{equation}\label{usefulwww}
T\biggl(\int_{\Omega}f d\mu \biggr)=\int_{\Omega}(Tf)d \mu \quad
\hbox{for each} \  T\in\mathcal B(X,Y).
\end{equation}
Let $m$ denote the normalized Lebesgue measure on $\mathbb T$. \ For
a complex Banach space $X$ and $1\le p\le \infty$, let
$$
L^p_X \equiv L^p(\mathbb T, X):=\bigl\{ f:\mathbb T \to X: f \
\hbox{is strongly measurable and} \ ||f||_p <\infty \bigr\},
$$
where
$$
||f||_p\equiv ||f||_{L^p_X}:=
\begin{cases}
\biggl( \int_{\mathbb T}||f(z)||_X^p
dm(z)\biggr)^{\frac{1}{p}}\ \
           (1\le p<\infty);\\
{\rm ess\ sup}_{z\in\mathbb T}\, ||f(z)||_X \ \ (p=\infty).
\end{cases}
$$
Then we can see that $L^p_X$ forms a Banach space. For $f \in
L^1_X$, the $n$-th Fourier coefficient of $f$, denoted by
$\widehat{f} (n)$, is defined by
$$
\widehat{f}(n):=\int_{\mathbb T} \overline z^n f(z)\, dm(z)\ \
\hbox{for each $n\in\mathbb Z$}.
$$
Also, $H^p_X\equiv H^p(\mathbb T, X)$ is defined by the set of $f\in
L^p_X$ with $\widehat{f}(n)=0$ for $n<0$. \ A function $f: \mathbb D
\to X$ is (norm) analytic if $f$ can be written as
$$
f(\zeta)=\sum_{n=0}^{\infty}x_n \zeta^n \quad(\zeta \in \mathbb D,
x_n \in X),
$$
Let $\hbox{Hol}(\mathbb D, X)$ denote the set of all analytic
functions $f: \mathbb D \to X$. \ Also we write $H^2(\mathbb D,X)$
for the set of all $f\in \hbox{Hol}(\mathbb D,X)$ satisfying
$$
||f||_{H^2(\mathbb D,X)}:=\sup_{0<r<1}\left(
\int_{\mathbb T} ||f(rz)||^2_X dm(z)\right)^{\frac{1}{2}}<\infty.
$$
Let $E$ be a separable complex Hilbert space. \ As in the
scalar-valued case, if $f\in H^2(\mathbb D,E)$, then there exists a
``boundary function" $bf\in H^2_E$ such that
$$
f(rz)=(bf \ast P_r)(z)\ \ (r\in [0,1)\ \hbox{and} \ z\in\mathbb T)
$$
(where $P_r$ denotes the Poisson kernel) and
$$
(bf)(z)=\lim_{rz\to z} f(rz)\ \ \hbox{nontangentially a.e. on $\mathbb T$}.
$$
Moreover, the mapping $f\mapsto bf$ is an isometric bijection (cf.
\cite[Theorem 3.11.7]{Ni2}). \ We conventionally identify
$H^2(\mathbb D,E)$ with $H^2_E\equiv H^2(\mathbb T,E)$. \ For $f,  g
\in L_E^2$ with a separable complex Hilbert space $E$, the inner
product $\langle f, \ g\rangle$ is defined by
$$
\bigl\langle f, \ g \bigr\rangle\equiv \bigl\langle f(z), \ g(z)
\bigr\rangle_{L^2_E} :=\int_{\mathbb T}\bigl\langle f(z), g(z)
\bigr\rangle_E dm(z).
$$
If $f,g\in L^2_X$ with
$X=M_{n\times m}$, then $\langle f, g\rangle=\int_{\mathbb T}
\hbox{tr}\, (g^*f)dm$.

\medskip

For a function $\Phi:\mathbb T \to \mathcal B(D,E)$,
write
$$
\Phi^*(z):=\Phi(z)^*\quad\hbox{for} \ z \in \mathbb T.
$$
A function $\Phi:\mathbb T \to \mathcal B(X, Y)$ is called {\it SOT
measurable} if $z \mapsto \Phi(z)x$ is strongly measurable for every
$x \in X$ and is called {\it WOT measurable} if $z\mapsto \Phi(z)x$
is weakly measurable for every $x \in X$.  \ We can easily check
that if $\Phi:\mathbb T \to \mathcal B(X, Y)$ is strongly
measurable, then $\Phi$ is SOT-measurable and if $D$ and $E$ are
separable complex Hilbert spaces then $\Phi: \mathbb T \to \mathcal
B(D, E)$ is SOT measurable if and only if $\Phi$ is WOT measurable.

\medskip

We then have:

\begin{lem}\label{remmeasurable}
If $\Phi:\mathbb T \to \mathcal B(D,E)$ is WOT measurable, then so
is $\Phi^*$.
\end{lem}
\begin{proof} Suppose that $\Phi$ is WOT measurable. \ Then the
function
$$
z \mapsto \overline{\bigl\langle \Phi^*(z) y, \
x\bigr\rangle}=\bigl\langle x, \  \Phi^*(z)y
\bigr\rangle=\bigl\langle \Phi(z)x, \  y \bigr\rangle
$$
is measurable for all $x \in D$ and $y\in E$.  \ Thus the function $
z \mapsto \bigl\langle \Phi^*(z)y, \ x\bigr\rangle $ is measurable
for all $x \in D$ and $y\in E$.
\end{proof}

\medskip

Let $\Phi:\mathbb T \to \mathcal B(D, E)$ be a WOT measurable
function. \ Then $\Phi$ is called {\it WOT integrable} if
$\bigl\langle \Phi(\cdot)x, y \bigr\rangle \in L^1$ for every $x \in
D$ and $y\in E$, and there exists an operator $U\in \mathcal B(D,E)$
such that $ \bigl\langle Ux, y\bigr\rangle=\int_{\mathbb T}
\bigl\langle \Phi(z)x, y \bigr\rangle dm(z)$. \ Also $\Phi$ is
called {\it SOT integrable} if $\Phi(\cdot)x$ is integrable for
every $x \in D$. \ In this case, the operator $ V: x\mapsto
\int_{\mathbb T}\Phi(z)x dm(z) $ is bounded, i.e., $V \in \mathcal
B(D, E)$. \  If $\Phi:\mathbb T \to \mathcal B(D, E)$ is SOT
integrable, then it follows from (\ref{usefulwww}) that for every $x
\in D$ and $y\in E$,
\begin{equation}\label{ffnskdll}
\biggl \langle \int_{\mathbb T}\Phi(z)xdm(z), \ y \biggr
\rangle=\int_{\mathbb T}\bigl\langle \Phi(z)x, \ y\bigr \rangle
dm(z),
\end{equation}
which implies that $\Phi$ is WOT integrable and that the SOT
integral of $\Phi$ is equal to the WOT integral of $\Phi$.

\medskip

We can say more:

\begin{lem}\label{remark3.3sdddf}
For $\Phi\in L^1_{\mathcal B(D,E)}$, the Bochner integral of $\Phi$
is  equal to the SOT integral of $\Phi$, in the sense that
$$
\left(\int_{\mathbb T}\Phi(z) dm(z)\right)x = \int_{\mathbb
T}\Phi(z)xdm(z) \quad \hbox{for all} \ x\in D.
$$
\end{lem}

\begin{proof}
This follows from a straightforward calculation.
\end{proof}

%
%
%
%
%
%

\chapter{Strong $L^2$-functions}

To examine Question \ref{mainq}, we need to consider
operator-valued functions defined on the unit circle constructed by
arranging the vectors in $F$ as their column vectors. \ Using this viewpoint, we will consider operator-valued functions  whose ``column"
vectors are $L^2$-functions. \ Note that (bounded linear) operators
between separable Hilbert spaces may be represented as infinite
matrices, so that column vectors of operators are well
justified. \ This viewpoint leads us to define (operator-valued) strong
$L^2$-functions. \ In this chapter we consider strong
$L^2$-functions and then derive some of their properties.

The terminology of a ``strong $H^2$-function" is reserved for the
operator-valued functions on the unit disk $\mathbb D$, following to
N.K. Nikolskii \cite{Ni1}: A function $\Phi:\mathbb D\to \mathcal
B(D,E)$ is called a {\it strong $H^2$-function} if $\Phi(\cdot)x\in
H^2(\mathbb D, E)$ for each $x\in D$. \ To describe this in detail, and to explain the crucial role that strong $L^2$-functions play in our theory, we need to introduce some additional notation and terminology.

\bigskip
Let $L^{\infty}(\mathcal B(D,E))$ be the space of all bounded (WOT)
measurable $\mathcal B(D,E)$-valued functions on $\mathbb T$. \  For
$\Psi \in L^{\infty}(\mathcal B(D,E))$, define
$$
||\Psi||_{\infty} :=\hbox{ess sup}_{z \in \mathbb T}||\Psi(z)||.
$$
For $1\le p< \infty$, we define the class $L^p_s(\mathcal
B(D,E))\equiv L^p_s(\mathbb T, \mathcal B(D,E))$
as the set of all (WOT) measurable $\mathcal
B(D,E)$-valued functions $\Phi$ on $\mathbb T$
such that $\Phi(\cdot)x\in L^p_E$.
A function $\Phi\in L^p_s(\mathcal B(D,E))$ is called a {\it strong
$L^p$-function}. \ We claim that
\begin{equation}\label{lps}
L^p_{\mathcal B(D, E)}\subseteq L^{p}_s(\mathcal B(D,E)):
\end{equation}
indeed if $\Phi \in L^p_{\mathcal B(D, E)}$, then for all $x \in D$
with $||x||=1$,
$$
||\Phi(z)x||_{L^p_E}^p=\int_{\mathbb T}||\Phi(z)x||_E^pdm(z) \leq
\int_{\mathbb T}||\Phi(z)||_{\mathcal B(D, E)}^pdm(z)=
||\Phi||_{L^p_{\mathcal B(D, E)}}^p,
$$
which gives (\ref{lps}). \ Also we can easily check that
\begin{equation}\label{linfty}
L^{\infty}_{\mathcal B(D, E)}\subseteq L^\infty (\mathcal B(D, E))
\subseteq L^{p}_s(\mathcal B(D,E)).
\end{equation}

\bigskip

\begin{rem}\label{ex3.3}
We may define a norm on $L^p_s(\mathcal B(D,E))$: i.e.,
$$
||\Phi||_{p}^{(s)}:=\hbox{sup}\Bigl\{||\Phi(z)x||_{L^p_E}: x\in D \
\hbox{with} \ ||x||=1\Bigr\}.
$$
Then $L^{p}_s(\mathcal B(D,E))$ forms a normed space for $1\le
p<\infty$. \ Moreover, we can show that $||\Phi||_{p}^{(s)}$ is a
complete norm for $1\le p<\infty$, i.e., $L^{p}_s(\mathcal B(D,E))$
is a Banach space for $1\le p<\infty$. \  However, in general, we
cannot guarantee that $||\Phi||_{p}^{(s)}=||\Phi||_{L^p_{\mathcal
B(D, E)}}$. \ To see this, let $C$ be the upper unit circle and
$1\le p< \infty$. \ Put
$$
\Phi:=\begin{bmatrix} \chi_{C}&0\\0&1-\chi_{C}\end{bmatrix}.
$$
Then $||\Phi(z)||=1$ for all $z\in\mathbb T$, so that
$||\Phi||_{L^p_{M_2}}=1$. \ Let $x:=[\alpha,\ \beta]^t$ be a unit
vector in $\mathbb C^2$.  \ Then we have that
$$
||\Phi(z)x||_{L^{p}_{\mathbb C^2}}^p=\int_{\mathbb
T}\bigl|\bigl|[\alpha\chi_C, \ \beta
(1-\chi_C)]^t\bigr|\bigr|^pdm(z)
=\frac{1}{2}(|\alpha|^p+|\beta|^p)\leq\frac{1}{\sqrt{2}},
$$
which gives $ ||\Phi||_p^{(s)}\neq ||\Phi||_{L^p_{M_2}}$.
\qed
\end{rem}

\bigskip

If $\Phi\in L^1_s (\mathcal B(D,E))$ and $x \in D$, then
$\Phi(\cdot)x\in L^1_E$. \ Thus the $n$-th Fourier coefficient
$\widehat{\Phi(\cdot)x}(n)$ of $\Phi(\cdot)x$ is given by
$$
\widehat{\Phi(\cdot)x}(n)=\int_{\mathbb
T}\overline{z}^n\Phi(z)x\,dm(z).
$$
We now define the $n$-th Fourier coefficient of $\Phi\in L^1_s (\mathcal
B(D,E))$, denoted by $\widehat \Phi (n)$, by
$$
\widehat\Phi(n)x:=\widehat{\Phi(\cdot)x}(n) \quad (n\in\mathbb Z, \
x\in D).
$$
We define
$$
H^2_s(\mathcal B(D,E))\equiv H^2_s(\mathbb T, \mathcal B(D,E))
:=\bigl\{\Phi\in L^2_s(\mathcal B(D,E))
: \ \widehat\Phi(n)=0\ \hbox{for} \ n<0\bigr\},
$$
or equivalently, $H^2_s(\mathcal B(D,E))$ is the set of all WOT
measurable functions $\Phi$ on $\mathbb T$ such that
$\Phi(\cdot)x\in H^2_E$ for each $x\in D$. \ We also define
$$
H^{\infty}(\mathcal B(D,E))\equiv H^\infty (\mathbb T, \mathcal B(D,E))
:=\bigl\{\Phi\in
L^{\infty}(\mathcal B(D,E)): \ \widehat\Phi(n)=0\ \hbox{for} \ n<0\bigr\}.
$$
On the other hand, we define $H^\infty(\mathbb D, \mathcal B(D,E))$
as the set of all analytic functions $\Phi:\mathbb D\to \mathcal B(D,E)$
satisfying
$$
||\Phi||_{H^\infty}:=\sup_{\zeta \in \mathbb D}||\Phi(\zeta)||.
$$
If $D$ and $E$ are separable Hilbert spaces,
we conventionally identify $H^\infty(\mathbb D, \mathcal B(D,E))$
with $H^\infty (\mathbb T, \mathcal B(D,E))$ (cf. \cite[Theorem 3.11.10]{Ni2}).

On the other hand, by (\ref{lps}), we have $L^1_{\mathcal
B(D,E)}\subseteq L^1_s (\mathcal B(D,E))$.  \ Thus if $\Phi\in
L^1_{\mathcal B(D,E)}$, then there are two definitions of
the $n$-th Fourier coefficient of $\Phi$. \ However, we can, by
Lemma \ref{remark3.3sdddf}, see that the $n$-th Fourier coefficient
of $\Phi$ as an element of $L^1_{\mathcal B(D,E)}$ coincides with
the $n$-th Fourier coefficient of $\Phi$ as an element of
$L^1_s(\mathcal B(D,E))$.

\bigskip

We now denote by $H^2_s(\mathbb D, \mathcal B(D,E))$ the set of all strong
$H^2$-functions with values in $\mathcal B(D,E)$.

We then have:

\begin{lem}\label{fffflll}  $H^2(\mathbb D, \mathcal B(D, E))\subseteq
H^2_s(\mathbb D, \mathcal B(D, E))$.
\end{lem}

\begin{proof} Let $\Phi \in H^2(\mathbb
D, \mathcal B(D, E))$. Then $\Phi$ can be written as
$$
\Phi(\zeta)=\sum_{n=0}^{\infty}A_n \zeta^n \quad(A_n\in  \mathcal
B(D, E)).
$$
Thus for each $x \in D$,
$$
\Phi(\zeta)x=\sum_{n=0}^{\infty}(A_nx) \zeta^n \in
\hbox{Hol}(\mathbb D, E).
$$
Observe that
$$
\aligned ||\Phi(\cdot)x||_{H^2(\mathbb D, E)}^2&=\hbox{sup}_{0<
r<1}\int_{\mathbb T}||\Phi(rz)x||_{E}^2 dm(z)\\
&\leq||\Phi||_{H^2(\mathbb D,
\mathcal B(D, E))}^2 \cdot ||x||_D^2\\
&<\infty,
\endaligned
$$
which implies $\Phi\in H^2_s(\mathbb D, \mathcal B(D, E))$.
\end{proof}

\medskip

\begin{thm}\label{fffbgdnfflll}
If $\dim D<\infty$, then
$$
H^2(\mathbb D, \mathcal B(D, E))= H^2_s(\mathbb D, \mathcal B(D,E)),
$$
where the equality is set-theoretic.
\end{thm}

\begin{proof}
By Lemma \ref{fffflll}, we have $H^2(\mathbb D, \mathcal B(D,
E))\subseteq H^2_s(\mathbb D, \mathcal B(D, E))$. \ For the reverse
inclusion, suppose $\Phi \in H^2_s(\mathbb D, \mathcal B(D, E))$ and
$\dim D=d<\infty$. Let $\{e_j: j=1,2,\cdots, d\}$ be an orthonormal
basis of $D$. \ Then for each $j=1,2,\cdots, d$,
\begin{equation}\label{5439064jgjg}
\phi_j (\zeta)\equiv\Phi(\zeta)e_j \in H^2(\mathbb D, E).
\end{equation}
Thus we may write
$$
\phi_j(\zeta)=\sum_{n=0}^{\infty}a_n^{(j)}\zeta^n \quad(a_n^{(j)}
\in E).
$$
For each $n=0,1,2,\cdots$, define $ A_n: D \to E$ by
$$
A_nx:=\sum_{j=1}^d \alpha_j a_n^{(j)} \quad\Bigl(\hbox{where}\
x:=\sum_{j=1}^d \alpha_je_j \Bigr).
$$
Then $A_n \in \mathcal B(D, E)$. We claim that
\begin{equation}\label{7111}
\Phi(\zeta)=\sum_{n=0}^{\infty}A_n\zeta^n \in \hbox{Hol}(\mathbb D,
\mathcal B(D, E)).
\end{equation}
To prove (\ref{7111}), let $\epsilon>0$ be arbitrary. \ For each
$\zeta \in \mathbb D$, there exists $M>0$ such that for all
$j=1,2,\cdots, d$,
$$
\Biggl|\Biggl|\sum_{n=M}^{\infty}a_n^{(j)}\zeta^n
\Biggr|\Biggr|_{E}<\frac{\epsilon}{d}.
$$
Let $x:=\sum_{j=1}^d \alpha_je_j$ with $||x||_{D}=1$. Then we have
$$
\aligned
\Biggl|\Biggl|\Bigl(\Phi(\zeta)-\sum_{n=0}^{M-1}A_n\zeta^n\Bigr)x
\Biggr|\Biggr|_E&=\Biggl|\Biggl|\sum_{n=M}^{\infty}\sum_{j=1}^d
\alpha_ja_n^{(j)}\zeta^n
\Biggr|\Biggr|_E\\
&\leq \sum_{j=1}^d\Biggl|\Biggl| \sum_{n=M}^{\infty}a_n^{(j)}\zeta^n
\Biggr|\Biggr|_E\\
&<\epsilon,
\endaligned
$$
which proves (\ref{7111}). \ For all $r\in [0,1)$, we have that
$$
\aligned ||\Phi(rz)x||_E^2&=\Biggl|\Biggl|\sum_{j=1}^d \alpha_j
\Phi(rz)e_j\Biggr|\Biggr|^2_E\\
&\leq \Biggl(\sum_{j=1}^d|\alpha_j||| \Phi(rz)e_j||_E\Biggr)^2\\
&\leq \sum_{j=1}^d|| \Phi(rz)e_j||_E^2.
\endaligned
$$
Thus $||\Phi(rz)||_{\mathcal B(D, E)}^2\leq \sum_{j=1}^d||
\Phi(rz)e_j||_E^2$, and hence it follows from (\ref{5439064jgjg})
that
$$
\aligned ||\Phi||_{H^2(\mathbb D, \mathcal B(D,
E))}&=\sup_{0<r<1}\int_{\mathbb T}||\Phi(rz)||_{\mathcal
B(D, E)}^2 dm(z)\\
&\leq \sup_{0< r<1}\int_{\mathbb
T}\sum_{j=1}^d||\Phi(rz)e_j||_E^2 dm(z)\\
&\leq\sum_{j=1}^d ||\phi_j||_{H^2(\mathbb D, E)}^2 <\infty,
\endaligned
$$
which implies $\Phi \in H^2(\mathbb D, \mathcal B(D, E))$. This completes the
proof.
\end{proof}

\medskip

\begin{rem}\label{rem99988}
Theorem \ref{fffbgdnfflll} may fail if the condition ``$\dim
D<\infty$ is dropped. \ For example, if $\Phi$ is defined on the
unit disk $\mathbb D$ by
$$
\Phi(\zeta):=\begin{bmatrix}\zeta&\zeta^2&\zeta^3&\cdots\end{bmatrix}:\ell^2
\to \mathbb C \qquad(\zeta \in \mathbb D),
$$
then $\Phi(\zeta)$ is a bounded linear operator for each $\zeta\in\mathbb D$:
indeed,
$$
\aligned ||\Phi(\zeta)||_{\mathcal B(\ell^2, \mathbb
C)}&=\sup_{||x||=1}\bigl|\Phi(\zeta)x\bigr|\\
&=\sup_{||x||=1}\Biggl|\sum_{n=1}^{\infty}\zeta^n x_n \Biggr|
\quad(x\equiv (x_n)\in \ell^2)\\
&=\sup_{||x||=1} \Bigl|\Bigl \langle (\zeta, \ \zeta^2, \ \zeta^3, \
\cdots), \ (\overline{x}_1,  \ \overline{x}_2, \ \overline{x}_3, \
\cdots)
\Bigr \rangle \Bigr|\\
&=\bigl|\bigl|(\zeta, \ \zeta^2, \ \zeta^3, \ \cdots)\bigr|\bigr|_{\ell^2}\\
&=\Bigl(\frac{|\zeta|^2}{1-|\zeta|^2}\Bigr)^{\frac{1}{2}}.
\endaligned
$$
Moreover, for each $x\equiv (x_n)\in \ell^2$,
$$
\Phi(\zeta)x =\sum_{n=1}^{\infty}x_{n}\zeta^n \in H^2(\mathbb D,
\mathbb C),
$$
which says that $\Phi \in H^2_s(\mathbb D, \mathcal B(\ell^2,
\mathbb C))$. \ However, we have $\Phi \notin H^2(\mathbb D,
\mathcal B(\ell^2, \mathbb C))$: indeed, for $\zeta=rz\in \mathbb
D$,
$$
||\Phi(\zeta)||_{\mathcal B(\ell^2, \mathbb
C)}^2=||\Phi(\zeta)\Phi(\zeta)^*||_{\mathcal B(\ell^2, \mathbb
C)}=\frac{r^2}{1-r^2},
$$
so that
$$
\aligned \sup_{0< r<1}\int_{\mathbb T}||\Phi(rz)||_{\mathcal
B(\ell^2, \mathbb C)}^2 dm(z)&= \sup_{0< r<1}\int_{\mathbb
T}\frac{r^2}{1-r^2} dm(z)\\
&= \sup_{0< r<1}\frac{r^2}{1-r^2}\\
&=\infty.
\endaligned
$$
\qed
\end{rem}

\bigskip

In general, the boundary values of strong $H^2$-functions do not
need to be bounded linear operators (defined almost everywhere on
$\mathbb T$). \ Thus we do not guarantee that the boundary value of
a strong $H^2$-function belongs to $H^2_s(\mathbb T, \mathcal
B(D,E))$. \ For example, if $\Phi$ is defined on the unit disk
$\mathbb D$ by\label{sh2ex}
$$
\Phi(\zeta)=\begin{bmatrix} 1&\zeta&\zeta^2&\zeta^3&\cdots
\end{bmatrix}:\ell^2\to\mathbb C \ \ (\zeta\in\mathbb D),
$$
then by Remark \ref{rem99988}, $\Phi$ is a strong $H^2$-function
with values in $\mathcal B(\ell^2,\mathbb C)$. \ However, the
boundary value
$$
\Phi(z)=\begin{bmatrix}
1&z&z^2&z^3&\cdots
\end{bmatrix}: \ell^2\to\mathbb C\ \ (z\in\mathbb T)
$$
is not bounded for all $z\in\mathbb T$
because
for any $z_0 \in \mathbb T$, if we let
$$
x_0:=\Bigl(1, \overline{z}_0, \frac{\overline{z}_0^2}{2},
\frac{\overline{z}_0^3}{3}, \cdots\Bigr)^t \in \ell^2,
$$
then
$$
\Phi(z_0)x_0=1+\sum_{n=1}^{\infty}\frac{1}{n}=\infty,
$$
which shows that $\Phi\notin H^2_s(\mathbb T, \mathcal B(D,E))$.

\medskip

In spite of it, there are useful relations between the set $H^2_s
(\mathbb D, \mathcal B(D,E))$ and the set $H^2_s(\mathbb T, \mathcal
B(D,E))$. \ To see this, let $\Phi\in H^2_s(\mathbb T, \mathcal
B(D,E))$. \ Then $\Phi(z)\in \mathcal B(D,E)$ for almost all
$z\in\mathbb T$ and $\Phi(z)x\in H^2_E$ for each $x\in D$. \ We now
define a (function-valued with domain $D$) function $p\Phi$ on the
unit disk $\mathbb D$ by the Poisson integral in the strong sense:
$$
\begin{aligned}
p\Phi(re^{i\theta})x&:=\left(\Phi(\cdot)x \ast
P_r\right)(e^{i\theta})\ \ (x\in D)\\
&=\int_0^{2\pi} P_r(\theta-t)\Phi(e^{it})x\,dm(t)\in E,
\end{aligned}
$$
where $P_r(\cdot)$ is the Poisson kernel. \ Then $p\Phi(\zeta)x\in
H^2 (\mathbb D, E)$. \ Thus, for all $\zeta\in\mathbb D$,
$p\Phi(\zeta)$ can be viewed as a function from $D$ into $E$. \ A
straightforward calculation shows that $p\Phi(\zeta)$ is a linear
map for each $\zeta\in\mathbb D$. \ Since $p\Phi(\zeta)x\in
H^2(\mathbb D,E)$ is the Poisson integral of $\Phi(z)x\in H^2_E$, we
will conventionally identify $\Phi(z)x$ and $p\Phi(\zeta)x$ for each
$x\in D$. \ From this viewpoint, we will also regard $\Phi\in
H^2_s(\mathbb T, \mathcal B(D,E))$ as an (linear, but not
necessarily bounded) operator-valued function defined on the unit
disk $\mathbb D$.\label{strongh2p}

\medskip

We thus have:

\begin{lem}\label{strongh2} The following inclusion holds:
$$
H^2_{\mathcal B(D,E)}\cup H^\infty(\mathcal B(D,E))
\subseteq H^2_s(\mathbb D, \mathcal B(D,E)).
$$
\end{lem}

\begin{proof}
Note that by (\ref{lps}) and (\ref{linfty}), $H^2_{\mathcal
B(D,E)}\cup H^\infty(\mathcal B(D,E)) \subseteq H^2_s(\mathbb T,
\mathcal B(D,E))$. \ Thus in view of the preceding remark, it
suffices to show $\Phi(\zeta)\in\mathcal B(D,E)$ for all $\zeta\in
\mathbb D$. \ To see this we first claim that there exists $M>0$
such that
\begin{equation}\label{0001}
\hbox{sup}\Bigl\{||\Phi(\cdot)x||_{L^1_E}: x \in D \ \hbox{with} \
||x||=1 \Bigr\}<M,
\end{equation}
To see this, if $\Phi \in H^2_{\mathcal B(D,E)}$, then for all $x \in D$
with $||x||=1$,
$$
\aligned
||\Phi(\cdot)x||_{L^1_E}
&\le ||\Phi(\cdot)x||_{L^2_E}\\
&\leq\Biggl(\int_{\mathbb T}||\Phi(z)||_{\mathcal B(D, E)}^2dm(z)\Biggr)^{\frac{1}{2}}\\
&=||\Phi||_{L^2_{\mathcal B(D, E)}}.
\endaligned
$$
If instead $\Phi \in H^\infty(\mathcal B(D,E))$, then for all $x \in D$ with
$||x||=1$,
$$
||\Phi(\cdot)x||_{L^1_E}=\int_{\mathbb T}||\Phi(z)x||_Edm(z)\leq
||\Phi(z)||_{\infty},
$$
which proves the claim (\ref{0001}). \ Now, let
$\zeta=re^{i\theta}\in \mathbb D$ and $x \in D$ with $||x||=1$. \
Then for $y \in E$ with $||y||\leq 1$,
$$
\aligned \Bigl |\bigl \langle \Phi(re^{i \theta})x, \ y
\bigr\rangle_E \Bigr |&=\Biggl |\Bigl \langle \int_{0}^{2 \pi}
P_r(\theta-t)\Phi(e^{i t})x dm(t),\ y \Bigr\rangle_E\Biggr |\\
&=\Biggl | \int_{0}^{2 \pi} \Bigl \langle
P_r(\theta-t)\Phi(e^{i t})x , \ y\Bigr\rangle_E dm(t)\Biggr |\quad\hbox{(by (\ref{ffnskdll}))}\\
&\leq \frac{1+r}{1-r}\int_{0}^{2 \pi} \bigl |\bigl \langle \Phi(e^{i
t})x , \ y\bigr\rangle_E \bigr | dm(t),
\endaligned
$$
which implies, by our assumption,
$$
\aligned ||\Phi(\zeta)x||_E
&\leq \frac{1+r}{1-r}\int_{0}^{2 \pi} \bigl |\bigl | \Phi(e^{i t})x
\bigr|\bigr |_E dm(t)\\
&=\frac{1+r}{1-r}||\Phi(\cdot)x||_{L^1_E}\\
&<\infty,
\endaligned
$$
which shows that $\Phi(\zeta)\in\mathcal B(D,E)$ for all $\zeta\in
\mathbb D$. \ Thus we have $\Phi\in H^2_s(\mathbb D, \mathcal
B(D,E))$.
\end{proof}

\bigskip

We now recall a notion from classical Banach space theory, about regarding a vector as an operator acting on the scalars. \ This notion is important as motivation for the study of strong $L^2$-functions. \  
Let $E$ be a separable complex Hilbert space. \  For a function
$f:\mathbb T \to E$, define $[f]:\mathbb T \to \mathcal B(\mathbb C,
E)$ by
\begin{equation}\label{bracket}
[f](z)\alpha:=\alpha f(z)\quad(\alpha\in\mathbb C).
\end{equation}
If $g:\mathbb T \to E$ is a countable-valued function of
the form
$$
g=\sum_{k= 1}^{\infty}x_k \chi_{\sigma_k} \quad(x_k \in E),
$$
then for each $\alpha \in \mathbb C$,
$$
\Biggl(\sum_{k= 1}^{\infty}[x_k]
\chi_{\sigma_k}\Biggr)\alpha=\sum_{k= 1}^{\infty}\alpha
x_k\chi_{\sigma_k}=\alpha g=[g]\alpha,
$$
which implies that $[g]$ is a countable-valued function of the form
$
[g]=\sum_{k= 1}^{\infty}[x_k] \chi_{\sigma_k}.
$

\medskip

We then have:


\begin{lem}\label{dhfbgbgbgbg}
Let $E$ be a separable complex Hilbert space and $1\leq p\leq
\infty$. \ Define $\Gamma: L^{p}_E \to L^{p}_{\mathcal B(\mathbb C,
E)}$ by
$$
\Gamma(f)(z)=[f](z),
$$
where $[f](z):\mathbb C \to E$ is given by $[ f](z)\alpha:=\alpha
f(z)$. \ Then

\medskip

\begin{itemize}
\item[(a)] $\Gamma$ is unitary, and hence $
L^{p}_E \cong L^{p}_{\mathcal B(\mathbb C, E)}$;
\item[(b)] $L^{p}_{\mathcal B(\mathbb C, E)}= L^{p}_s(\mathcal B(\mathbb C,
E))$  for $1\le p<\infty$;
\item[(c)]  $\widehat{[f]}(n)=[\widehat{f}(n)]$ for
$f \in L^p_E$ and $n \in \mathbb Z$.
\end{itemize}
\medskip
In particular,  $H^{p}_E \cong H^{p}_{\mathcal
B(\mathbb C, E)} =H^{p}_s(\mathcal B(\mathbb C, E))$ for $1\le p<\infty$.
\end{lem}

\begin{proof}
(a)   Let $f \in L^p_E$ $(1\leq p\leq \infty)$ be arbitrary.  \ We
first show that  $[f]\in L^p_{\mathcal B(\mathbb C, E)}$. \ Since
$f$ is strongly measurable, there exist countable-valued functions
$f_n$ such that $f(z)=\lim_{n}f_n(z)$ for almost all $z \in \mathbb
T$. \ Observe that for almost all $z \in \mathbb T$,
$$
||[f](z)||_{\mathcal B(\mathbb C,
E)}=\hbox{sup}_{|\alpha|=1}||[f](z)\alpha||_E=||f(z)||_E.
$$
Thus we have that
$$
\bigl|\bigl|[f_n](z)-[f](z)\bigr|\bigr|_{\mathcal B(\mathbb C, E)}
=\bigl|\bigl|f_n(z)-f(z)\bigr|\bigr|_E\rightarrow 0 \quad \hbox{as }
\ n\rightarrow \infty,
$$
which implies that $[f]$ is strongly measurable and $
||[f]||_{L^p_{\mathcal B(\mathbb C, E)}}=||f||_{L^p_E}$. \ Thus
$\Gamma$ is an isometry. \ For $h \in L^{p}_{\mathcal B(\mathbb C,
E)}$, let $g(z):=h(z)1\in L^p_E$. \ Then for all $\alpha \in \mathbb
C$, we have
$$
\Gamma(g)(z)\alpha=\alpha h(z)1=h(z)\alpha,
$$
which implies that $\Gamma$ is a surjection from $L^{p}_E$ onto
$L^{p}_{\mathcal B(\mathbb C, E)}$. \ Thus $\Gamma$ is unitary, so
that $L^p_E\cong L^{p}_{\mathcal B(\mathbb C, E)}$. \ This proves
(a). \

\smallskip
(b)  Suppose $h \in L^{p}_s(\mathcal B(\mathbb C, E))$ $(1\leq p<
\infty)$. \ If $g(z):=h(z)1\in L^p_E$, then  $h=[g]\in
L^{p}_{\mathcal B(\mathbb C, E)}$. \ The converse is clear.

\smallskip

(c) Let $f \in L^p_E$.  \ Then for all $\alpha \in \mathbb C$ and
$n\in\mathbb Z$,
$$
\widehat{[f]}(n)\alpha=\int_{\mathbb T} \overline{z}^n[f](z)\alpha
dm=\alpha\int_{\mathbb T} \overline{z}^n f(z) dm=\alpha
\widehat{f}(n)=[\widehat{f}(n)]\alpha,
$$
which gives (c).

The last assertion follows at once from (b) and (c).
\end{proof}

\medskip
For $\mathcal X$ a closed subspace of $D$, $P_{\mathcal X}$ denotes
the orthogonal projection from $D$ onto $\mathcal{X}$. \ Then we
have:

\begin{lem}\label{thmved}
If $\dim D<\infty$, then
\begin{itemize}
\item[(a)] $L^2_s(\mathbb T, \mathcal B(D, E))= L^{2}_{\mathcal B(D, E)}$;
\item[(b)] $H^2_s(\mathbb T, \mathcal B(D, E))= H^{2}_{\mathcal B(D,E)}$,
\end{itemize}
where the equalities are set-theoretic.
\end{lem}

\begin{proof}
(a) Let $d:=\dim D<\infty$. \ It follows from (\ref{lps}) that
$L^{2}_{\mathcal B(D, E)}\subseteq L^2_s(\mathcal B(D, E))$. \ For
the reverse inclusion, let $\{e_j\}_{j=1}^d$ be an orthonormal basis
of $D$. \ Suppose $\Phi \in L^2_s(\mathcal B(D, E))$. \ Then
$$
\phi_j(z)\equiv \Phi(z)e_j\in L^2_E  \qquad (j=1,2,\cdots, d).
$$
It thus follows from Lemma \ref{dhfbgbgbgbg} that $[\phi_j]\in
L^2_{\mathcal B(\mathbb C, E)}$. \ For $j=1,2,\cdots, d$,  define
$\Phi_j:\mathbb T \to \mathcal B(D, E)$ by
$$
\Phi_j:=[\phi_j]P_{D_j} \quad \Bigl( \mathbb C \cong D_j:=\bigvee
e_j\Bigr).
$$
Since $[\phi_j]$ is strongly measurable, it is easy to show that
$\Phi_j$ is strongly measurable for each $j=1,2,\cdots$. \ It
follows from Lemma \ref{dhfbgbgbgbg} that
$$
\aligned ||\Phi_j||_{L^{2}_{\mathcal B(D, E)}}^2&=\int_{\mathbb
T}\bigl|\bigl|\Phi_j(z)\bigr|\bigr|_{\mathcal
B(D, E)}^2dm(z)\\
&=\int_{\mathbb T}\bigl|\bigl|[\phi_j](z)\bigr|\bigr|_{\mathcal
B(\mathbb C, E)}^2dm(z)\\
&=\bigl|\bigl|[\phi_j]\bigr|\bigr|_{L^2_{\mathcal B(\mathbb C,
E)}}^2\\
&=||\phi_j||_{L^2_E}^2\\
&<\infty.
\endaligned
$$
Thus $\Phi_j \in L^{2}_{\mathcal B(D, E)}$, and hence $
\Phi=\sum_{j=1}^d \Phi_j\in L^{2}_{\mathcal B(D, E)}$. \ This proves
(a).

\medskip

(b) This follows from Lemma \ref{remark3.3sdddf} and (a).
\end{proof}

\bigskip

To proceed, we define a ``boundary function" $b\Phi$ for each
function $\Phi \in H^2_s(\mathbb D, \mathcal B(D, E))$ with
$\dim\,D<\infty$. \
In this case, we may assume that $D=\mathbb C^d$.

Let $\Phi \in H^2_s(\mathbb D, \mathcal B(D, E))$ and
$\{e_j\}_{j=1}^d$ be the canonical basis for $\mathbb C^d$. \ Then
$\phi_j(\zeta)\equiv\Phi(\zeta)e_j \in H^2(\mathbb D, E)$. \ Thus we
have
\begin{equation}\label{nmnmmnkgkekf}
\phi_j(z)\equiv(b\phi_j)(z):=\lim_{rz \to z}\phi_j(rz)\in H^2_E.
\end{equation}
It follows from Lemma \ref{dhfbgbgbgbg} that for each
$j=1,2,3,\cdots,d$,
$$
[\phi_j]\in H^2_{\mathcal B(\mathbb C, E)}=
H^2_s(\mathbb T, \mathcal B(\mathbb C, E)),
$$
where $[\phi_j](z)\alpha:=\alpha \phi_j(z)$ for all $\alpha \in
\mathbb C$. \ Note that there exists a subset $\sigma \subset\mathbb
T$ with $m(\sigma)=0$ such that
\begin{equation}\label{jmuynhgbvfc}
\phi_j(z) \in E \quad \hbox{for each} \ z \in \mathbb
T_0\equiv\mathbb T\setminus\sigma.
\end{equation}
Define a function $b$ on $H^2_s(\mathbb D, \mathcal B(D, E))$ by
\begin{equation}\label{defkkgkhk}
(b \Phi)(z):=\bigl[[\phi_1](z), [\phi_2](z),\cdots,
[\phi_d](z)\bigr] \quad(z \in \mathbb T_0).
\end{equation}
Then we have that for all $x\in D$,
\begin{equation}\label{cbdjefberbgh}
(b \Phi)(z)x=\lim_{rz\to z}\Phi(rz)x \in E  \ \  (z \in \mathbb
T_0).
\end{equation}
A straightforward calculation shows that
$(b \Phi)(z)$ is a linear mapping from $D$ into $E$ for almost all
$z \in \mathbb T$.

\bigskip

We thus have:

\begin{thm}\label{xsxsndsnafsd}
If $\dim D <\infty$, then the function $b$ defined by
(\ref{defkkgkhk}) is a linear bijection from $H^2_s(\mathbb D,
\mathcal B(D, E))$ onto $H^2_s(\mathbb T, \mathcal B(D, E))$.
\end{thm}

\begin{proof}
Let $d:=\dim D<\infty$. Then we may assume that $D=\mathbb C^d$. \
Let $\{e_j\}_{j=1}^d$ be the canonical basis for $\mathbb C^d$ and
$\mathbb T_0$ be defined as the above.

\medskip

(1) $b$ is well-defined: Let $\Phi \in H^2_s(\mathbb D, \mathcal
B(\mathbb C^d, E))$. \ Then it follows from (\ref{jmuynhgbvfc}) that
for each $z_0 \in \mathbb T_0$,
$$
||(b\Phi)(z_0)||_{\mathcal B(\mathbb C^d,
E)}\leq\sum_{n=1}^d||\phi_{j}(z_0)||_E<\infty
$$
which implies that $(b\Phi)(z_0)$ is bounded for each $z_0 \in
\mathbb T_0$. \ If $x\equiv(x_1,x_2,\cdots, x_d)^t \in \mathbb C^d$,
then
$$(b\Phi)(z)x=\sum_{n=1}^d x_{j}\phi_{j}(z) \in
H^2_E,
$$
which implies that $b\Phi \in H^2_s(\mathcal B(\mathbb C^d, E))$,
and hence $b$ is well-defined.

\medskip

(2) $b$ is linear: Immediate from a direct calculation.

\medskip

(3) $b$ is one-one: Let $\Phi, \Psi \in H^2_s(\mathbb D, \mathcal
B(\mathbb C^d, E))$. \ If $b \Phi=b \Psi$, then it follows that for
each $x \in \mathbb C^d$ and $rz \in \mathbb D$,
$$
\aligned
\Phi(rz)x&=((b\Phi)x \ast P_r)(z)\\
&=\int_{0}^{2 \pi} P_r(\theta-t)(b\Phi)(e^{i t})x dm(t)\\
&=\int_{0}^{2 \pi} P_r(\theta-t)(b\Psi)(e^{i t})x dm(t)\\
&=\Psi(rz)x \quad(z=e^{i \theta}),
\endaligned
$$
which gives the result.

\medskip

(4) $b$ is onto: Let $A \in H^2_s(\mathbb T, \mathcal B(\mathbb C^d,
E))$. \ Then $A(z)e_j\in H^2_E$ for all $j=1,2,\cdots, d$. \ For
each $j=1,2,\cdots, d$, let
$$
\phi_j(rz):=(Ae_j \ast P_r)(z) \in H^2(\mathbb D, E)
$$
and define
$$
\Phi(\zeta):=[\phi_1(\zeta), \phi_2(\zeta), \cdots, \phi_d(\zeta)]
\quad(\zeta:=rz).
$$
Then $\Phi\in H^2_s(\mathbb D, \mathcal B(\mathbb C^d, E))$. \ It
follows from (\ref{cbdjefberbgh}) that for all $x=(x_1, x_2, \cdots,
x_d)^t \in \mathbb C^d$ and for almost all $z \in \mathbb T$,
$$
\aligned (b \Phi)(z)x&=\lim_{rz\to z}\Phi(rz)x\\
&=\lim_{rz\to z}\sum_{j=1}^d x_j\phi_j(rz)\\
&=\sum_{j=1}^d x_jA(z)e_j\\
&=A(z)x,
\endaligned
$$
which implies that $b$ is onto.  \ This completes the proof.
\end{proof}

\bigskip

We thus have:

\begin{cor}\label{xsxsndsnafsdd}
If $\dim D <\infty$, then the
function $b$ defined by (\ref{defkkgkhk}) is an isometric bijection
from $H^2(\mathbb D, \mathcal B(D, E))$ onto $H^2_{\mathcal B(D,
E)}$.
\end{cor}

\begin{proof}
By Theorem \ref{xsxsndsnafsd} together with Theorem
\ref{fffbgdnfflll} and Lemma \ref{thmved}, the function $b$ defined
by (\ref{defkkgkhk}) is a linear bijection from $H^2(\mathbb D,
\mathcal B(D, E))$ onto $H^2_{\mathcal B(D,E)}$. \ In view of the
Banach space-valued version of the usual Hardy space theory (cf.
\cite[Theorem 3.11.6]{Ni2}), it suffices to show that
\begin{equation}\label{kckdvkdfvk}
\Phi(re^{it})=(b \Phi \ast P_r)(e^{it}).
\end{equation}
Indeed, if $z \in \mathbb T$, $r \in (0,1)$, and $x \in D$, then
$$
\aligned (b \Phi \ast P_r)(e^{it})x&=\Biggl(\int_{0}^{2
\pi}P_r(\theta-t)(b \Phi)(e^{it})dm(t)\Biggr)x\\
&=\int_{0}^{2
\pi}P_r(\theta-t)(b \Phi)(e^{it})xdm(t) \quad (\hbox{by Lemma} \ \ref{remark3.3sdddf}) \\
&=\Phi(re^{it})x,
\endaligned
$$
which gives (\ref{kckdvkdfvk}).
\end{proof}

\bigskip

According to the convention of the usual Hardy space theory, we will
identify $b\Phi$ with $\Phi\in H^2(\mathbb D, \mathcal B(D, E))$. \
In this sense, we eventually have:

\medskip

\begin{cor}
If $\dim\,D<\infty$, then
$$
H^2_s(\mathbb D, \mathcal B(D,E))=H^2 (\mathbb D, \mathcal B(D,E))
=H^2_{\mathcal B(D,E)}=H^2_s(\mathbb T, \mathcal B(D,E)),
$$
where the first and last equalities are set-theoretic, while the second equality
establishes an isometric isomorphism.
\end{cor}

\begin{proof}
This follows from Theorem \ref{fffbgdnfflll}, Lemma \ref{thmved},
and Corollary \ref{xsxsndsnafsdd}.
\end{proof}

\bigskip

A function $\Delta \in H^{\infty}(\mathcal B(D, E))$ is called an
{\it inner} function with values in $\mathcal B(D,E)$ if $\Delta(z)$
is an isometric operator from $D$ into $E$ for almost all
$z\in\mathbb T$, i.e., $\Delta^*\Delta=I_D$ a.e. on $\mathbb T$. \ $\Delta$ is called a {\it two-sided inner} function if $\Delta\Delta^*=I_E$
a.e. on $\mathbb T$ and $\Delta^*\Delta=I_D$ a.e. on $\mathbb T$. \
If $\Delta$ is an inner function with values in $\mathcal B(D, E)$,
we may assume that $D$ is a subspace of $E$, and if further $\Delta$
is two-sided inner then we may assume that $D=E$.

We write $\mathcal{P}_{D}$ for the set of all polynomials
with values in $D$, i.e.,
$p(z)=\sum_{k=0}^n \widehat p(k) z^k$, where $\widehat p(k)\in D$.
If $F$
 is a strong $H^2$-function with values in $\mathcal B(D,E)$,
then the function $Fp$ belongs to $H^2_{E}$ for all $p\in \mathcal
P_D$. \ The strong $H^2$-function $F$ is called {\it outer} if
$\hbox{cl}\, F \mathcal{P}_D=H^2_{E}$.  We then have an analogue of
the scalar factorization theorem:

\medskip

\noindent{\bf Inner-Outer Factorization for strong $H^2$-functions} \
(cf. \cite[Corollary I.9]{Ni1}). \
Every strong $H^2$-function $F$
with values in $\mathcal{B}(D, E)$ can be expressed in the form
$$
F=F^iF^e,
$$
where $F^e$ is an outer function with values in $\mathcal{B}(D,
E^\prime)$ and $F^i$ is an inner function with values in
$\mathcal{B}(E^\prime, E)$ for some subspace $E^\prime$ of $E$.

\bigskip

For a function $\Phi: \mathbb T \to \mathcal B(D,E)$, write
$$
\breve{\Phi}(z):= \Phi(\overline{z}), \quad \widetilde\Phi
:= \breve{\Phi}^*.
$$
We call $\breve\Phi$ the {\it flip} of $\Phi$. \ For
$\Phi\in L^2_s(\mathcal B(D, E))$, we denote by
$\breve{\Phi}_-\equiv\mathbb P_-\Phi$ and $\Phi_+\equiv\mathbb
P_+\Phi$ the functions
$$
\aligned &((\mathbb P_- \Phi)(\cdot))x:=P_-(\Phi(\cdot)x) \quad
\hbox{a.e.
on} \ \mathbb T \quad(x \in D);\\
&((\mathbb P_+ \Phi)(\cdot)) x:=P_+(\Phi(\cdot)x) \quad \hbox{a.e.
on} \ \mathbb T \quad(x \in D),
\endaligned
$$
where $P_+$ and $P_-$ are the orthogonal projections from $L^2_E$
onto $H^2_E$ and $L^2_E \ominus H^2_E$, respectively. \ Then we may
write $\Phi\equiv \breve{\Phi}_-+\Phi_+ $. \ Note that if $\Phi\in
L^2_s(\mathcal B(D, E))$, then $ \Phi_+,  \ \Phi_- \in
H^2_s(\mathcal B(D, E))$.

In the sequel, we will often encounter the adjoints of inner matrix
functions. \ If $\Delta$ is a two-sided inner matrix function, it is
easy to show that $\Delta^*$ is of bounded type, i.e., all entries
of $\Delta^*$ are of bounded type (see p.~\pageref{matrixbt}). \ We
may predict that if $\Delta$ is an inner matrix function then
$\Delta^*$ is of bounded type.  \ However the following example
shows that this is not the case.

\medskip

\begin{ex}\label{ex3.6}
Let $h(z):= e^{\frac{1}{z-3}}$.  Then $h \in H^{\infty}$ and
$\overline{h}$ is not of bounded type. \ Let
$$
f(z):=\frac{h(z)}{\sqrt{2}||h||_{\infty}}.
$$
Clearly, $\overline{f}$ is not of bounded type. \ Let
$h_1(z):=\sqrt{1-|f(z)|^2}$. \ Then $h_1\in L^{\infty}$ and
$|h_1|\geq \frac{1}{\sqrt{2}}$. \ Thus
there exists an outer function $g$ such that $|h_1|=|g|$ a.e. on
$\mathbb T$ (see \cite[Corollary 6.25]{Do1}). \ Put
$$
\Delta:=\begin{bmatrix}f\\g \end{bmatrix} \quad (f,g\in H^\infty).
$$
Then $\Delta^*\Delta=|f|^2+|g|^2=|f|^2+|h_1|^2=1$ a.e. on $\mathbb
T$, which implies that $\Delta$ is an inner function. \ Note that
$\Delta^*$ is not necessarily of bounded type.
\end{ex}

\medskip

For a function $\Phi\in H^2_s(\mathcal B(D,E))$, we say that an
inner function $\Delta$ with values in $\mathcal B(D^{\prime},E)$ is
a {\it left inner divisor} of $\Phi$ if $\Phi=\Delta A$ for $A\in
H^2_s(\mathcal B(D,D^\prime))$. \ For $\Phi\in H^2_s(\mathcal
B(D_1,E))$ and $\Psi\in H^2_s (\mathcal B(D_2,E))$, we say that
$\Phi$ and $\Psi$ are {\it left coprime} if the only common left
inner divisor of both $\Phi$ and $\Psi$ is a unitary operator. \
Also,
we say that $\Phi$ and $\Psi$ are {\it right coprime} if
$\widetilde\Phi$ and $\widetilde\Psi$ are left coprime. \ Left or
right coprime-ness seems to be somewhat delicate problem. \ Left or
right coprime-ness for matrix-valued functions was developed in
\cite{CHKL}, \cite{CHL1}, \cite{CHL2}, \cite{CHL3}, and \cite{FF}.

\medskip

\begin{lem} \label{rem.sdcfcfc}
If $\Theta$ is a two-sided inner function, then any left inner divisor of $\Theta$ is
two-sided inner.
\end{lem}

\begin{proof}
Suppose that $\Theta$ is a two-sided inner function with values in
$\mathcal B(E)$ and $\Delta$  is a left inner divisor, with values
in $\mathcal B(E^{\prime}, E)$, of $\Theta$. \ Then we may write
$\Theta=\Delta A$ for some $A\in H^2_s(\mathcal B(E, E^{\prime}))$.
\ Since $\Theta$ is two-sided inner, it follows that $I_E=\Theta
\Theta^*=\Delta AA^*\Delta^*$ a.e. on $\mathbb T$, so that
$I_{E^{\prime}}= \Delta^*\Delta=AA^*$ a.e. on $\mathbb T$. \ Thus
$I_E=\Delta \Delta^*$ a.e. on $\mathbb T$, and hence $\Delta$ is
two-sided inner.
\end{proof}

\medskip

\begin{lem}\label{corfgghh2.9}
If $\Phi\in L^{\infty}(\mathcal B(D,E))$, then $\Phi^*\in
L^{\infty}(\mathcal B(E,D))$. \ In this case,
\begin{equation}\label{2999}
\widehat{\Phi^*}(-n)=\widehat{\widetilde{\Phi}}(n)=\widehat{\Phi}(n)^*
\quad(n \in \mathbb Z).
\end{equation}
In particular, $\Phi\in H^{\infty}(\mathcal B(D,E))$ if and only if
$\widetilde{\Phi}\in H^{\infty}(\mathcal B(E,D))$.
\end{lem}

\begin{proof}  Suppose $\Phi\in L^{\infty}(\mathcal B(D,E))$.  \ Then
$$
\hbox{ess sup}_{z \in \mathbb T}||\Phi^*(z)||=\hbox{ess sup}_{z \in
\mathbb T}||\Phi(z)||<\infty,
$$
which together with Lemma \ref{remmeasurable} implies $\Phi^*\in
L^{\infty}(\mathcal B(E,D))$. \ The first equality of the assertion
(\ref{2999}) comes from the definition. \ For the second equality,
observe that for each $x\in D$, $y \in E$ and $n \in \mathbb Z$,
$$
\aligned \bigl\langle \widehat{\Phi}(n) x, \ y \bigr
\rangle&=\biggl\langle \int_{\mathbb T}\overline{z}^n
\Phi(z)x dm(z), \ y \biggr \rangle\\
&= \int_{\mathbb T}\bigl\langle\overline{z}^n
\Phi(z)x , \ y \bigr \rangle dm(z)\quad\hbox{(by (\ref{ffnskdll}))}\\
&= \int_{\mathbb T}\bigl\langle x , \ \overline{z}^n
\widetilde\Phi (z) y \bigr \rangle dm(z)\\
&=\bigl\langle  x, \ \widehat{\widetilde{\Phi}}(n)y \bigr \rangle.
\endaligned
$$
\end{proof}

\bigskip

\begin{lem}\label{lem3.4ed}
Let $1\leq p<\infty$.  If $\Phi \in L^\infty(\mathcal B(D,E))$, then
$\Phi L^p_s(\mathcal B(E^{\prime}, D)) \subseteq L^p_s(\mathcal
B(E^{\prime},E))$. \  Also, if $\Phi \in H^\infty(\mathcal B(D,E))$,
then $ \Phi H^2_s(\mathcal B(E^{\prime}, D)) \subseteq
H^2_s(\mathcal B(E^{\prime},E)). $
\end{lem}

\begin{proof}   Suppose that
$\Phi \in L^\infty(\mathcal B(D,E))$ and $A \in L^p_s(\mathcal
B(E^{\prime}, D))$.  Let $x\in E^{\prime}$ be arbitrary. \ Then we
have $A(z)x \in L^p_D$. \ Let $\{d_k\}_{k\geq 1}$ be an orthonormal
basis for $D$. Thus we may write
\begin{equation}\label{kkkngd}
A(z)x=\sum_{k\geq 1}\langle A(z)x, \ d_k \rangle d_k \quad \hbox{
for almost all} \ z\in \mathbb T.
\end{equation}
Thus it follows that for all $y\in E$,
$$
\bigl\langle \Phi(z)A(z)x, \ y \bigr \rangle=\sum_{k\geq
1}\bigl\langle A(z)x, \ d_k \bigr\rangle \bigl \langle \Phi(z)d_k,y
\bigr \rangle,
$$
which implies that $\Phi A$ is WOT measurable. \ On the other hand,
since $\Phi \in L^\infty(\mathcal B(D,E))$, it follows that
$$
\int_{\mathbb T}||(\Phi A)(z)x||^p_E dm(z)\le
||\Phi||_{\infty}^p\int_{\mathbb T}||A(z)x||_D^p dm(z)<\infty
\quad(x \in E^{\prime}),
$$
which implies that $\Phi A \in L^p_s(\mathcal B(E^{\prime},E))$. \
This proves the first assertion. \ For the second assertion, suppose
$\Phi \in H^\infty(\mathcal B(D,E))$ and $A \in H^2_s(\mathcal
B(E^{\prime}, D))$. \ Then $\Phi A\in L^2_s (\mathcal
B(E^{\prime},E))$. \  Assume to the contrary that $\Phi A \notin
H^2_s(\mathcal B(E^{\prime}, E))$. \ Thus, there exists $n_0>0$ such
that $\widehat{\Phi A}(-n_0) \neq 0$. \ Thus for some $x_0 \in
E^{\prime}$,
\begin{equation}\label{bbbxvdvvd}
\int_{\mathbb T}z^{n_0}\Phi(z)A(z)x_0dm(z) \neq 0.
\end{equation}
Then by (\ref{ffnskdll}), there exists a nonzero $y_0
\in E$ such that
\begin{equation}
0 \neq \biggl \langle \int_{\mathbb T}z^{n_0}\Phi(z)A(z)x_0dm(z), \
y_0 \biggr \rangle=\int_{\mathbb T}\bigl \langle  A(z)x_0, \
\overline{z}^{n_0}\Phi^*(z)y_0\bigr \rangle dm(z).
\end{equation}
On the other hand, since $\Phi \in H^\infty(\mathcal B(D,E))$, it
follows from Lemma \ref{corfgghh2.9} that $
\widehat{\Phi^*}(n_0)=\widehat{\Phi}(-n_0)^*=0$. \ Thus it follows
from (\ref{ffnskdll}) that
$$
0=\bigl\langle \widehat{\Phi^*}(n_0)y_0, \ A(z)x_0
\bigr\rangle=\int_{\mathbb T}\bigl\langle
\overline{z}^{n_0}\Phi^*(z)y_0, \ A(z)x_0 \bigr\rangle dm(z),
$$
a contradiction.
\end{proof}

\medskip

\begin{cor}\label{lemma4.100000}
Let $1\leq p<\infty$. \ If $\Phi \in L^\infty(\mathcal B(D,E))$,
then $\Phi L^p_D \subseteq  L^p_E$.  \ Also, if $\Phi \in
H^\infty(\mathcal B(D,E))$, then $\Phi H^2_D \subseteq  H^2_E$.
\end{cor}

\begin{proof}  Suppose that $\Phi \in L^\infty(\mathcal B(D,E))$.
For $f \in L^P_D$, we can see that $[\Phi f]=\Phi[f]$. \ The result
thus follows from Lemma \ref{dhfbgbgbgbg} and Lemma \ref{lem3.4ed}.
\end{proof}

\bigskip

For an inner function $\Delta\in H^\infty(\mathcal B(E^\prime,E))$,
$\mathcal H(\Delta)$ denotes the orthogonal complement of the
subspace $\Delta H^2_{E^\prime}$ in $H^2_E$, i.e.,
$$
\mathcal H(\Delta):=H^2_{E} \ominus \Delta H^2_{E^\prime}.
$$
The space $\mathcal H(\Delta)$ is often called a {\it model space}
or a {\it de Branges-Rovnyak space} (cf. \cite{dR}, \cite{Sa}, \cite{SFBK}).

\bigskip

We then have:

\begin{cor}\label{thm2.4edfrgt}
Let $\Delta$ be an inner function with values in $\mathcal B(D,E)$.
Then $f\in \mathcal H(\Delta)$ if and only if  $f\in H^2_E$ and
$\Delta^* f \in L^2_D \ominus H^2_{D}$.
\end{cor}

\begin{proof}
Let $f\in H^2_E$. By Lemma \ref{corfgghh2.9} and Corollary
\ref{lemma4.100000}, $\Delta^*f\in L^2_D$. \  Then $f \in \mathcal
H(\Delta)$ if and only if $\bigl\langle f, \Delta g \bigr\rangle=0$
for all $g \in H^2_D$ if and only if $\bigl\langle \Delta^* f, \ g
\bigr \rangle=0$ for all $g\in H^2_D$, which gives the result.
\end{proof}

%
%
%
%
%
%

\chapter{The Beurling-Lax-Halmos Theorem} \

In this chapter we introduce the Beurling-Lax-Halmos Theorem and the
Douglas-Shapiro-Shields factorization. \ Then we coin the new
notions of complementary factor of an inner function, degree of
non-cyclicity, strong $L^2$-functions of bounded type, and
meromorphic pseudo-continuation of bounded type for operator-valued
functions.

\vskip 1 cm

%
%
%
%

\noindent {\bf \S\ 4.1. The Beurling-Lax-Halmos Theorem} \

\bigskip
\noindent We first review a few essential facts for (vectorial)
Toeplitz operators and (vectorial) Hankel operators, and for that we
will use \cite{BS}, \cite{Do1}, \cite{Do2}, \cite{MR}, \cite{Ni1},
\cite{Ni2}, and \cite{Pe} for general references. \ For $\Phi\in
L^2_s(\mathcal{B}(D,E)) $, the Hankel operator $H_\Phi: H^2_D\to
H^2_E$ is a densely defined operator defined by
$$
 H_\Phi p:=J P_-(\Phi p)  \quad(p \in \mathcal{P}_{D}),
$$
where $J$ denotes the unitary operator from $L^2_{E}$ to $L^2_{E}$
given by $(Jg)(z) :=\overline{z} g(\overline{z})$ for $g \in
L^2_{E}$. \ Also a Toeplitz operator $T_{\Phi}: H^2_D\to H^2_E$ is a
densely defined operator defined by
$$
T_{\Phi}p:= P_+(\Phi p) \quad(p \in \mathcal{P}_{D}).
$$

\medskip

The following lemma gives a characterization of bounded Hankel
operators on $H^2_D$.

\medskip

\begin{lem}\label{boundedhankel}  \cite[Theorem 2.2]{Pe} \
Let $\Phi\in L^2_s(B(D,E))$. \
 Then $H_{\Phi}$ is extended to a bounded
operator on $H^2_D$ if and only if there exists a function $\Psi\in
L^{\infty}(\mathcal B(D, E))$ such that
$\widehat{\Psi}(n)=\widehat{\Phi}(n)$ for $n<0$ and
$$
||H_{\Phi}||=\hbox{dist}_{L^{\infty}}(\Psi, H^{\infty}(\mathcal B(D,
E)).
$$
\end{lem}

\medskip

The following basic properties can be easily derived: \ If $D$, $E$,
and $D^\prime$ are separable complex Hilbert spaces and $\Phi \in
L^{\infty} (\mathcal B(D,E))$, then
\begin{align}
&T_\Phi^*=T_{\Phi^*}, \ H_{\Phi}^*=H_{\widetilde\Phi};\label{form_1}\\
&H_\Phi T_\Psi = H_{\Phi\Psi}\quad \hbox{if} \
\Psi \in H^{\infty} (\mathcal B(D^\prime, D));\label{formula}\\
&H_{\Psi\Phi}=T_{\widetilde{\Psi}}^*H_\Phi \quad \hbox{if} \ \Psi
\in H^{\infty} (\mathcal B(E, D^\prime)).\label{formula22}
\end{align}

\bigskip
\noindent A {\it shift} operator $S_E$ on $H^2_E$ is defined by
$$
(S_E f)(z):=zf(z)\quad \hbox{for each} \ f\in H^2_E.
$$
Thus we may write $S_E=T_{zI_E}$.

\medskip

The following theorem is a fundamental result in modern
operator theory.

\medskip

\noindent {\bf The Beurling-Lax-Halmos Theorem.}\label{beur}
\cite{Be},  \cite{La}, \cite{Ha}, \cite{FF}, \cite{Pe} A subspace
$M$ of $H^2_{E}$ is invariant for the shift operator $S_E$ on
$H^2_{E}$ if and only if
$$
M=\Delta H^2_{E^{\prime}},
$$
where $E^{\prime}$ is a subspace of $E$ and $\Delta$ is an inner
function with values in $\mathcal B(E^{\prime}, E)$. \ Furthermore,
$\Delta$ is unique up to a unitary constant right factor, i.e., if
$M=\Theta H^2_{E^{\prime \prime}}$, where $\Theta$ is an inner
function with values  in $\mathcal B(E^{\prime\prime}, E)$, then
$\Delta=\Theta V$, where $V$ is a unitary operator from $E^{\prime}$
onto $E^{\prime\prime}$.

\bigskip

As customarily done, we say that two inner functions $A, B \in
H^{\infty} (\mathcal B(D,E))$ are {\it equal} if they are equal up
to a unitary constant right factor. \ If $\Phi\in
L^{\infty}(\mathcal B(D,E))$, then by (\ref{formula}) and
(\ref{formula22}),
$$
H_{\Phi^*} S_E = S_E^*H_{\Phi^*},
$$
which implies that the kernel of the Hankel operator $H_{\Phi^*}$ is
an invariant subspace of the shift operator $S_E$ on $H^2_{E}$. \
Thus, by the Beurling-Lax-Halmos Theorem,
$$
\hbox{ker}\, H_{\Phi^*}=\Delta H^2_{E^{\prime}}
$$
for some inner function $\Delta$ with values in $\mathcal
B(E^{\prime}, E)$.  We note that $E^{\prime}$ may be the zero space
and $\Delta$ need not be two-sided inner.

\bigskip

We however have:

\begin{lem}\label{rem2.4}
If $\Phi\in L^\infty (\mathcal B(D,E))$ and $\Delta$ is a two-sided
inner function with values in $\mathcal B(E)$, then the following
are equivalent:

\medskip

\begin{itemize}

\item[(a)] $\ker H_{\Phi^*}=\Delta H^2_E$;
\item[(b)] $\Phi=\Delta A^*$,
where $A\in H^\infty (\mathcal B(E,D))$ is such that $\Delta$ and
$A$ are right coprime.
\end{itemize}
\end{lem}

\begin{proof} Let $\Phi\in L^\infty (\mathcal B(D,E))$ and $\Delta$
be a two-sided inner function with values in $\mathcal B(E)$.

\smallskip

(a) $\Rightarrow$ (b): Suppose $\ker H_{\Phi^*}=\Delta H^2_{E}$. If
we put $A:=\Phi^*\Delta\in H^\infty (\mathcal B(E,D))$, then
$\Phi=\Delta A^*$. \ We now claim that $\Delta$ and $A$ are right
coprime. \ To see this, suppose $\Omega$ is a common left inner
divisor, with values in $\mathcal B(E^{\prime}, E)$, of
$\widetilde{\Delta}$ and $\widetilde{A}$. \ Then we may write
$\widetilde{\Delta}=\Omega\widetilde{\Delta}_1$ and
$\widetilde{A}=\Omega\widetilde{A}_1$, where $\widetilde\Delta_1\in
H^\infty(\mathcal B(E,E^{\prime}))$ and $\widetilde A_1\in
H^\infty(\mathcal B(D, E^{\prime}))$. \ Since ${\Delta}$ is
two-sided inner, it follows from Lemma \ref{rem.sdcfcfc} and Lemma
\ref{corfgghh2.9} that $\Omega$ and $\Delta_1$ are two-sided inner.
\  Since $\Phi=\Delta_1 A_1^*$, we have
$$
\Delta_1 H^2_{E^{\prime}}\subseteq \ker H_{\Phi^*}=\Delta H^2_E
=\Delta_1\widetilde{\Omega} H^2_E,
$$
which implies $H^2_{E^{\prime}}=\widetilde{\Omega} H^2_E$. \ Thus by
the Beurling-Lax-Halmos Theorem, $\widetilde{\Omega}$ is a unitary
constant and so is $\Omega$. \ Therefore, $\Delta$ and $A$ are right
coprime.

\smallskip

(b) $\Rightarrow$ (a): Suppose (b) holds. Clearly, $\Delta H^2_E
\subseteq \ker H_{\Phi^*}$.  \ By the Beurling-Lax-Halmos Theorem,
$\ker H_{\Phi^*}=\Theta H^2_{E^\prime}$ for some inner function
$\Theta$, so that $\Delta H^2_{E} \subseteq \Theta H^2_{E^\prime}$.
\ Thus
$\Theta$ is a left inner divisor of $\Delta$ (cf. \cite{FF},
\cite{Pe}) so that, by Lemma \ref{rem.sdcfcfc}, we may write
$\Delta=\Theta \Delta_0$ for some two-sided inner function
$\Delta_0$ with values in $\mathcal B(E, E^{\prime})$. \ Put $G:=
\Phi^*\Theta \in H^{\infty}(\mathcal B(E^\prime,D))$. \ Then
$G=A\Delta_0^*$, and hence,
$\widetilde{A}=\widetilde{\Delta}_0\widetilde{G}$. \ But since
$\Delta$ and $A$ are right coprime, $\widetilde{\Delta}_0$ is a
unitary operator, and so is $\Delta_0$.  \ Therefore $\ker
H_{\Phi^*}=\Delta H^2_{E}$, which proves (a).
\end{proof}

\medskip

We recall that the factorization in Lemma \ref{rem2.4}(b) is called
the ({\it canonical}) {\it Douglas-Shapiro-Shields factorization}\label{dssf} of
$\Phi\in L^\infty (\mathcal B(D,E))$\label{DSS} (see \cite{DSS},
\cite{FB}, \cite{Fu}). \ Consequently, Lemma \ref{rem2.4} may be
rephrased as: If $\Phi\in L^\infty (\mathcal B(D,E))$, then the
following are equivalent:
\begin{itemize}
\item[(a)] $\Phi$ admits a Douglas-Shapiro-Shields factorization;
\item[(b)] $\ker
H_{\Phi^*}=\Delta H^2_{E}$ for some two-sided inner function
$\Delta\in H^\infty (\mathcal B(E))$.
\end{itemize}

\bigskip

The following lemma will be frequently used in the sequel.

\medskip

\noindent {\bf Complementing Lemma.} \cite[p.~49, p.~53]{Ni1}
\label{Complementing} Let $\Psi\in H^{\infty} (\mathcal B(E^\prime,
E))$ with $E^\prime\subseteq E$ and $\dim\, E^\prime<\infty$, and
let $\theta$ be a scalar inner function. \ Then the following statements are
equivalent:
\begin{itemize}
\item[(a)] There exists a function $G$ in
$H^{\infty} (\mathcal B(E,E^\prime))$ such that $G\Psi=\theta
I_{E^\prime}$;
\item[(b)] There exist functions $\Phi$ and $\Omega$ in
$H^{\infty} (\mathcal B(E))$ with $\Phi|_{E^\prime}=\Psi$,
$\Phi|_{(E\ominus E^\prime)}$ being an inner function such that
$\Omega \Phi=\Phi \Omega={\theta}I_{E}$.
\end{itemize}
In addition, if $\dim\,E<\infty$, then (a) and (b) are equivalent to
the following statement:
\begin{itemize}
\item[(c)]  $\hbox{ess inf}_{z \in \mathbb T} \min
\bigl\{||\Psi(z)x||:||x||=1\bigr\} > 0$.
\end{itemize}

\bigskip

We recall that if $\Phi$ is a strong $H^2$-function with values in
$\mathcal B(D,E)$, with $\dim E<\infty$, the {\it local rank} of
$\Phi$ is defined by (cf. \cite{Ni1})
$$
\hbox{Rank}\,\Phi:=\hbox{max}_{\zeta \in \mathbb
D}\,\hbox{rank}\,\Phi(\zeta),
$$
where $\hbox{rank}\,\Phi(\zeta):=\dim \Phi(\zeta)(D)$.

\bigskip

As we have remarked in the Introduction, if $\Phi$ is a strong
$L^2$-function with values in $\mathcal B(D,E)$, then $H^*_{\breve
\Phi}$ need not be a Hankel operator. \ Of course, if $\Phi\in
L^\infty(\mathcal B(D,E))$, then by (\ref{form_1}),
$H_{\breve\Phi}^*=H_{\widetilde{\breve\Phi}}=H_{\Phi^*}$. \ By
contrast, for a strong $L^2$-function $\Phi$ with  values in
$\mathcal B(D,E)$, $H_{\breve\Phi}^*\ne H_{\Phi^*}$ in general even
though $\Phi^*$ is also a strong $L^2$-function. \ We note that if
$\Phi^*$ is a strong $L^2$-function with values in $\mathcal B(E,
D)$, then $\ker H_{\Phi^*}$ is possibly trivial because
$H_{\Phi^*}$ is defined in the dense subset of polynomials in
$H^2_{E}$. \ Thus it is much better to deal with $H_{\breve\Phi}^*$
in place of $H_{\Phi^*}$. \ Even though $H_{\breve\Phi}^*$ need not
be a Hankel operator, we can show that the kernel of  $H^*_{\breve
\Phi}$ is still of the form $\Delta H^2_{D^\prime}$ for some inner
function $\Delta$. \ To see this, we observe:

\medskip

\begin{lem}\label{thm4.2}
Let $\Phi$ be a strong $L^2$-function with values in $\mathcal B(D,
E)$. \ Then,
$$
\begin{aligned}
\hbox{ker}\, H_{\breve{\Phi}}^*=\Bigl\{f \in H^2_{E}: \int_{\mathbb
T}\bigl \langle \Phi(z)x, \
 z^n f(z)\bigr\rangle_{E} dm(z)=0  \quad
 &\hbox{for all} \  x \in D\\ & \hbox{and} \
 n=1,2,3, \cdots\Bigr\}.
\end{aligned}
$$
\end{lem}

\begin{proof} Observe that
$$
\aligned f \in \hbox{ker}\,H_{\breve{\Phi}}^* &\Longleftrightarrow
\bigl\langle H_{\breve{\Phi}}p, \ f \bigr\rangle_{L^2_E}=0 \quad
\hbox{for all} \ p \in \mathcal P_D\\
&\Longleftrightarrow \bigl\langle \breve\Phi(z) p(z), \ (Jf)(z)
\bigr\rangle_{L^2_E}=0 \quad
\hbox{for all} \ p \in \mathcal P_D\\
&\Longleftrightarrow \int_{\mathbb T}\bigl\langle
\Phi(\overline{z})x z^k, \ \overline{z}f(\overline{z})
\bigr\rangle_E dm(z)=0\quad \hbox{for all} \  x \in D\ \hbox{and} \
 k=0,1,2,\cdots\\
&\Longleftrightarrow \int_{\mathbb T} \bigl \langle \Phi(z)x, \ z^n
f(z)\bigr\rangle_Edm(z)=0 \quad \hbox{for all} \  x \in D\
\hbox{and} \
 n=1,2,3,\cdots,
\endaligned
$$
which gives the result.
\end{proof}

\medskip

We then have:

\begin{lem}\label{kerhadjoint}
If $\Phi$ is a strong $L^2$-function with values in $\mathcal
B(D,E)$, then
\begin{equation}\label{31400}
\hbox{ker}\, H_{\breve\Phi}^*=\Delta H^2_{E^{\prime}},
\end{equation}
where $E^\prime$ is a subspace of $E$ and $\Delta$ is an inner
function with values in $\mathcal B(E^{\prime}, E)$.
\end{lem}

\begin{proof}
By Lemma \ref{thm4.2}, if $f\in \ker H_{\breve\Phi}^*$, then $zf\in
\ker H_{\breve\Phi}^*$. \ Since $\hbox{ker}\,H_{\breve\Phi}^*$ is
always closed, it follows that  $\hbox{ker}\,H_{\breve\Phi}^*$ is an
invariant subspace for $S_E$. \ Thus, by the Beurling-Lax-Halmos
Theorem, there exists an inner function $\Delta$ with values in
$\mathcal B(E^{\prime}, E)$ such that $\hbox{ker}\,
H_{\breve\Phi}^*=\Delta H^2_{E^{\prime}}$ for a subspace $E^\prime$
of $E$.
\end{proof}

\vskip 1cm

%
%
%
%
%
%

\noindent {\bf \S\ 4.2. Complementary factors of inner functions} \

\bigskip
\noindent Let $\{\Theta_i\in H^\infty(\mathcal B(E_i, E)): i\in J\}$
be a family of inner functions.  Then the greatest common left inner
divisor $\Theta_d\label{thetad}$ and the least common left inner
multiple $\Theta_m\label{thetam}$ of the family $\{\Theta_i: i\in
J\}$ are the inner functions defined by
$$
\Theta_d H^2_{D}:=\bigvee_{i \in J}\Theta_{i}H^2_{E_i} \quad
\hbox{and} \quad \Theta_m H^2_{D^{\prime}}:=\bigcap_{i\in
J}\Theta_{i}H^2_{E_i}.
$$
By the Beurling-Lax-Halmos Theorem, $\Theta_d$ and $\Theta_m$ exist,
and are unique up to a unitary constant right factor.   We write
$$
\Theta_d \equiv \hbox{left-g.c.d.}\,\{\Theta_i: i\in J\}\quad
\hbox{and} \quad \Theta_m \equiv \hbox{left-l.c.m.}\,\{\Theta_i:
i\in J\}.
$$
If $\Theta_i$ is a scalar inner function, we write
$$
\hbox{g.c.d.}\, \{\Theta_i: i\in J\}
\equiv\hbox{left-g.c.d.}\,\{\Theta_i: i\in J\}
$$
and
$$
\hbox{l.c.m.}\,\{\Theta_i: i\in J\}\equiv
\hbox{left-l.c.m.}\,\{\Theta_i: i\in J\}.
$$

\bigskip

For $\Phi \in L^{\infty}(\mathcal B(D, E))$, we symbolically define
the kernel of $\Phi$ by
$$
\hbox{ker}\, \Phi:= \bigl\{f \in H^2_{D}: \Phi(z)f(z)=0 \ \hbox{for
almost all} \ z \in \mathbb T \bigr\}.
$$
Note that the kernel of $\Phi$ consists of functions in $H^2_D$, but
not in $L^2_D$, such that $\Phi f=0$ a.e. on $\mathbb T$. \ Since $\ker \Phi$ is an invariant subspace for $S_D$, it follows from the Beurling-Lax-Halmos Theorem that $\ker \Phi=\Omega H_{D^{\prime}}^2$, for some inner function $\Omega \in H^{\infty}(D^{\prime},D)$.

Let $\Delta$ be an inner function with values in $\mathcal B(D,E)$.
\  If $g \in \hbox{ker}\, \Delta^*$, then $g \in H^2_E$, so that by
Lemma \ref{strongh2} and Lemma \ref{dhfbgbgbgbg}, $[g]$ is a strong
$H^2$-function with values in $\mathcal B(\mathbb C, E)$ (see
p.\pageref{bracket} for the definition of $[g]$). \ Write
$$
[g]=[g]^i[g]^e \quad(\hbox{inner-outer factorization)},
$$
where $[g]^e$ is an outer function with values in
$\mathcal{B}(\mathbb C, E^\prime)$ and $[g]^i$ is an inner function
with values in $\mathcal{B}(E^\prime, E)$ for some subspace
$E^\prime$ of $E$. \ If $g \neq 0$, then $[g]^e$ is a nonzero outer
function, so that $E^{\prime}=\mathbb C$. \ Thus, $[g]^i\in
H^{\infty}(\mathcal B(\mathbb C, E))$. \ If instead $g= 0$, then
$E^{\prime}=\{0\}$. \ Therefore, in this case, $[g]^i\in
H^{\infty}(\mathcal B(\{0\}, E))$.

\medskip

We then have:

\begin{lem}\label{thm2.9}
Let $\Delta$ be an inner function with values in $\mathcal B(D,E)$. \ Then we may write 
\begin{equation} \label{kerdelta*}
\ker \Delta^*=\Omega H_{D^{\prime}}^2 
\end{equation}
for some inner function $\Omega$ with values in $B(D^{\prime},E)$. \ Put
\begin{equation}\label{2.9.9.9}
\Delta_c:=\hbox{left-g.c.d.} \bigl\{\,[g]^i: g\in \hbox{ker}\,
\Delta^* \bigr\}.
\end{equation}
Then we have
\medskip
\begin{itemize}
\item[(a)]  $\Omega = \Delta_c$;
\item[(b)]  $[\Delta, \Delta_c]$ is an inner function with values in
$\mathcal B(D\oplus D^\prime, E)$;
\item[(c)]  $\hbox{ker}\, H_{\Delta^*}=[\Delta, \Delta_c] H^2_{D \oplus
D^{\prime}}\equiv \Delta H^2_D \bigoplus \Delta_c H^2_{D^{\prime}}$,
\end{itemize}
where $[\Delta, \Delta_c]$ is obtained by complementing $\Delta_c$
to $\Delta$, in other words, $[\Delta, \Delta_c]$ is regarded as a
$1\times 2$ operator matrix.
\end{lem}

\medskip

\begin{df}
The inner function $\Delta_c$ in (\ref{2.9.9.9}) is said to be the {\it complementary factor} of
the inner function $\Delta$.
\end{df}

\medskip

\begin{proof}[Proof of Lemma \ref{thm2.9}] \ If $\hbox{ker}\, \Delta^*=\{0\}$, then (a) and (b) are trivial. \
Suppose that $\hbox{ker}\, \Delta^*\neq \{0\}$. \ Recall that
\begin{equation}\label{kernelunit}
\Delta_c=\hbox{left-g.c.d.} \bigl\{\, [g]^i :  g \in \hbox{ker}\,
\Delta^*\bigr\}\in H^{\infty}(\mathcal B(D^{\prime\prime},E)),
\end{equation}
where $D^{\prime\prime}$ is a nonzero subspace of $E$. If $g \in
\hbox{ker}\, \Delta^*$, then it follows from (\ref{kerdelta*}) that
$$
\aligned \Delta_c H^2_{D^{\prime\prime}}&=\bigvee \Bigl\{\ [g]^i
H^2:\, g \in \hbox{ker}\, \Delta^*\Bigr\}\\
&=\bigvee \Bigl\{\ [g]\mathcal P_{\mathbb C}:\  g \in \hbox{ker}\,
\Delta^*\Bigr\}\\
&\subseteq \, \hbox{ker}\, \Delta^*=\Omega H^2_{D^{\prime}}.
\endaligned
$$
For the reverse inclusion, let $0 \neq g \in \hbox{ker}\, \Delta^*$.
\  Then it follows that
$$
g(z)=[g](z)1=([g]^i[g]^e)(z)1=[g]^i(z)\bigl([g]^e(z)1\bigr) \in
[g]^i H^2.
$$
Thus we have
$$
\Omega H^2_{D^{\prime}} =\hbox{ker}\, \Delta^* \subseteq \bigvee
\Bigl\{ \ [g]^i H^2:\, g \in \hbox{ker}\, \Delta^*\Bigr\}= \Delta_c
H^2_{D^{\prime\prime}}.
$$
Therefore, by the Beurling-Lax-Halmos Theorem, $\Omega=\Delta_c$ and
$D^\prime=D^{\prime\prime}$, which gives (a). \ Note that
$\Delta^*\Delta_c=0$. \ We thus have
$$
\begin{bmatrix} \Delta^*\\ \Delta_c^*\end{bmatrix}
[\Delta,   \Delta_c]=\begin{bmatrix} I_D&0\\
0&I_{D^\prime}\end{bmatrix},
$$
which implies that $[\Delta, \Delta_c]$ is an inner function with
values in $\mathcal B(D\oplus D^\prime, E)$, which gives (b). \ For
(c), we first note that $\Delta H^2_{D}$ and $\hbox{ker}\, \Delta^*$
are orthogonal and
$$
\Delta H^2_{D} \bigoplus \hbox{ker}\, \Delta^*\subseteq \hbox{ker}\,
H_{\Delta^*}.
$$
For the reverse inclusion, suppose that $f \in H^2_E$ and $f \notin
\Delta H^2_{D} \bigoplus \hbox{ker}\, \Delta^* \equiv M$. \ Write
$$
f_1:= P_{M}f \ \ \hbox{and} \ \ f_2:= f-f_1 \neq 0.
$$
Since $f_2 \in H^2_E \ominus  M=\mathcal H(\Delta) \cap (H^2_E
\ominus \hbox{ker}\, \Delta^*)$, it follows from Corollary
\ref{thm2.4edfrgt} that $\Delta^*f_2 \in L^2_D \ominus H^2_D$ and
$\Delta^*f_2 \neq 0$. \ We thus have $H_{\Delta^*}f=J(\Delta^*
f_2)$, and hence, $||H_{\Delta^*}f||=||\Delta^* f_2|| \neq 0$, which
implies that $f \notin \hbox{ker}\, H_{\Delta^*}$. \ We thus have
that
$$
\hbox{ker}\, H_{\Delta^*}=\Delta H^2_{D} \bigoplus \hbox{ker}\,
\Delta^*.
$$
Thus it follows from (a) that
$$
\hbox{ker}\, H_{\Delta^*}=\Delta H^2_{D} \bigoplus \Delta_c
H^2_{D^{\prime}}=[\Delta,  \Delta_c] H^2_{D \oplus D^{\prime}},
$$
which gives (c). \ This completes the proof.
\end{proof}

\vskip 1cm

%
%
%
%
%
%

\noindent {\bf \S\ 4.3. The degree of non-cyclicity} \

\bigskip
\noindent For a subset $F$ of $H^2_E$, let $E_F^*$ denote the
smallest $S_E^*$-invariant subspace containing $F$, i.e.,
$$
E_F^*=\bigvee \bigl\{S_E^{*n} F:\ n\ge 0\bigr\}.
$$
Then by the Beurling-Lax-Halmos Theorem, $E_F^*=\mathcal H(\Delta)$
for an inner function $\Delta$ with values in $\mathcal B(D, E)$. \
In general, if $\dim E=1$, then every $S_E^*$-invariant subspace $M$
admits a cyclic vector, i.e., $M=E_f^*$ for some $f\in H^2$. \
However, if $\dim E\ge 2$, then this is not such a case. \ For
example, if $M=\mathcal H(\Delta)$ with
$\Delta=\left[\begin{smallmatrix} z&0\\0&z\end{smallmatrix}\right]$,
then $M$ does not admit a cyclic vector, i.e., $M\ne E_f^*$ for any
vector $f\in H^2_{\mathbb C^2}$.

\medskip

If $\Phi\in H^2_s(\mathcal B(D,E))$ and $\{d_k\}_{k\geq 1}$ is an
orthonormal basis for $D$, write
$$
\phi_k:= \Phi d_k \in H^2_E\cong H^2_s(\mathcal B(\mathbb C, E)).
$$
We then define
$$
\{\Phi\}:= \{\phi_k\}_{k\ge 1}\subseteq H^2_E.
$$
Hence, $\{\Phi\}$ may be regarded as the set of ``column" vectors
$\phi_k$ (in $H^2_E$) of $\Phi$, in which case we may think of $\Phi$ as
an infinite matrix-valued function.

\bigskip


\begin{lem}\label{thm327}
For $\Phi\in H^2_s(\mathcal B(D,E))$, we have
\begin{equation}\label{ephistar}
E_{\{\Phi\}}^*=\hbox{cl ran}\,H_{\overline z \breve\Phi}.
\end{equation}
\end{lem}

\begin{rem}
By definition, $\{\Phi\}$ depends on the
orthonormal basis of $D$. \ However, Lemma \ref{thm327} shows that
$E_{\{\Phi\}}^*$ is independent of a particular choice of the
orthonormal basis of $D$ because the right-hand side of
(\ref{ephistar}) is independent of the orthonormal basis of $D$.
\end{rem}

\begin{proof}[Proof of Lemma \ref{thm327}] \ 
We first claim that if $f\in H^2_{E}$, then
\begin{equation}\label{invinv}
E_f^*= \hbox{cl ran} H_{[\overline{z}\breve{f}]}.
\end{equation}
To see this, observe that for each $k=1,2, \cdots$,
$$
\aligned S_E^{*k} f
&=\overline z \sum_{j=0}^\infty \widehat{f}(k+j) z^{j+1}\\
&=J\Biggl(\sum_{j=0}^\infty \widehat{f}(k+j) \overline z^{j+1}\Biggr)\\
&=JP_-\Biggl(z^{k-1}\sum_{j=0}^\infty \widehat{f}(j)\overline z^j\Biggr)\\
&=JP_-\left(z^{k-1}\breve f\right)\\
&=H_{[\overline{z}\breve{f}]} z^{k},
\endaligned
$$
which proves (\ref{invinv}).  \ Let $\{d_k\}_{k \geq 1}$ be an
orthonormal basis for $D$, and let $\phi_k:= \Phi d_k$. \ Since by
(\ref{invinv}), $E_{\phi_k}^*=\hbox{cl
ran}\,H_{[\overline{z}\breve{\phi}_k]}$ for each $k=1,2,3, \cdots$,
it follows that
$$
E_{\{\Phi\}}^*=\bigvee \hbox{ran}\, H_{[\overline{z}
\breve{\phi}_k]}= \hbox{cl ran}\,H_{\overline{z}\breve{\Phi}},
$$
which gives the result.
\end{proof}

\bigskip

We now introduce:

\medskip

\begin{df}  Let $F \subseteq H^2_{E}$.
The {\it degree of non-cyclicity}, denoted by $\hbox{nc}(F)$, of $F$
is defined by the number
$$
\hbox{nc}(F):=\sup_{\zeta\in\mathbb D}\, \dim \bigl\{g(\zeta): g\in
H^2_E\ominus E_F^*\bigr\}.
$$
We will often refer to $\hbox{nc}(F)$ as the {\it nc-number} of $F$.
\end{df}

\bigskip

Since $E_F^*$ is an invariant subspace for $S_E^*$, it follows from
the Beurling-Lax-Halmos Theorem that $E_F^*=\mathcal H(\Delta)$ for
some inner function $\Delta$ with values in $\mathcal B(D, E)$. \
Thus
$$
\hbox{nc}(F)=\sup_{\zeta \in \mathbb D}\, \hbox{dim}\,
\bigl\{g(\zeta): \, g\in \Delta H^2_{D}\bigr\}=\dim D.
$$
In particular, $\hbox{nc}(F)\le \dim E$. \ We note that
$\hbox{nc}(F)$ may take $\infty$. \ So it is customary to make the
following conventions: (i) if $n$ is real then $n+\infty=\infty$;
(ii) $\infty+\infty=\infty$. If $\dim E=r<\infty$, then
$\hbox{nc}(F)\le r$ for every subset $F \subseteq H^2_{E}$.  If $F
\subseteq H^2_{E}$ and $\dim E=r<\infty$, then the {\it degree of
cyclicity}, denoted by $\hbox{dc}(F)$, of $F \subseteq H^2_{E}$ is
defined by the number (cf. \cite{VN})
$$
\hbox{dc}(F):=r-\hbox{nc}(F).
$$
In particular, if $E_F^*=\mathcal H(\Delta)$, then $\Delta$ is
two-sided inner if and only if $\hbox{nc} (F)=r$.

\bigskip

The following theorem gives an answer to Question \ref{q111}.

\medskip

\begin{thm}\label{thm7566}
Let  $\Phi$ be a strong $L^2$-function with values in $\mathcal B(D,
E)$.  \ In view of the Beurling-Lax-Halmos Theorem and Lemma
\ref{kerhadjoint}, we may write
$$
 E_{\{\Phi_+\}}^*=\mathcal H(\Delta) \quad
\hbox{and} \quad \ker H_{\breve{\Phi}}^*=\Theta H^2_{E^{\prime}},
$$
for some inner functions $\Delta$ and $\Theta$ with values in
$\mathcal B(E^{\prime\prime}, E)$ and $\mathcal B(E^{\prime}, E)$,
respectively. \ Then
\begin{equation}\label{mmmm}
\Delta=\Theta \Delta_1
\end{equation}
for some two-sided inner function $\Delta_1$ with values in
$\mathcal B(E^{\prime\prime}, E^{\prime})$. Hence, in particular,
\begin{equation}\label{sssdv}
\ker H_{\breve{\Phi}}^*=\Theta H^2_{E^{\prime}} \Longleftrightarrow
\ \hbox{nc}\{\Phi_+\}=\dim E^\prime.
\end{equation}
\end{thm}

\begin{proof}
Suppose that $\hbox{ker}\, H_{\breve{\Phi}}^*=\Theta
H^2_{E^{\prime}}$ for some inner function $\Theta$ with values in
$\mathcal B(E^{\prime},E)$ and $E_{\{\Phi_+\}}^*=\mathcal H(\Delta)$
for some inner function $\Delta$ with values in $\mathcal
B(E^{\prime\prime}, E)$. \  Then it follows from Lemma \ref{thm327}
that
$$
\mathcal H(\Delta)=E_{\{\Phi_+\}}^*=\hbox{cl}\,\hbox{ran}\,
H_{\overline z \breve\Phi}=\bigl(\hbox{ker}\,H_{\overline z
\breve\Phi}^*\bigr)^{\perp}.
$$
It thus follows from Lemma \ref{thm4.2} that
$$
\aligned \Delta H^2_{E^{\prime\prime}}
&=\hbox{ker}\,H_{\overline{z}\breve{\Phi}}^*\\
&=\Bigl\{f \in H^2_{E}: \int_{\mathbb T}\bigl \langle \Phi(z)x, \
z^n f(z)\bigr\rangle_E dm(z)=0 \quad \hbox{for all} \  x \in D\\
&\qquad\qquad\qquad\qquad\qquad\qquad\qquad\qquad\qquad\qquad\hbox{and} \
n=0,1,2,3, \cdots\Bigr\}\\
&\subseteq \Bigl\{f \in H^2_{E}: \int_{\mathbb T}\bigl \langle
\Phi(z)x, \ z^n f(z)\bigr\rangle_E dm(z)=0 \quad \hbox{for all} \  x
\in D \\ &\qquad\qquad\qquad\qquad\qquad\qquad\qquad\qquad\qquad\qquad\hbox{and} \
n=1,2,3, \cdots\Bigr\}\\
&=\hbox{ker}\, H_{\breve{\Phi}}^*=\Theta H^2_{E^{\prime}},
\endaligned
$$
which implies that $\Theta$ is a left inner divisor of $\Delta$.
Thus we can write
\begin{equation}\label{mmmma}
\Delta=\Theta \Delta_1
\end{equation}
for some inner function $\Delta_1 \in H^{\infty}(\mathcal
B(E^{\prime\prime}, E^{\prime}))$. \ By the same argument as above,
we also have $ z\Theta H^2_{E^{\prime}}\subseteq\Delta
H^2_{E^{\prime\prime}}$, so that we may write $z \Theta=\Delta
\Delta_2$ for some inner function $\Delta_2\in H^{\infty}(\mathcal
B(E^{\prime}, E^{\prime\prime}))$. Therefore by (\ref{mmmma}), we
have $zI_{E^{\prime}}=\Delta_1 \Delta_2$, and hence by Lemma
\ref{rem.sdcfcfc}, $\Delta_1$ is two-sided inner. \ This proves
(\ref{mmmm}) and in turn (\ref{sssdv}). \ This completes the proof.
\end{proof}

\medskip

From Theorem \ref{thm7566}, we get several corollaries.

\medskip

\begin{cor}\label{corfkgkhgkhkhk}
Let $\Phi$ be a strong $L^2$-function with value in $\mathcal B(D,
E)$.  \ Then the following statements are equivalent:
\begin{itemize}
\item[(a)] $E^*_{\{\Phi_+\}}=H^2_E$;
\item[(b)] $\hbox{nc}\{\Phi_+\}=0$;
\item[(c)] $\hbox{ker}\,H_{\breve{\Phi}}^*=\{0\}$.
\end{itemize}
\end{cor}

\begin{proof}
Immediate from  Theorem \ref{thm7566}.
\end{proof}

\medskip

\begin{cor}\label{innerdc}  Let
$\Delta$ be an inner function with values in $\mathcal B(D, E)$. \
If $\Delta_c$ is the complementary factor of $\Delta$, with values in $\mathcal
B(D^{\prime},E)$, then
$$
\hbox{nc}\{\Delta\}=\dim D +\dim D^{\prime}.
$$
\end{cor}

\begin{proof}
Immediate from Lemma \ref{thm2.9}(c) and Theorem \ref{thm7566}.
\end{proof}

\medskip

\begin{cor}\label{btdc}
If $\Phi$ is an $n\times m$ matrix $L^2$-function, i.e., $\Phi \in
L^2_{M_{n \times m}}$, then the following are equivalent:
\begin{itemize}
\item[(a)] $\Phi$ is of bounded type;
\item[(b)] $\ker H^*_{\Phi}=\Delta H^2_{\mathbb C^n}$ for some
   two-sided inner matrix function $\Delta$;
\item[(c)]  $\hbox{nc}\,\{\Phi_-\}=n$.
\end{itemize}
\end{cor}

\begin{proof}  The equivalence (a) $\Leftrightarrow$ (c)
follows from \cite[Corollary 2, p.~47]{Ni1} and (\ref{btp}), and the
equivalence (b) $\Leftrightarrow$ (c) follows at once from Theorem
\ref{thm7566}.
\end{proof}

\medskip

The equivalence (a) $\Leftrightarrow$ (b) of Corollary \ref{btdc}
was known from \cite{GHR} for the cases of $\Phi \in
L^{\infty}_{M_{n}}$. \ On the other hand, it was known (\cite[Lemma
4]{Ab}) that if $\phi\in L^\infty$, then
\begin{equation}\label{ab}
\hbox{$\phi$ is of bounded type} \Longleftrightarrow \ker
H_\phi\ne\{0\}.
\end{equation}
The following corollary shows that (\ref{ab}) still holds for
$L^2$-functions.

\medskip

\begin{cor}\label{cor566we}
If $\phi \in L^2$, then $\phi$ is of bounded type if and only if
$\ker H_{\phi}^* \neq \{0\}$.
\end{cor}

\begin{proof}  Immediate from Corollary \ref{btdc}.
\end{proof}

\medskip

\begin{cor}\label{cor312}
If $\Delta$ is an $n\times r$ inner matrix function then the
following are equivalent:
\begin{itemize}
\item[(a)] $\Delta^*$ is of bounded type;
\item[(b)] $\breve{\Delta}$ is of bounded type;
\item[(c)] $[\Delta, \Delta_c]$ is two-sided inner,
\end{itemize}
where $\Delta_c$ is the complementary factor of $\Delta$.
\end{cor}

\begin{proof}
The equivalence (a) $\Leftrightarrow$ (b) is trivial.  \ The
equivalence (b) $\Leftrightarrow$ (c) follows from Lemma
\ref{thm2.9} and Corollary \ref{btdc}.
\end{proof}

\bigskip

The following corollary gives an answer to Question \ref{q222}.

\begin{cor}\label{cor5.hh3333}
If $\Delta$ is an $n\times r$ inner matrix function, then $[\Delta,
\Omega]$ is inner for some $n\times q \ (q\geq 1)$ inner matrix
function $\Omega$ if and only if
$$
q\le \hbox{nc}\{\Delta\}-r.
$$
In particular, $\Delta$ is complemented to a two-sided inner
function if and only if $\hbox{nc}\{\Delta\}=n$.
\end{cor}

\begin{proof} Suppose that  $[\Delta, \Omega]$ is
an inner matrix function for some $n\times q \ (q\geq 1)$ inner
matrix function $\Omega$. \ Then
$$
I_{r+q}=[\Delta, \Omega]^*[\Delta, \Omega]=\begin{bmatrix}
I_r&\Delta^*\Omega\\ \Omega^* \Delta&I_{q}\end{bmatrix},
$$
which implies that $\Omega H^2_{\mathbb C^q}\subseteq \ker
\Delta^*$. \ Since by Lemma \ref{thm2.9}, $\hbox{ker}\,
\Delta^*=\Delta_c H^2_{\mathbb C^p}$, it follows that $\Omega
H^2_{\mathbb C^q}\subseteq \Delta_c H^2_{\mathbb C^p}$, so that
$\Delta_c$ is a left inner divisor of $\Omega$. \ Thus we can write
$$
\Omega=\Delta_c \Omega_1 \quad\hbox{for some $p\times q$ inner
matrix function} \  \Omega_1.
$$
Thus we have $q\le p$. \ But since by Corollary \ref{innerdc},
$\hbox{nc}\{\Delta\}=r+p$, it follows that $q\le
\hbox{nc}\{\Delta\}-r$. \ For the converse, suppose that $q\le
\hbox{nc}\{\Delta\}-r$. \ Then it follows from Corollary
\ref{innerdc} that the complementary factor $\Delta_c$ of $\Delta$
is in $H^{\infty}_{M_{n \times p}}$ for some $p\geq q$.  \ Thus if
we take $\Omega:=\Delta_c\vert_{\mathbb C^q}$, then
$[\Delta,\Omega]$ is inner.
\end{proof}

\bigskip

We give an illuminating example of how to find the nc number.

\begin{ex}
 Let $f$ and $g$ be given in Example \ref{ex3.6}, and let
$$
\Phi:=\begin{bmatrix}f&f&0\\ g&g&0\\
0&0&a\end{bmatrix} \quad(a \in H^{\infty})
$$
To find the degree of non-cyclicity of $\Phi$, write $
\Psi:=\left[\begin{smallmatrix}f&f\\g&g\end{smallmatrix}\right]$. \
Then it follows that
$$
\begin{bmatrix}h_1\\h_2\\h_3\end{bmatrix} \in \hbox{ker}\, H_{\breve\Phi}^*
\Longleftrightarrow
\begin{bmatrix}h_1\\h_2 \end{bmatrix} \in \ker
H_{\Psi^*} \ \hbox{and} \ h_3 \in \ker H_{\overline{a}}.
$$
Case 1: If $\overline{a}$ is not of bounded type, then $\hbox{ker}\,
H_{\breve\Phi}^*=[f \ g \ 0]^t H^2$. \ By Theorem \ref{thm7566},
$\hbox{nc}\{\Phi\}=1$.

\smallskip

\noindent Case 2: If $\overline{a}$ is of bounded type of the form
$a=\theta \overline{b}$ (coprime), then
$$
\hbox{ker}\,
H_{\Phi^*}^*=\begin{bmatrix}f&0\\g&0\\0&\theta\end{bmatrix}
H^2_{\mathbb C^2}.
$$
By Theorem \ref{thm7566}, $\hbox{nc}\{\Phi\}=2$.
\end{ex}

\vskip 1cm

%
%
%
%

\noindent {\bf \S\ 4.4. Strong $L^2$-functions of bounded type} \

\bigskip
\noindent We introduce the notion of ``bounded type" for strong
$L^2$-functions. \ Recall that a matrix-valued function of bounded
type was defined by a matrix whose entries are of bounded type (see
p.~\pageref{matrixbt}). \ But this definition is not appropriate for
operator-valued functions, in particular strong $L^2$-functions,
even though the terminology of ``entry'' can be properly interpreted. \
Thus we need a new idea about how to define a ``bounded type" strong
$L^2$-functions, which is equivalent to the condition that each
entry is of bounded type when the function is matrix-valued. \
Our motivation stems from the equivalence
(a)$\Leftrightarrow$(b) in Corollary \ref{btdc}.

\medskip

\begin{df}\label{dfbundd}
A strong $L^2$-function $\Phi$ with values in $\mathcal B(D,E)$ is
said to be of {\it bounded type} if $\hbox{ker}\,
H_{{\Phi}}^*=\Theta H^2_{E}$ for some two-sided inner function
$\Theta$ with values in $\mathcal B(E)$.
\end{df}

\medskip

\ On the other hand, in \cite{FB}, it was shown that if $\Phi$
belongs to $L^\infty (\mathcal B(D,E))$, then $\Phi$ admits a
Douglas-Shapiro-Shields factorization (see p.~\pageref{DSS}) if and
only if $E_{\{\Phi_+\}}^*=\mathcal H(\Theta)$ for a two-sided inner
function $\Theta$.  \ Thus, by Theorem \ref{thm7566}, we can see that
if $\Phi\in L^\infty(\mathcal B(D, E))$, then
\begin{equation}\label{btf}
\hbox{$\breve{\Phi}$ is of bounded type}\\
\Longleftrightarrow\ \hbox{$\Phi$ admits a  Douglas-Shapiro-Shields
factorization}.
\end{equation}

\medskip

We can prove more:

\begin{lem}\label{thmdfgfgty}
Let $\Phi$ be a strong $L^{2}$-function with values in $\mathcal
B(D,E)$. \ Then the following are equivalent:

\begin{itemize}

\item[(a)]  $\breve{\Phi}$ is of bounded type;
\item[(b)] $E_{\{\Phi_+\}}^*=\mathcal H(\Delta)$ for some two-sided inner function
      $\Delta$ with values in $\mathcal B(E)$;
\item[(c)]  $E_{\{\Phi_+\}}^*\subseteq \mathcal H(\Theta)$ for
      some two-sided inner function $\Theta$ with values in $\mathcal B(E)$;
\item[(d)] $\{\Phi_+\}\subseteq \mathcal H(\Theta)$ for some two-sided inner
             function $\Theta$ with values in $\mathcal B(E)$;
\item[(e)] For $\{\varphi_{k_1}, \varphi_{k_2}, \cdots \} \subseteq \{\Phi\}$, write $\Psi\equiv[\varphi_{k_1}, \varphi_{k_2},
\cdots]$. Then $\breve{\Psi}$ is of bounded type.
\end{itemize}
\end{lem}

\begin{proof} (a) $\Rightarrow$ (b):  Suppose that $\breve{\Phi}$ is of bounded
type. \ Then $\hbox{ker}\, H_{\breve{\Phi}}^*=\Theta H^2_{E}$ for
some two-sided inner function $\Theta$ with values in $\mathcal
B(E)$. \ It thus follows from Theorem \ref{thm7566} that
$E_{\{\Phi_+\}}^*=\mathcal H(\Delta)$ for some two-sided inner
function $\Delta$ with values in $\mathcal B(E)$.

(b) $\Rightarrow$ (c), (c) $\Rightarrow$ (d): Clear.

(d) $\Rightarrow$ (e):  Suppose that $\{\varphi_{k_1},
\varphi_{k_2}, \cdots \} \subseteq \{\Phi\}$ and
$\{\Phi_+\}\subseteq \mathcal H(\Theta)$ for some two-sided inner
function $\Theta\in H^{\infty}(\mathcal B(E))$. \ Write $\Psi\equiv
[\varphi_{k_1}, \varphi_{k_2}, \cdots ]$. \ Then
$\{\Psi_+\}\subseteq \mathcal H(\Theta)$, so that $E_{\{\Psi_+\}}^*
\subseteq \mathcal H(\Theta)$. \ Suppose that
$E_{\{\Psi_+\}}^*=\mathcal H(\Delta)$ for some inner function
$\Delta$ with values in $\mathcal B(D^{\prime}, E)$. \ Thus $\Theta
H^2_{E} \subseteq \Delta H^2_{D^{\prime}}$, so that by Lemma
\ref{rem.sdcfcfc}, $\Delta$ is two-sided inner. \ Thus, by Theorem
\ref{thm7566}, $\hbox{ker}\, H_{\breve{\Psi}}^*=\Omega H^2_{E}$ for
some two-sided inner function $\Omega$ with values in $\mathcal
B(E)$, so that $\Psi$ is of bounded type.

(e) $\Rightarrow$ (a): Clear.
\end{proof}

\medskip

\begin{cor}\label{thmthjdkfjkf}
Let $\Delta$ be an inner function with values in $\mathcal B(D, E)$.
\  Then
$$
\breve{\Delta} \ \hbox{is of bounded type} \Longleftrightarrow
[\Delta, \Delta_c]\ \hbox{is two-sided inner},
$$
where $\Delta_c$ is the complementary factor of $\Delta$. \ Hence,
in particular, if $\Delta$ is a two-sided inner function with values
in $\mathcal B(E)$, then $\breve{\Delta}$ is of bounded type.
\end{cor}

\begin{proof}
The first assertion follows from Lemma \ref{thm2.9}. \ The second
assertion follows from the first assertion together with the
observation that if $\Delta$ is two-sided inner then $[\Delta,
\Delta_c]=\Delta$.
\end{proof}

\medskip

\begin{cor}\label{remdjfghhghgh}
Let $\Delta$ be an inner function with values in $\mathcal B(D, E)$.
\  Then $[\Delta, \Omega]$ is two-sided inner for some inner
function $\Omega$ with values in $\mathcal B(D^{\prime}, E)$ if and
only if $\breve{\Delta}$ is of bounded type.
\end{cor}

\begin{proof}  Suppose that  $[\Delta, \Omega]$ is two-sided
inner for some inner function $\Omega$ with values in $\mathcal
B(D^{\prime}, E)$.  \ Then $\Delta^* \Omega=0$, so that $\Omega
H^2_{D^{\prime}}\subseteq \hbox{ker}\, \Delta^*=\Delta_c
H^2_{D^{\prime\prime}}$. \ Thus $\Delta_c$ is a left inner divisor
of $\Omega$, and hence $[\Delta, \Delta_c]$ is a left inner divisor
of $[\Delta, \Omega]$. \ Therefore by Lemma \ref{rem.sdcfcfc},
$[\Delta, \Delta_c]$ is two-sided inner, so that by Corollary
\ref{thmthjdkfjkf}, $\breve{\Delta}$ is of bounded type. \ The
converse follows at once  from Corollary \ref{thmthjdkfjkf} with
$\Omega=\Delta_c$.
\end{proof}

\bigskip

We now ask: If $\Delta\equiv[\delta_1, \delta_2, \cdots,
\delta_m] \in H^{\infty}_{M_{n\times m}}$ is an inner matrix
function, does there exist $j$ ($1\leq j\leq m$) such that
$\hbox{dc}\{\delta_j\}=\hbox{dc}\{\Delta\}$\, ? \ The answer, however, is
negative. \ To see this, let $f$ and $g$ be given in Example
\ref{ex3.6} and let
$$
\Delta:=\begin{bmatrix}f&0\\g&0\\0&f\\0&g\end{bmatrix}\equiv
\begin{bmatrix} \delta_1, \delta_2\end{bmatrix}.
$$
Since
$$
\begin{bmatrix}f&0&0\\g&0&0\\0&1&0\\0&0&1\end{bmatrix}\ \hbox{is
inner},
$$
in view of Corollary \ref{cor5.hh3333}, we have
$\hbox{dc}(\delta_1)\le 1$. \ But since $\hbox{dc}(\delta_1)\ne 0$
(because $\delta_1^*$ is not of bounded type), it follows that
$\hbox{dc}\{\delta_1\}=1$. Similarly, $\hbox{dc}\{\delta_2\}=1$. \
However, we have $\hbox{dc}\{\Delta\}=2$, because we can show that
$\Delta_c=0$.

\vskip 1cm

%
%
%
%
%
%

\noindent {\bf \S\ 4.5. Meromorphic pseudo-continuations of bounded
type} \

\bigskip
\noindent In general, if a strong $L^2$-function $\Phi$ is of
bounded type then we cannot guarantee that each entry
$\phi_{ij}\equiv\langle \Phi d_j, \ e_i\rangle$ is of bounded type, where  $\{d_j\}$ and $\{e_i\}$ are orthonormal bases
of $D$ and $E$, respectively. \ But if we strengthen the assumption
then we may have the assertion. \ To see this, for a function $\Psi:
\mathbb{D}^e\equiv\{z: 1<|z|\le\infty\}\to \mathcal B(D,E)$,  we define
$\Psi_{\mathbb D}: \mathbb D\to \mathcal B(E,D)$ by
$$
\Psi_{\mathbb D}(\zeta):=\Psi^*(1/\overline{\zeta})\quad \hbox{for
$\zeta\in\mathbb D$}.
$$
If $\Psi_{\mathbb D}$ is a strong $H^2$-function, inner, and
two-sided inner with values in $\mathcal B(E,D)$, then we shall say
that $\Psi$ is a strong $H^2$-function, inner, and two-sided inner
in $\mathbb{D}^e$ with values in $\mathcal B(D,E)$, respectively.

A $\mathcal B(D, E)$-valued function $\Psi$ is said to be {\it
meromorphic of bounded type} in $\mathbb{D}^e$ if it can be represented by
$$
\Psi=\frac{G}{\theta},
$$
where $G$ is a strong $H^2$-function in $\mathbb{D}^e$, with values in
$\mathcal B(D,E)$ and $\theta$ is a scalar inner function in $\mathbb{D}^e$.
(cf. \cite{Fu}). \ A function $\Phi \in L^2_s(\mathcal B(D, E))$ is
said to have a {\it meromorphic pseudo-continuation $\hat\Phi$ of
bounded type} in $\mathbb{D}^e$ if $\hat\Phi$ is meromorphic of bounded type
in $\mathbb{D}^e$ and $\Phi$ is the nontangential SOT limit of $\hat\Phi$,
that is, for all $x \in D$,
$$
\Phi(z)x= \hat{\Phi}(z)x:= \lim_{rz\to z}\hat{\Phi}(rz)x \quad
\hbox{for almost all} \  z\in\mathbb T.
$$
Note that for almost all $z\in\mathbb T$,
$$
\Phi(z)x=\lim_{rz\to z}\hat{\Phi}(rz)x=\lim_{rz\to z}
\hat{\Phi}_{\mathbb D}^*(r^{-1}z)x=\hat{\Phi}_{\mathbb D}^*(z) x
\quad (x \in D).
$$

\medskip
We then have:


\begin{lem}\label{thgg668} Let $\Phi$ be a strong $L^{2}$-function
with values in $\mathcal B(D,E)$. \ If $\Phi$ has a meromorphic
pseudo-continuation of bounded type in $\mathbb{D}^e$, then $\breve\Phi$ is of
bounded type.
\end{lem}

\begin{proof} Suppose that  $\Phi$ has a meromorphic pseudo-continuation of bounded
type in $\mathbb{D}^e$.  \ Thus the meromorphic pseudo-continuation $\hat\Phi$
of $\Phi$ can be written as
$$
\hat{\Phi}(\zeta):=\frac{G(\zeta)}{\delta(\zeta)} \quad(\zeta \in
\mathbb{D}^e),
$$
where $G$ is a strong $H^2$-function in $\mathbb{D}^e$, with values in
$\mathcal B(D,E)$ and $\delta$ is a scalar inner function in $\mathbb{D}^e$.\
Then for all $x \in D$,
$$
\Phi(z)x=\hat{\Phi}_{\mathbb D}^*(z) x=\delta_{\mathbb
D}(z)G_{\mathbb D}^*(z)x \quad \ \hbox{for almost all} \ z \in
\mathbb {T}.
$$
Thus for all $x\in D$, $p \in \mathcal P_E$, and $n=1,2,3,\cdots$,
$$
\aligned \int_{\mathbb T}\bigl\langle \Phi(z)x,\,
z^n{\delta_{\mathbb D}}(z)p(z)\bigr \rangle_{E}\,dm(z)
&=\int_{\mathbb T}\bigl\langle G_{\mathbb D}^*(z)x, \
z^np(z)\bigr \rangle_E\, dm(z)\\
&=\bigl\langle x, \ z^n G_{\mathbb D}(z)p(z)\bigr \rangle_{L^2_D}=0,
\endaligned
$$
where the last equality follows from the fact that $z^n G_{\mathbb
D}(z)p(z)\in zH^2_D$. \ Thus by Lemma \ref{thm4.2}, we can see that
\begin{equation}\label{newsharp}
\delta_{\mathbb D} H^2_E=\hbox{cl}\, \delta_{\mathbb D} \mathcal P_E
\subseteq \hbox{ker}\, H_{\breve \Phi}^*.
\end{equation}
In view of Lemma \ref{kerhadjoint}, $\hbox{ker}\,
H_{\breve\Phi}^*=\Delta H^2_{E^{\prime}}$ for some inner function
$\Delta$ with values in $\mathcal B(E^\prime, E)$. \ Thus $\Delta$
is a left inner divisor of $\delta_{\mathbb D}I_{E}$ (cf. \cite{FF},
\cite{Pe}). \ Thus, it follows from Lemma \ref{rem.sdcfcfc} that
that $\Delta$ is two-sided inner, so that $\breve\Phi$ is of bounded
type.
\end{proof}

\medskip

The following lemma was proved in \cite{Fu1} under the more restrictive setting of
$H^{\infty} (\mathcal B(D, E))$.

\begin{lem}\label{thgg6688}
Let $\Phi \in L^\infty (\mathcal B(D, E))$. Then the following are
equivalent:

\begin{itemize}
\item[(a)]  $\Phi$ has a meromorphic pseudo-continuation of
bounded type in $\mathbb{D}^e$;
\item[(b)]  $\theta H^2_{E} \subseteq \hbox{ker}\,H_{\Phi^*}$ for some scalar
inner function $\theta$;
\item[(c)] $\Phi=\theta A^*$  for
a scalar inner function $\theta$ and some $A\in H^\infty (\mathcal
B(E,D))$.
\end{itemize}
\end{lem}

\begin{proof} First of all, recall that
$L^\infty (\mathcal B(D, E))\subseteq L^2_s(\mathcal B(D, E))$.

(a) $\Rightarrow$ (b): This follows from (\ref{newsharp}) in the
proof of Lemma \ref{thgg668}.

(b) $\Rightarrow$ (c): Suppose that $\theta H^2_{E} \subseteq
\hbox{ker}\,H_{\Phi^*}$ for some scalar inner function $\theta$. \
Put $A:=\theta \Phi^*$. Then  $A$ belongs to $H^\infty (\mathcal
B(E,D))$ and $\Phi=\theta A^*$.

(c) $\Rightarrow$ (a): Suppose that $\Phi=\theta A^*$  for a scalar
inner function $\theta$ and some $A\in H^\infty (\mathcal B(E,D))$.
\  Thus it follows from Lemma \ref{strongh2} that $A$ is a strong
$H^2$-function. Let
$$
{\hat\Phi}(\zeta):=\frac{A^*(1/\overline{\zeta})}
{\overline{\theta}(1/\overline{\zeta})} \quad (\zeta \in \mathbb{D}^e).
$$
Then $\hat{\Phi}$ is meromorphic of bounded type in $\mathbb{D}^e$ and for
all $x \in D$,
$$
\hat{\Phi}(z)x=\frac{A^*(z)x}
{\overline{\theta}(z)}=\theta(z)A^*(z)x=\Phi(z)x \quad \hbox{for
almost all} \ z \in \mathbb  T,
$$
which implies that $\Phi$ has a meromorphic pseudo-continuation of
bounded type in $\mathbb{D}^e$.
\end{proof}

\medskip

An examination of the proof of  Lemma \ref{thgg6688} shows that
Lemma \ref{thgg6688} still holds for every function $\Phi \in
L^2_{\mathcal B(D,E)}$.

\medskip

\begin{cor}\label{cor32900}
If $\Phi\in L^2_{\mathcal B(D,E)}$, then Lemma \ref{thgg6688} holds
with ${A}\in H^2_{\mathcal B(E,D)}$ in place of $A \in
H^\infty(\mathcal B(E,D))$.
\end{cor}

\medskip

The following proposition gives an answer to an opening remark of
this section.

\begin{prop}\label{corapp-thm1}
Let $D$ and $E$ be separable complex Hilbert spaces and let
$\{d_j\}$ and $\{e_i\}$ be orthonormal bases of $D$ and $E$,
respectively. \ If $\Phi \in L^2_{\mathcal B(D,E)}$ has a
meromorphic pseudo-continuation of bounded type in $\mathbb{D}^e$, then
$\breve{\phi}_{ij}(z)\equiv\langle \breve{\Phi}(z) d_j, \
e_i\rangle_E$ is of bounded type for each $i,j$.
\end{prop}

\begin{proof}
Let $\Phi\in L^2_{\mathcal B(D,E)}$. \ Suppose that  $\Phi$ has a
meromorphic pseudo-continuation of bounded type in $\mathbb{D}^e$. \ Then by
Corollary \ref{cor32900}, $\Phi=\theta A^*$ for a scalar inner
function $\theta$ and some $A\in H^2_{\mathcal B(E, D)}$. \ Write
$$
\phi_{ij}(z):=\langle {\Phi}(z) d_j, \ e_i\rangle_E \quad \hbox{and}
\quad a_{ij}(z):=\langle \widetilde{A}(z) d_j, \ e_i\rangle_E.
$$
Then for each $i,j$,
$$
\begin{aligned}
\int_{\mathbb T}|\phi_{ij}(z)|^2dm(z)
&=\int_{\mathbb T}|\langle {\Phi}(z) d_j, \ e_i\rangle_E|^2dm(z)\\
&\leq \int_{\mathbb T}||\Phi(z)||^2_{\mathcal B(D,E)}dm(z) <\infty,
\end{aligned}
$$
which implies $\phi_{ij}\in L^2$. \ Similarly, $a_{ij}\in L^2$ and
for $n=1,2,3,\cdots$,
$$
\widehat{a_{ij}}(-n)=\int_{\mathbb T}z^n\langle \widetilde{A}(z)
d_j, \ e_i\rangle_E dm(z) =\langle d_j, \  z^{-n} \breve{A}(z)
e_i\rangle_{L^2_D} =0,
$$
which implies $a_{ij}\in H^2$. \ Note that
$$
\breve{\phi}_{ij}(z)=\breve{\theta}(z) \langle \widetilde{A}(z) d_j,
\ e_i\rangle_E=\breve{\theta}(z)a_{ij}(z)\,,
$$
which implies that $\breve{\phi}_{ij}$ is of bounded type for each
$i,j$.
\end{proof}

\medskip

\begin{ex}
The converse of Lemma \ref{thgg668} is not true in general. \ To see
this, let $\{\alpha_n\}$ be a sequence of distinct points in
$\mathbb D$ such that $\sum_{n=1}^{\infty}(1-|\alpha_n|)=\infty$ and
put $\Delta:=\hbox{diag}(b_{\alpha_n})$, where $b_{\alpha_n}(z):=
\frac{z-\alpha_n}{1-\overline{\alpha_n}z}$. \ Then $\Delta$ is
two-sided inner, and hence by Lemma \ref{thmthjdkfjkf},
$\breve{\Delta}$ is of bounded type. \ On the other hand, by Lemma
\ref{thm2.9}, $ \hbox{ker}\, H_{\Delta^*}=\Delta H^2_{\ell^2}$. \
Thus if $\Delta$ had a meromorphic pseudo-continuation of bounded
type in $\mathbb{D}^e$, then by Lemma \ref{thgg6688}, we would have $\theta
H^2_{\ell^2}\subseteq \Delta H^2_{\ell^2}$ for a scalar inner
function $\theta$, so that we should have $\theta(\alpha_n)=0$ for
each $n=1,2,\cdots$, and hence $\theta=0$, a contradiction. \
Therefore, $\Delta$ cannot have a meromorphic pseudo-continuation of
bounded type in $\mathbb{D}^e$.
\end{ex}

\bigskip

For matrix-valued cases, a function having a meromorphic
pseudo-continuation of bounded type in $\mathbb{D}^e$ is actually a function whose flip is of bounded type.

\begin{cor}\label{cor512,222}
For $\Phi\equiv[\phi_{ij}]\in L^2_{M_{n\times m}}$, the following
are equivalent:

\begin{itemize}
\item[(a)] $\Phi$ has a meromorphic pseudo-continuation
of bounded type in $\mathbb{D}^e$;
\item[(b)] $\breve\Phi$ is of bounded type;
\item[(c)] $\breve{\phi}_{ij}$ is of bounded type for each $i,j$.
\end{itemize}
\end{cor}

\begin{proof}  (a) $\Rightarrow$ (b):
This follows from Lemma \ref{thgg668}.

(b) $\Rightarrow$ (a):  Suppose that $\breve{\Phi}$ is of bounded
type. \ Then $\hbox{ker}\, H_{\breve{\Phi}}^*=\Theta H^2_{\mathbb
C^n}$ for some two-sided inner function $\Theta \in
H^{\infty}_{M_n}$. \ Thus by the Complementing Lemma (cf.
p.~\pageref{Complementing}), there exist a scalar inner function
$\theta$ and a function $G$ in $H^{\infty}_{M_{n}}$ such that
$G\Theta=\Theta G=\theta I_n$, and hence, $\theta H^2_{\mathbb C^n}
=\Theta G H^2_{\mathbb C^n}\subseteq \Theta H^2_{\mathbb C^n}=
\hbox{ker}\,H^*_{\breve{\Phi}}$. \ It thus follows from Corollary
\ref{cor32900} that $\breve\Phi$ has a meromorphic
pseudo-continuation of bounded type in $\mathbb{D}^e$.

(a) $\Leftrightarrow$ (c):  This follows from Corollary
\ref{cor32900} and Proposition \ref{corapp-thm1}.
\end{proof}

\medskip

However, by contrast to the matrix-valued case, it may happened that
an $L^\infty$-function $\Phi$ is not of bounded type in the sense of
Definition \ref{dfbundd} even though each entry $\phi_{ij}$ of
$\Phi$ is of bounded type.

\medskip

\begin{ex}  Let $\{\alpha_j\}$ be a sequence of distinct points in $(0,1)$
satisfying $\sum_{j=1}^{\infty}(1-\alpha_j)<\infty$. \ For each $j
\in \mathbb{Z}_+$, choose a sequence $\{\alpha_{ij}\}$ of distinct
points on the circle $C_j:=\{z\in \mathbb C: |z|=\alpha_j\}$. \ Let
$$
B_{ij}:=\frac{\overline{b}_{\alpha_{ij}}}{(i+j)!} \quad(i,j \in
\mathbb Z_+),
$$
where $b_{\alpha}(z):=\frac{z-\alpha}{1-\overline{\alpha}z}$, and
let
$$
\Phi:=[B_{ij}]=\begin{bmatrix}
\frac{\overline{b}_{\alpha_{11}}}{2!}&\frac{\overline{b}_{\alpha_{12}}}{3!}
            &\frac{\overline{b}_{\alpha_{13}}}{4!}&\cdots\\
\frac{\overline{b}_{\alpha_{21}}}{3!}&\frac{\overline{b}_{\alpha_{22}}}{4!}
            &\frac{\overline{b}_{\alpha_{23}}}{5!}&\cdots\\
\frac{\overline{b}_{\alpha_{31}}}{4!}&\frac{\overline{b}_{\alpha_{32}}}{5!}
            &\frac{\overline{b}_{\alpha_{33}}}{6!}&\cdots\\
\vdots&\vdots&\vdots
\end{bmatrix}.
$$
Observe that
$$
\sum_{i,j}|B_{ij}(z)|^2=\sum_{i}\frac{i}{((1+i)!)^2}\leq\sum_{i}
\frac{1}{(1+i)^2}<\infty,
$$
which implies that $\Phi \in L^{\infty}(\mathcal B(\ell^2))$. \ For
a function $f \in H^2_{\ell^2}$, we write $f=(f_1,f_2,f_3,
\cdots)^t$ ($f_n\in H^2$). Thus if $f=(f_1,f_2,f_3, \cdots)^t \in
\hbox{ker}\, H_{\Phi}$, then
$\sum_{j}\frac{\overline{b}_{\alpha_{ij}}}{(i+j)!}f_j \in H^2$ for
each $i\in\mathbb Z_+$, which forces that $f_j(\alpha_{ij})=0$ for
each $i,j$. \ Thus $f_j=0$ for each $j$ (by the Identity Theorem). \
Therefore we can conclude that $\hbox{ker}\,
H_{\widetilde{\Phi}}^*=\{0\}$, so that $\widetilde{\Phi}$ is not of
bounded type. \ But we note that every entry of $\widetilde{\Phi}$
is of bounded type.
\end{ex}

\bigskip

We conclude this chapter with an application to $C_0$-contractions.

The class $C_{0\, \bigcdot}$ denotes the set of all contractions
$T\in\mathcal {B(H)}$ satisfying the condition (\ref{stronglimit}). \
The class $C_{00}$ denotes the set of all contractions $T\in\mathcal
{B(H)}$ such that $\lim_{n\to\infty} T^n x=0$ and
$\lim_{n\to\infty}T^{*n}x=0$ for each $x\in H$. \ It was known
(\cite[p.43]{Ni1}) that if $T$ is a $C_{0\, \bigcdot}$-contraction  with
characteristic function $\Delta$ (i.e., $T\cong S_E^*\vert_{\mathcal
H(\Delta)}$), then
\begin{equation}\label{C00}
T\in C_{00}\ \Longleftrightarrow \ \hbox{$\Delta$ is {\it two-sided}
inner}.
\end{equation}
A contraction $T\in\mathcal {B(H)}$ is called a {\it completely
non-unitary} (c.n.u.) if there exists no nontrivial reducing
subspace on which $T$ is unitary. \ The class $C_0$ is the set of
all c.n.u. contractions $T$ such that there exists a nonzero
function $\varphi\in H^\infty$ annihilating $T$, i.e.,
$\varphi(T)=0$, where $\varphi(T)$ is given by the calculus of
Sz.-Nagy and ~Foia\c s. We can easily check that $C_0\subseteq
C_{00}$. \ Moreover, it is well known (\cite[p.73]{Ni1}) that if
$T:= P_{\mathcal H(\Delta)}S_{E}\vert_{\mathcal H(\Delta)}\in
C_{00}$ and $\varphi\in H^\infty$, then
\begin{equation}\label{C0char}
\varphi(T)=0\ \Longleftrightarrow\ \exists\, G\in H^\infty (\mathcal
B(E))\ \hbox{such that} \ G\Delta=\Delta G=\varphi I_E.
\end{equation}
The theory of spectral multiplicity for operators of class $C_0$ has
been well developed (see \cite[Appendix 1]{Ni1}, \cite{SFBK}). \ If
$T\in C_0$, then there exists an inner function $m_T$ such that
$m_T(T)=0$ and
$$
\varphi\in H^\infty, \ \varphi(T)=0\ \Longrightarrow \
\varphi/m_T\in H^\infty.
$$
The function $m_T$ is called the {\it minimal annihilator} of the
operator $T$.

\bigskip

In view of (\ref{C00}), we may ask what is a condition on the
characteristic function $\Delta$ of $T$ for a $C_{0
\, \bigcdot}$-contraction $T$ to belong to the class $C_0$. \ The
following proposition gives an answer.

\begin{prop}\label{C0condition}
Let $T:=S_E^*|_{\mathcal H(\Delta)}$ for an inner function $\Delta$
with values in $\mathcal B(D, E)$. \ Then the following are
equivalent:
\begin{itemize}
\item[(a)] $T \in C_0$;
\item[(b)] $\Delta$ is two-sided inner and has a meromorphic pseudo-continuation of
bounded type in $\mathbb{D}^e$.
\end{itemize}
Hence,  in particular, if $\Delta$ is an inner matrix function then
$T\in C_{0}$ if and only if $T\in C_{00}$.
\end{prop}

\begin{proof}
(a) $\Rightarrow$ (b): Suppose $T\in C_0$, and hence $\varphi(T)=0$
for some nonzero function $\varphi\in H^{\infty}$. \ Then $T\in
C_{00}$, so that by the above remark, $\Delta$ is two-sided inner. \
Thus by the Model Theorem, we have
$$
T\cong P_{\mathcal H(\widetilde{\Delta})}S_{E}|_{\mathcal
H(\widetilde{\Delta})}.
$$
It thus follows from (\ref{C0char}) that there exists $\Omega \in
H^{\infty}(\mathcal B(E))$ such that $
\widetilde{\Delta}\Omega=\Omega \widetilde{\Delta}=\varphi I_E$. \
Thus $H_{\Delta^*}\bigl(\widetilde{\varphi}H^2_E\bigr)=
H_{\Delta^*}\bigl(\Delta\widetilde{\Omega}H^2_E\bigr)=0$. \ We thus
have
$$
\widetilde{\varphi}^iH^2_E\subseteq\hbox{cl}\,\widetilde{\varphi}H^2_E
\subseteq \hbox{ker}\, H_{\Delta^*}.
$$
It thus follows from Lemma \ref{thgg6688} that $\Delta$ has a
meromorphic pseudo-continuation of bounded type in $\mathbb{D}^e$. \ This
gives the implication (a)$\Rightarrow$(b).

(b) $\Rightarrow$ (a): Suppose that $\Delta$ is two-sided inner and
has a meromorphic pseudo-continuation of bounded type in $\mathbb{D}^e$. \
Then by Lemma \ref{thm2.9} and Lemma \ref{thgg6688}, there exists a
scalar function $\delta$ such that $\delta
H^2_E\subseteq\hbox{ker}\, H_{\Delta^*}=\Delta H^2_E$. \ Thus we may
write $\delta I_{E}=\Delta \Omega=\Omega \Delta$ for some $\Omega
\in H^{\infty}(\mathcal B(E))$. \ Thus we have
$$
\delta\bigl(P_{\mathcal
H(\Delta)}S_E|_{\mathcal H(\Delta)}\bigr)=P_{\mathcal H(\Delta)}
(\delta I_E)|_{\mathcal H(\Delta)}=0,
$$
so that
$$
\widetilde{\delta}(T)=\bigl(\delta(T^*)\bigr)^*= \Bigr(\delta
\bigl(P_{\mathcal H(\Delta)}S_E|_{\mathcal
H(\Delta)}\bigr)\Bigr)^*=0,
$$
which gives $T\in C_0$. \ This prove the implication
(b)$\Rightarrow$(a).

The second assertion follows from the first together with Corollary
\ref{thmthjdkfjkf} and Corollary \ref{cor512,222}.
\end{proof}

%
%
%
%
%
%
%

\chapter{A canonical decomposition of strong $L^2$-functions} \

In this chapter, we establish a canonical decomposition of strong
$L^2$-functions. \ To better understand this canonical decomposition, we first consider an example of a
matrix-valued $L^2$-function that does not admit a
Douglas-Shapiro-Shields factorization. \ Suppose that $\theta_1$ and
$\theta_2$ are coprime inner functions. \ Consider
$$
\Phi:=\begin{bmatrix}\theta_1&0&0\\
0&\theta_2&0\\0&0&a\end{bmatrix}\equiv [\phi_1, \phi_2, \phi_3] \in
H^{\infty}_{M_3},
$$
where $a\in H^\infty$ is such that $\overline{a}$ is not of bounded
type. \ Then a direct calculation shows that
$$
\hbox{ker}\, H_{\Phi^*}=
\begin{bmatrix}\theta_1&0\\0&\theta_2\\0&0\end{bmatrix}H^2_{\mathbb C^2}
\equiv \Delta H^2_{\mathbb C^2}.
$$
Since $\Delta$ is not two-sided inner, it follows from Lemma
\ref{rem2.4} that $\Phi$ does not admit a Douglas-Shapiro-Shields
factorization. \ For a decomposition of $\Phi$, suppose that
\begin{equation}\label{3001}
\Phi=\Omega A^*,
\end{equation}
where $\Omega, A \in H^2_{M_{3\times k}} (k=1,2)$, $\Omega$ is an
inner function, and $\Omega$ and $A$ are right coprime. We then have
\begin{equation}\label{3002}
\Phi^* \Omega=A \in H^2_{M_{3\times k}}.
\end{equation}
But since $\overline{a}$ is not of bounded type, it follows from
(\ref{3002}) that the 3rd row vector of $\Omega$ is zero. \ Thus by
(\ref{3001}), we must have $a=0$, a contradiction. Therefore we
could not get any decomposition of the form $\Phi=\Omega A^*$ with a
$3\times k$ inner matrix function $\Omega$ for each $k=1,2,3$. \ To
get another idea, we note that $\hbox{ker}\, \Delta^*=[0 \ 0\ 1]^t
H^2\equiv\Delta_c H^2$. \ Then by a direct manipulation, we can get
\begin{equation}
\Phi=\begin{bmatrix}\theta_1&0&0\\
0& \theta_2&0\\0&0&a\end{bmatrix} =
\begin{bmatrix}\theta_1&0\\0&\theta_2\\0&0\end{bmatrix}
\begin{bmatrix}1&0\\0&1\\0&0\end{bmatrix}^*
   +\begin{bmatrix}0\\0\\1\end{bmatrix}\begin{bmatrix}0&0&a\end{bmatrix}
\equiv\Delta A^*+\Delta_c C\,,\label{7.ex}
\end{equation}
where $\Delta$ and $A$ are right coprime because $\widetilde\Delta
H^2_{\mathbb C^3} \bigvee \widetilde A H^2_{\mathbb C^3}
=H^2_{\mathbb C^2}$.

To encounter another situation, consider
$$
\Phi:=\begin{bmatrix}f&f&0\\g&g&0\\0&0&\theta
\overline{a}\end{bmatrix}\equiv [\phi_1, \phi_2, \phi_3] \in
H^{\infty}_{M_3},
$$
where $f$ and $g$ are given in Example \ref{ex3.6},  $\theta$ is
inner, and $a\in H^{\infty}$ is such that $\theta$ and $a$ are
coprime. \ It then follows from Lemma \ref{thm2.9} that
$$
\ker H_{[\overline{f} \ \overline{g}]} =\begin{bmatrix} f\\g
\end{bmatrix} H^2.
$$
We thus have that
$$
\ker H_{\Phi^ *} = \ker H_{[\overline{f} \ \overline{g}]} \bigoplus
   \ker H_{\overline\theta a}
=\begin{bmatrix}f&0\\g&0\\0&\theta\end{bmatrix}H^2_{\mathbb C^2}
\equiv \Delta H^2_{\mathbb C^2}.
$$
Thus by Lemma \ref{rem2.4}, $\Phi$ does not admit a
Douglas-Shapiro-Shields factorization. \ Observe that
\begin{equation}\label{7.exm}
\Phi=\begin{bmatrix}f&f&0\\g&g&0\\0&0&\theta
\overline{a}\end{bmatrix} =
\begin{bmatrix}f&0\\g&0\\0&\theta\end{bmatrix}
\begin{bmatrix}1&0\\1&0\\0&a\end{bmatrix}^*=\Delta
A^*.
\end{equation}
Since $\widetilde{\theta}$ and $\widetilde{a}$ are coprime, it
follows that $\Delta$ and $A$ are right coprime. \ Note that
$\Delta$ is not two-sided inner and $\hbox{ker}\, \Delta^*=\{0\}$.
\bigskip

The above examples (\ref{7.ex}) and (\ref{7.exm}) seem to signal that the decomposition of a matrix-valued $H^2$-functions
$\Phi$ satisfying $\ker H_{\breve{\Phi}}^*=\Delta H^2_{\mathbb C^n}$
may be affected by the kernel of $\Delta^*$ and in turn, the
complementary factor $\Delta_c$ of $\Delta$. \ Indeed, if we regard
$\Delta^*$ as an operator acting from $L^2_E$, and hence $\ker
\Delta^*\subseteq L^2_E$, then $B$ in the canonical decomposition
(\ref{canonical}) satisfies the inclusion $\{B\}\subseteq \ker
\Delta^*$. The following theorem gives a canonical decomposition of
strong $L^2$-functions which realizes the idea inside those
examples.

\bigskip

We are ready for:
\medskip

\begin{thm}\label{vectormaintheorem} (A canonical
decomposition of strong $L^2$-functions) If $\Phi$ is a strong
$L^2$-function with values in $\mathcal B(D, E)$, then $\Phi$ can be
expressed in the form
\begin{equation}\label{canonical}
\Phi=\Delta A^*+B,
\end{equation}
where
\begin{itemize}
\item[(i)] $\Delta$ is an inner function with values
              in $\mathcal B(E^{\prime}, E)$,
              $\widetilde{A}\in H^2_s(\mathcal B(D,E^\prime))$, and
               $B\in L^2_s(\mathcal B(D,E))$;
\item[(ii)] $\Delta$ and $A$ are right coprime;
\item[(iii)] $\Delta^*B =0$;
\item[(iv)] $\hbox{\rm nc}\{\Phi_+\}\le \hbox{\rm dim}\, E^\prime$.
\end{itemize}
\smallskip\noindent
In particular, if $\dim E^\prime<\infty$ {\rm (}for instamce,
$\dim E<\infty${\rm )}, then the expression {\rm (\ref{canonical})}
is unique {\rm (}up to a unitary constant right factor{\rm )}.
\end{thm}

\begin{proof}\label{pfa} \ If $\ker
H_{\breve{\Phi}}^*=\{0\}$, take $E^{\prime}:=\{0\}$ and $B:=\Phi$. \
Then $\widetilde{\Delta}$ and $\widetilde{A}$ are zero operator with
codomain $\{0\}$. \ Thus $\Phi=\Delta A^*+B$, where $\Delta$ and $A$
are right coprime. \ It also follows from Theorem \ref{thm7566} that
$\hbox{nc}\{\Phi_+\}=0$, which gives the inequality (iv). \ If
instead $\ker H_{\breve{\Phi}}^* \ne\{0\}$, then in view of Lemma
\ref{kerhadjoint}, we may suppose $\ker H_{\breve{\Phi}}^*=\Delta
H^2_{E^{\prime}}$ for some nonzero inner function $\Delta$ with
values in $\mathcal B(E^{\prime}, E)$. \ Put $A:=\Phi^*\Delta$. Then
it follows from  Lemma \ref{lem3.4ed} that $A^*$ is a strong
$L^2$-function with values in $\mathcal B(D, E^\prime)$. \ Thus
$\widetilde{A}=\breve A^*$ is a strong $L^{2}$-function with values
in $\mathcal B(D, E^{\prime})$. \ Since $\ker
H_{\breve{\Phi}}^*=\Delta H^2_{E^{\prime}}$, it follows that for all
$p\in \mathcal P_D$ and $h\in H^2_{E^{\prime}}$
$$
\aligned 0&=\langle H_{\breve{\Phi}}p, \ \Delta h\rangle_{L^2_E}\\
&=\int_{\mathbb T}\bigl\langle \breve{\Phi}(z)p(z), \
\overline{z}\Delta(\overline{z})h(\overline{z})\bigr\rangle_E dm(z)\\
&=\int_{\mathbb T}\bigl\langle
\widetilde{\Delta}(z)\breve{\Phi}(z)p(z), \
\overline{z} {h}(\overline z)\rangle_{E^{\prime}}dm(z)\\
&=\bigl\langle H_{\widetilde{A}}p, \
h\bigr\rangle_{L^2_{E^{\prime}}},
\endaligned
$$
which implies $H_{\widetilde{A}}=0$. Thus by Lemma
\ref{boundedhankel}, $\widetilde{A}$ belongs to $H^2_s(\mathcal
B(D,E))$. \ Put $B := \Phi-\Delta A^*$. \ Then by Lemma
\ref{lem3.4ed}, $B$ is a strong $L^2$-function with values in
$\mathcal B(D, E)$. \ Observe that
$$
\Phi=\Delta A^*+B\quad\hbox{and}\quad\Delta^* B=0.
$$
For the first assertion, we need to show that $\Delta$ and $A$ are
right coprime. \ To see this, we suppose that $\Omega$ is a common
left inner divisor, with values in $\mathcal B(E^{\prime\prime},
E^{\prime})$, of $\widetilde{\Delta}$ and $\widetilde{A}$. \ Then we
may write
$$
\widetilde{\Delta}=\Omega\widetilde{\Delta}_1 \quad \hbox{and} \quad
\widetilde{A}=\Omega\widetilde{A}_1,
$$
where $\widetilde\Delta_1\in H^\infty(\mathcal
B(E,E^{\prime\prime}))$ and $\widetilde A_1\in H^2_s(\mathcal B(D,
E^{\prime\prime}))$. Thus we have
\begin{equation}\label{AAD3333}
\Delta=\Delta_1\widetilde{\Omega} \quad \hbox{and} \quad
A=A_1\widetilde{\Omega}.
\end{equation}
Since $\Omega$ is inner, it follows that $\Delta_1=\Delta
\widetilde\Omega^*$, and hence, by Lemma \ref{corfgghh2.9},
$\Delta_1$ is inner. \  We now claim that
\begin{equation}\label{important}
\Delta_1H^2_{E^{\prime\prime}} = \hbox{ker}\,
H_{\breve{\Phi}}^*=\Delta H^2_{E^{\prime}}.
\end{equation}
Since $\Omega$ is an inner function with values in $\mathcal
B(E^{\prime\prime}, E^{\prime})$, we know that
$\widetilde{\Omega}\in H^\infty(\mathcal B(E^\prime,
E^{\prime\prime}))$ by Lemma \ref{corfgghh2.9}. \ Thus it follows
from Corollary \ref{lemma4.100000} and (\ref{AAD3333}) that
$$
\Delta H^2_{E^{\prime}}=\Delta_1\widetilde{\Omega}
H^2_{E^{\prime}}\subseteq \Delta_1H^2_{E^{\prime\prime}}.
$$
For the reverse inclusion, by (\ref{AAD3333}), we may write $
\Phi=\Delta_1 A_1^*+B$. \ Since $0=\Delta^*B
=\widetilde{\Omega}^*\Delta_1^*B $, it follows that $\Delta_1^*B
=0$. \ Therefore for all $f \in H^2_{E^{\prime\prime}}$, $x \in D$
and $n=1,2,\cdots$, we have
$$
\aligned \int_{\mathbb T}\bigl \langle \Phi(z)x, \ z^n
\Delta_1(z)f(z)\bigr\rangle_{E}dm(z)&= \int_{\mathbb T}\bigl \langle
\bigl(\Delta_1(z)A_1^*(z)+B (z)\bigr)x, \
z^n \Delta_1(z)f(z)\bigr\rangle_{E}dm(z)\\
&=\int_{\mathbb T}\bigl \langle A_1^*(z)x, \ z^n
f(z)\bigr\rangle_{E^{\prime\prime}}dm(z)\\
&=\bigl \langle A_1^*(z)x, \ z^n
f(z)\bigr\rangle_{L^2_{E^{\prime\prime}}}\\
&=0,
\endaligned
$$
where the last equality follows from the fact that
$A_1^*(z)x=\widetilde{A}_1(\overline{z})x\in
L^2_{E^{\prime\prime}}\ominus z H^2_{E^{\prime\prime}}$. \ Thus by
Lemma \ref{thm4.2}, we have
$$
\Delta_1H^2_{E^{\prime\prime}} \subseteq \hbox{ker}\,
H_{\breve{\Phi}}^* =\Delta H^2_{E^{\prime}},
$$
which proves (\ref{important}). \ Thus it follows from the
Beurling-Lax-Halmos Theorem and (\ref{AAD3333}) that
$\widetilde{\Omega}$ is a unitary operator, and so is $\Omega$. \
Therefore $A$ and $\Delta$ are right coprime. \ The assertion (iv)
on the nc-number comes from Theorem \ref{thm7566}. \ This proves the
first assertion (\ref{canonical}).

Suppose $\dim E^\prime<\infty$. \ For the uniqueness of the
expression (\ref{canonical}), we suppose that $\Phi=\Delta_1
A_1^*+B_1=\Delta_2 A_2^*+B_2$ are two canonical decompositions of
$\Phi$. \ We want to show that $\Delta_1=\Delta_2$, which gives
$$
A_1^*=\Delta_1^*(\Delta_1 A_1^*+B_1)=\Delta_2^*(\Delta_2 A_2^*+B_2)
=A_2^*
$$
and in turn, $B_1=B_2$, which implies that the representation
(\ref{canonical}) is unique. \ To prove $\Delta_1=\Delta_2$, it
suffices to show that if $\Phi=\Delta A^*+B$ is a canonical
decomposition of $\Phi$, then
\begin{equation}\label{888}
\ker H_{\breve\Phi}^*=\Delta H^2_{E^{\prime}}.
\end{equation}
If $E^{\prime}=\{0\}$, then $\hbox{\rm nc}\{\Phi_+\}=0$. \ Thus it
follows from Corollary \ref{corfkgkhgkhkhk} that
$$
\ker H_{\breve\Phi}^*=\{0\}=\Delta H^2_{E^{\prime}},
$$
which proves (\ref{888}). If instead $E^{\prime}\ne \{0\}$, then we
suppose $r:=\dim E^\prime <\infty$. \ Thus, we may assume that
$E^\prime = \mathbb C^r$, so that $\Delta$ is an inner function
with values in $\mathcal B(\mathbb C^r, E)$. \ Suppose that
$\Phi=\Delta A^*+B$ is a canonical decomposition of $\Phi$ in
$L^2_s(\mathcal B(D, E))$. \ We first claim that
\begin{equation}\label{qqq3}
\Delta H^2_{\mathbb C^r} \subseteq \hbox{ker}\, H_{\breve{\Phi}}^*.
\end{equation}
Observe that for each $g \in H^2_{\mathbb C^r}$, $x \in D$ and
$k=1,2,3, \cdots,$
$$
\aligned \int_{\mathbb T}\bigl \langle \Phi(z)x, \ z^k
\Delta(z)g(z)\bigr\rangle_{E}dm(z) &=\int_{\mathbb T}\bigl \langle
A^*(z)x, \ z^k g(z)\bigr\rangle_{\mathbb
C^r}dm(z)\\
&=\bigl \langle \widetilde{A}(\overline{z})x, \ z^k
g(z)\bigr\rangle_{L^2_{\mathbb
C^r}}\\
&=0.
\endaligned
$$
It thus follows from Lemma \ref{thm4.2} that $\Delta H^2_{\mathbb
C^r} \subseteq \hbox{ker}\, H_{\breve{\Phi}}^*$, which proves
(\ref{qqq3}). \ In view of Lemma \ref{kerhadjoint}, we may assume
that $\ker H_{\breve{\Phi}}^* = \Theta H^2_{E^{\prime\prime}}$ for
some inner function $\Theta$ with values in $\mathcal
B(E^{\prime\prime}, E)$. \ Then by Theorem \ref{thm7566},
\begin{equation}\label{34001}
p\equiv \dim E^{\prime\prime}=\hbox{nc}\,\{\Phi_+\}\leq r.
\end{equation}
Thus we may assume $E^{\prime\prime}\equiv \mathbb C^p$. \ Since
\begin{equation}\label{nfbkg9}
\Delta H^2_{\mathbb C^r} \subseteq \hbox{ker}\,
H_{\breve{\Phi}}^*=\Theta H^2_{\mathbb C^p},
\end{equation}
it follows that $\Theta$ is left inner divisor of $\Delta$, i.e.,
there exists a $p\times r$ inner matrix function $\Delta_1$ such
that $ \Delta=\Theta \Delta_1$. \ Since $\Delta_1$ is inner, it
follows that $r\leq p$. \ But since by (\ref{34001}), $p\le r$, we
must have $r=p$, which implies that $\Delta_1$ is two-sided inner. \
Thus we have
\begin{equation}\label{gkgkgkkgkgkgkkg}
\Theta^*\Phi =\Delta_1A^*+\Delta_1\Delta^*B=\Delta_1A^*.
\end{equation}
Since $\hbox{ker}\, H_{\breve{\Phi}}^*=\Theta H^2_{\mathbb C^r}$, it
follows from Lemma \ref{thm4.2} and (\ref{gkgkgkkgkgkgkkg}) that for
all $f \in H^2_{\mathbb C^r}$, $x \in D$ and $n=1,2,\cdots$,
\begin{equation}\label{kdjfffffffj}
\int_{\mathbb T}\bigl \langle \Delta_1(z) A^*(z)x, \ z^n
f(z)\bigr\rangle_{\mathbb C^r}dm(z)=\int_{\mathbb T}\bigl \langle
\Phi(z)x, \ z^n \Theta(z)f(z)\bigr\rangle_{E}dm(z) =0.
\end{equation}
Write $\Psi :=\Delta_1 A^*$. \ Then by Lemma \ref{lem3.4ed},
$\Psi\in L^2_s(\mathcal B(D, \mathbb C^r))$. \ Thus by Lemma
\ref{boundedhankel}, Lemma \ref{thm4.2} and (\ref{kdjfffffffj}), we
have $\breve{\Psi}\in H^2_s(\mathcal B(D,\mathbb C^r))$. \ Since
$\widetilde{A}=\widetilde{\Delta}_1 \breve{\Psi}$, it follows that
$\widetilde{\Delta}_1$ is a common left inner divisor of
$\widetilde{\Delta}$ and $\widetilde{A}$. \ But since $\Delta$ and
$A$ are right coprime, it follows that $\widetilde{\Delta}_1$ is a
unitary matrix, and so is $\Delta_1$,
which proves (\ref{888}). \ This proves the uniqueness of the
expression (\ref{canonical}) when $\dim E^\prime<\infty$.

This completes the proof.
\end{proof}

\bigskip

The proof of Theorem \ref{vectormaintheorem} shows that the inner
function $\Delta$ in a canonical decomposition (\ref{canonical}) of
a strong $L^2$-function $\Phi$ can be obtained from equation
$$
\ker H_{\breve{\Phi}}^*=\Delta H^2_{E^{\prime}}
$$
which is guaranteed by the Beurling-Lax-Halmos Theorem (see
Corollary \ref{kerhadjoint}). \ In this case, the expression
(\ref{canonical}) will be called the {\it BLH-canonical
decomposition} of $\Phi$ in the viewpoint that $\Delta$ comes from
the Beurling-Lax-Halmos Theorem. However, if $\dim E^\prime=\infty$
(even though $\dim D<\infty$), then it is possible to get another
inner function $\Theta$ of a canonical decomposition
(\ref{canonical}) for the same function: in this case, $\ker
H_{\breve{\Phi}}^* \ne \Theta H^2_{E^{\prime\prime}}$. Indeed, the
following remark shows that the canonical decomposition
(\ref{canonical}) is not unique in general.

\bigskip

\begin{rem}\label{canoexam}
If $\dim E^\prime=\infty$ (even though $\dim D <\infty$), the
canonical decomposition (\ref{canonical}) may not be unique even if
$\breve\Phi$ is of bounded type. \ To see this, let $\Phi$ be an
inner function with values in $\mathcal B(\mathbb C^2, \ell^2)$
defined by
$$
\Phi:=\begin{bmatrix}\theta_1&0\\
0&0\\
0&\theta_2\\
0&0\\
0&0\\
0&0\\
\vdots&\vdots
\end{bmatrix},
$$
where $\theta_1$ and $\theta_2$ are scalar inner functions. \ Then
$$
\hbox{ker}\, H_{\breve{\Phi}}^*=\hbox{ker}\, H_{\Phi^*}
=\hbox{diag}(\theta_1,1,\theta_2,1, 1,1,\cdots)H^2_{\ell^2} \equiv
\Theta H^2_{\ell^2},
$$
which implies that $\breve{\Phi}$ is of bounded type since $\Theta$
is two-sided inner (see Definition \ref{dfbundd}). \ Let
$$
A:=\Phi^*\Theta=
\begin{bmatrix}
1&0&0&0&0&0&0&\cdots\\
0&0&1&0&0&0&0&\cdots
\end{bmatrix}\quad\hbox{and}\quad B:=0.
\
$$
Then $\widetilde A$ belongs to belongs to $H^2_s (\mathcal B(\mathbb
C^2, \ell^2))$ and $\widetilde\Theta H^2_{\ell^2} \bigvee \widetilde
A H^2_{\mathbb C^2} =H^2_{\ell^2}$, which implies that $\Theta$ and
$A$ are right coprime. \ Clearly, $\Theta^*B=0$ and
 $\hbox{nc}\{\Phi_+\}\le \dim \ell^2=\infty$.
Therefore,
$$
\Phi=\Theta A^*
$$
is the BLH-canonical decomposition of $\Phi$. \ On the other hand,
to get another canonical decomposition of $\Phi$, let
$$
\Delta:=
\begin{bmatrix}
\theta_1&0&0&0&0&0&0&\cdots\\
0&0&0&0&0&0&0&\cdots\\
0&\theta_2&0&0&0&0&0&\cdots\\
0&0&1&0&0&0&0&\cdots\\
0&0&0&1&0&0&0&\cdots\\
0&0&0&0&1&0&0&\cdots\\
\vdots&\vdots&\vdots&\vdots&\vdots&\vdots&\vdots&\ddots
\end{bmatrix}.
$$
Then $\Delta$ is an inner function. \ If we define
$$
A_1:=
\begin{bmatrix}
1&0&0&0&0&0&0&\cdots\\
0&1&0&0&0&0&0&\cdots
\end{bmatrix}\quad\hbox{and}\quad B:=0\,,
$$
then $\widetilde{A}_1$ belongs to $H^2_s \mathcal (B(\mathbb C^2,
\ell^2))$ such that $\Delta$ and $A_1$ are right coprime,
$\Delta^*B=0$ and $\hbox{nc}\{\Phi_+\}\leq \dim \ell^2=\infty$. \
Therefore $\Phi=\Delta A_1^*$ is also a canonical decomposition of
$\Phi$. \ In this case, $\hbox{ker}\, H_{\breve{\Phi}}^* \ne \Delta
H^2_{\ell^2}$. \ Therefore, the canonical decomposition of $\Phi$ is
not unique.
\end{rem}

\bigskip

\begin{rem} Let $\Delta$ be an inner matrix function with values in
$\mathcal B(E^{\prime}, E)$. \ Then Theorem \ref{vectormaintheorem}
says that if $\dim E^{\prime}<\infty$, the expression
(\ref{canonical}) satisfying the conditions (i) - (iv) in Theorem
\ref{vectormaintheorem} gives $\hbox{ker}\, H_{\breve{\Phi}}^* =
\Delta H^2_{E^\prime}$. \ We note that the condition (iv) on
nc-number cannot be dropped from the assumptions of Theorem
\ref{vectormaintheorem}. \ To see this, let
$$
\Delta:=\frac{1}{\sqrt{2}}\begin{bmatrix}z\\1\end{bmatrix},  \quad
A:=\begin{bmatrix}\sqrt{2}\\0\end{bmatrix} \hbox{and} \quad B:=0.
$$
If
$$
\Phi:=\Delta A^*+B =\begin{bmatrix} z&0\\1&0\end{bmatrix},
$$
then $\Phi$ satisfies the conditions (i), (ii), and (iii), but $\ker
H_{\breve \Phi}^*=zH^2\oplus H^2 \neq \Delta H^2$. \ Note that by
Theorem \ref{thm7566}, $\hbox{nc}\{\Phi_+\}=2$, which does not
satisfy the condition on nc-number, say $\hbox{nc}\{\Phi_+\}\le 1$.
\end{rem}

\bigskip

\begin{cor}\label{remarksdscjkdfv}
If $\breve\Delta$ is of bounded type then $B$ in {\rm
(\ref{canonical})} is given by
$$
B=\Delta_c\Delta_c^*\Phi,
$$
where $\Delta_c$ is the complementary factor of $\Delta$, with
values in $\mathcal B(D^{\prime},E)$. Moreover, if $\dim
E^\prime<\infty$, then $\dim D^{\prime}$ can be computed by the
formula
$$
\dim D^{\prime}=\hbox{nc}\{\Delta\}-\hbox{nc}\{\Phi_+\}.
$$
\end{cor}

\begin{proof}
Suppose that $\breve\Delta$ is of bounded type. \ Then by Corollary
\ref{thmthjdkfjkf}, $[\Delta, \Delta_c]$ is two-sided inner, where
$\Delta_c$ is the complementary factor of $\Delta$, with values in
$\mathcal B(D^{\prime},E)$. \ We thus have
$$
I=[\Delta, \Delta_c] [\Delta,
\Delta_c]^*=\Delta\Delta^*+\Delta_c\Delta_c^*,
$$
so that
$$
B=\Phi-\Delta A^*=(I-\Delta \Delta^*)\Phi=\Delta_c\Delta_c^*\Phi.
$$
This proves the first assertion. \ The second assertion follows at
once from the facts that $\hbox{nc}\{\Phi_+\}= \dim E^\prime<\infty$
(by Theorem \ref{thm7566}) and $\hbox{nc}\{\Delta\}= \dim E^\prime +
\dim D^{\prime}$ (by Corollary \ref{innerdc}).
\end{proof}

\bigskip

The following corollary is an extension of Lemma \ref{rem2.4} (the
Douglas-Shapiro-Shields factorization) to strong $L^2$-functions.


\begin{cor}\label{ext24}
If $\Phi$ is a strong $L^2$-function with values in $\mathcal
B(D,E)$, then the following are equivalent:
\begin{itemize}
\item[(a)] The flip $\breve\Phi$ of $\Phi$ is of bounded type;
\item[(b)] $\Phi=\Delta A^*$ ($\Delta$ is two-sided inner) is  a canonical
decomposition of $\Phi$.
\end{itemize}
\end{cor}

\begin{proof} The implication (a)$\Rightarrow$(b) follows from
the proof of Theorem \ref{vectormaintheorem}. \ For the implication
(b)$\Rightarrow$(a), suppose $\Phi=\Delta A^*$ ($\Delta$ is
two-sided inner) is a canonical decomposition of $\Phi$. \ By Lemma
\ref{kerhadjoint}, there exists an inner function $\Theta$ with
values in $\mathcal B(D^{\prime}, E)$ such that $\hbox{ker}\,
H_{\breve\Phi}^*=\Theta H^2_{D^{\prime}}$. \ Then it follows from
Lemma \ref{thm4.2} that $ \Delta H^2_E \subseteq \ker
H_{\breve{\Phi}}^*=\Theta H^2_{D^{\prime}}$. \  Since $\Delta$ is
two-sided inner, we have that by Lemma \ref{rem.sdcfcfc}, $\Theta$
is two-sided inner, and hence the flip $\breve{\Phi}$ of $\Phi$ is
of bounded type. This completes the proof.
\end{proof}

\bigskip

If $\Delta$ is an inner matrix function such that
$\Delta\Delta^*\Phi$ is analytic (even though $\breve\Delta$ is not
of bounded type) then the perturbation part $B$ of the canonical
decomposition may be also determined in terms of the complementary
factor of $\Delta$.

\medskip

\begin{cor}\label{hjnfdcvwsa}
Let $\Phi$ be an $n\times m$ matrix-valued $H^2$-function. Then the
following are equivalent:
\begin{itemize}
\item[(a)] $\ker H_{\breve{\Phi}}^*=\Delta H^2_{\mathbb C^r}$ for an
$n\times r$ inner matrix function $\Delta$ such that
$\Delta\Delta^*\Phi$ is analytic;
\item[(b)] $\Phi=\Delta A^*+\Delta_c \Delta_c^* \Phi$ is a canonical decomposition of $\Phi$, where
$\Delta_c$ is the complementary factor of $\Delta$.
\end{itemize}
\end{cor}

\begin{proof} (a)$\Rightarrow$(b): Suppose that  $\ker H_{\breve{\Phi}}^*=\Delta H^2_{\mathbb C^r}$ for an
$n\times r$ inner matrix function $\Delta$ such that
$\Delta\Delta^*\Phi$ is analytic. \ Then by the proof of Theorem
\ref{vectormaintheorem}, we can write
$$
\Phi=\Delta A^*+B,
$$
where $B=(I-\Delta\Delta^*)\Phi$.  \ Write $\Phi\equiv [\phi_1,
\phi_1, \cdots, \phi_m]$. \ Since $\Delta\Delta^*\Phi \in
H^{2}_{M_{n \times m}}$ and $\Delta^*(I-\Delta \Delta^*)=0$, it
follows from Corollary \ref{lemma4.100000} and Lemma \ref{thm2.9}
that for each $j=1,2,\cdots m$,
$$
(I-\Delta \Delta^*)\phi_j \in \hbox{ker}\, \Delta^* =\Delta_c
H^2_{\mathbb C^p},
$$
which implies that $B=(I-\Delta\Delta^*)\Phi = \Delta_c D$ for some
$D \in H^2_{M_{p\times m}}$. \ Thus
$$
\Delta_c^*B=\Delta_c^*(I-\Delta\Delta^*)\Phi = D,
$$
so that
$$
B=\Delta_c D=\Delta_c
\Delta_c^*(I-\Delta\Delta^*)\Phi=\Delta_c\Delta_c^*\Phi.
$$

(b)$\Rightarrow$(a): Suppose that $\Phi=\Delta A^*+\Delta_c
\Delta_c^* \Phi$ is a canonical decomposition of $\Phi$. \ Since
$\Phi$ is a matrix-valued function, it follows from Theorem
\ref{vectormaintheorem} that
$$
\Delta_c \Delta_c^* \Phi=B=(I-\Delta\Delta^*)\Phi,
$$
so that
$$
\Phi=\Delta_c \Delta_c^* \Phi+\Delta\Delta^*\Phi.
$$
But since $\langle \Delta_c \Delta_c^* \phi_j, \ \Delta\Delta^*\phi_j
\rangle=0$ for all $j=1,2,\cdots, m$, it follows that
$\Delta\Delta^*\Phi \in H^2_{M_{n \times m}}$. \ This completes the
proof.
\end{proof}

\medskip

\begin{cor}
Let $\Phi$ be an $n\times m$ matrix-valued $H^2$-function satisfying
$\ker H_{\breve{\Phi}}^*=\Delta H^2_{\mathbb C^r}$ for an $n\times
r$ inner matrix function $\Delta$ such that $\Delta\Delta^*$ is
analytic. Then $\Phi$ can be written as
\begin{equation}\label{cd}
\Phi=\Delta A^*+\Delta_c C \quad(\textrm{with } C:=P_+\Delta_c^*\Phi\in
H^2_{M_{p\times m}}),
\end{equation}
where $\Delta_c$ is the complementary factor of $\Delta$.
\end{cor}

\begin{proof}We  claim that if
$\Delta\Delta^*$ is analytic, then
\begin{equation}\label{751}
\left(I-\Delta\Delta^*\right)H^2_{\mathbb C^n} =\Delta_cH^2_{\mathbb
C^p}.
\end{equation}
To see this, let $f \in \Delta_c H^2_{\mathbb C^p}$.  Then
$f=\Delta_c g$ for some $g \in H^2_{\mathbb C^p}$. Observe that
$$
(I-\Delta \Delta^*)f=(I-\Delta \Delta^*)\Delta_c g=\Delta_c g=f,
$$
which implies that $f \in(I-\Delta \Delta^*)H^2_{\mathbb C^n}$. Thus
we have $\Delta_c H^2_{\mathbb C^p}\subseteq (I-\Delta
\Delta^*)H^2_{\mathbb C^n}$. \ The converse inclusion follows from
the proof of Corollary \ref{hjnfdcvwsa}. This proves (\ref{751}). \
Thus $I-\Delta \Delta^*$ is the orthogonal projection that maps from
$H^2_{\mathbb C^n}$ onto $\Delta_c H^2_{\mathbb C^p}$. Therefore by
the Projection Lemma in \cite[P.~43]{Ni1}, we have
$$
\left(I-\Delta\Delta^*\right)\vert_{H^2_{\mathbb C^n}}=
\Delta_cP_+\Delta_c^*,
$$
so that
$$
\Phi=\Delta A^*+B=\Delta A^*+ \Delta_c P_+\Delta_c^*\Phi,
$$
as desired
\end{proof}

%
%
%
%
%

\chapter{The Beurling degree} \

We first consider Question \ref{q333}. Question \ref{q333} can be
rephrased as: {\it If $\Delta$ is an inner function with values in
$\mathcal B(E^{\prime},E)$, does there exist a strong $L^2$-function
$\Phi$ with values in $\mathcal B(D,E)$ satisfying the equation}
\begin{equation}\label{q4}
\ker H_{\breve\Phi}^*=\Delta H^2_{E^\prime}\, ?
\end{equation}
To closely understand an answer to Question \ref{q333}, we examine a
question whether there exists an inner function $\Omega$
satisfying $\hbox{ker}\, H_{\Omega^*}=\Delta H^2_{E^\prime}$ if
$\Delta$ is an inner function with values in $\mathcal B(E^\prime,
E)$. \ In fact, the answer to this question is negative. \ Indeed,
if $\hbox{ker}\, H_{\Omega^*}=\Delta H^2_{E^{\prime}}$ for some
inner function $\Omega \in H^{\infty}(\mathcal B(D, E))$, then by
Lemma \ref{thm2.9}, we have $[\Omega, \Omega_c]=\Delta$, and hence
$\Delta_c=0$. \ Conversely, if $\Delta_c=0$ then by again Lemma
\ref{thm2.9}, we should have $\hbox{ker}\, H_{\Delta^*}=\Delta
H^2_{E^\prime}$. \ Consequently, $\hbox{ker}\, H_{\Omega^*}=\Delta
H^2_{E^{\prime}}$ for some inner function $\Omega$ if and only if
$\Delta_c=0$. \ Thus if
$$
\Delta:=\begin{bmatrix} 1\\ 0\end{bmatrix},
$$
then there exists no inner function $\Omega$ such that $\hbox{ker}\,
H_{\Omega^*}=\Delta H^2$. \ On the other hand, we note that the
solution $\Phi$ is not unique although there exists an inner
function $\Phi$ satisfying the equation (\ref{q4}). \ For example,
if $\Delta:=\hbox{\rm diag}\,(z,1,1)$, then the following $\Phi$ are
such solutions:
$$
\Phi=\begin{bmatrix}z\\0\\0\end{bmatrix}, \ \
\begin{bmatrix}z&0\\0&1\\0&0\end{bmatrix}, \ \ \Delta.
$$

\bigskip

The following theorem gives an affirmative answer to Question
\ref{q333}: indeed, we can always find a strong $L^2$-function
$\Phi$ with values in $\mathcal B(D,E)$ satisfying the equation
$\ker H_{\breve\Phi}^*=\Delta H^2_{E^\prime}$.


\begin{thm} \label{kkknvnv}
Let $\Delta$ be an inner function with values in $\mathcal
B(E^\prime, E)$. \ Then there exists a function $\Phi$ in
$H^2_s(\mathcal B(D,E))$, with either $D=E^\prime$ or $D=\mathbb
C\oplus E^\prime$, satisfying
$$
\hbox{ker}\, H_{\breve{\Phi}}^*=\Delta H^2_{E^\prime}.
$$
\end{thm}

\begin{proof}  If $\hbox{ker}\, \Delta^* = \{0\}$, take
$\Phi=\Delta$. \ Then it follows from Lemma \ref{thm2.9} that
$$
\hbox{ker}\, H_{\breve{\Phi}}^*=\hbox{ker}\, H_{\Delta^*}=\Delta
H^2_{E^\prime}.
$$
If instead $\hbox{ker}\, \Delta^*\neq \{0\}$, let $\Delta_c$ be the
complementary factor of $\Delta$ with values in $\mathcal
B(E^{\prime\prime}, E)$ for some nonzero Hilbert space
$E^{\prime\prime}$. \ Choose a cyclic vector $g \in
H^2_{E^{\prime\prime}}$ of $S^*_{E^{\prime\prime}}$ and define
$$
\Phi:=\bigl[[z\Delta_c g], \Delta\bigr],
$$
where $[z\Delta_c g](z):\mathbb C \rightarrow E$ is given by
$[z\Delta_c g](z)\alpha:=\alpha z \Delta_c(z) g(z)$. \ Then it
follows from Lemma \ref{dhfbgbgbgbg} and Corollary
\ref{lemma4.100000} that $\Phi$ belongs to  $H^2_s(\mathcal
B(D,E))$, where $D=\mathbb C\oplus E^\prime$. \ For each $x\equiv
\alpha \oplus x_0 \in D$, $f\in H^2_{E^\prime}$, and $
n=1,2,3,\cdots$, we have
$$
\aligned \int_{\mathbb  T}\bigl\langle \Phi(z)x, \ z^n
\Delta(z)f(z)\bigr\rangle_{E}dm(z) &=\int_{\mathbb T}\bigl\langle
\alpha z\Delta_c(z) g(z)+\Delta(z)x_0, \ z^n
\Delta(z)f(z)\bigr\rangle_{E}dm(z)\\
&=\int_{\mathbb T}\bigl\langle x_0, \ z^n
f(z)\bigr\rangle_{E^{\prime}}dm(z)\quad
\hbox{(since $\Delta^*\Delta_c=0$)}\\
&=0.
\endaligned
$$
It thus follows from Lemma \ref{thm4.2} that
\begin{equation}\label{ffmvvlgkg}
\Delta H^2_{E^\prime}\subseteq \hbox{ker}\, H_{\breve{\Phi}}^*.
\end{equation}
For the reverse inclusion, suppose $h \in \hbox{ker}\,
H_{\breve{\Phi}}^*$. \ Then by Lemma \ref{thm4.2}, we have that for
each $x_0 \in E^{\prime}$ and $ n=1,2,3,\cdots$,
$$
\int_{\mathbb T}\bigl\langle \Delta(z) x_0, \ z^nh(z) \bigr
\rangle_{E}dm(z)=0,
$$
which implies, by Lemma \ref{thm4.2}, that $h \in \hbox{ker}\,
H_{\Delta^*}$. \ It thus follows from Lemma \ref{thm2.9} that
\begin{equation}\label{ssbbdbfbf}
\hbox{ker}\, H_{\breve{\Phi}}^*\subseteq \hbox{ker}\,
H_{\Delta^*}=\Delta H^2_{E^\prime}\bigoplus \Delta_c
H^2_{E^{\prime\prime}}.
\end{equation}
Assume to the contrary that $\hbox{ker}\, H_{\breve{\Phi}}^* \neq
\Delta H^2_{E^\prime}$. \ Then by (\ref{ffmvvlgkg}) and
(\ref{ssbbdbfbf}), there exists a nonzero function $f\in
H^2_{E^{\prime\prime}}$ such that $\Delta_c f \in \hbox{ker}\,
H_{\breve\Phi}^*$. \ It thus follows from Lemma \ref{thm4.2} that
for each $x\equiv \alpha\oplus x_0 \in D$ and $n=1,2,3,\cdots$,
$$
\begin{aligned}
0 &=\int_{\mathbb T}\bigl\langle \Phi(z)x, \ z^n
\Delta_c(z)f(z)\bigr\rangle_{E}dm(z)\\
&=\int_{\mathbb T}\bigl\langle \alpha z\Delta_c(z)
g(z)+\Delta(z)x_0, \ z^n
\Delta_c(z)f(z)\bigr\rangle_{E}dm(z)\\
&=\int_{\mathbb T}\bigl\langle z[g](z)\alpha, \ z^n
f(z)\bigr\rangle_{E^{\prime\prime}}dm(z)\quad \hbox{(since
$\Delta^*\Delta_c=0$)},
\end{aligned}
$$
which implies that $f \in \hbox{ker}\,
H_{\overline{z}\breve{[g]}}^*$. \ Since $g$ is a cyclic vector of
$S^*_{E^{\prime\prime}}$, it thus follows from Lemma \ref{thm327}
that
$$
f\in \bigl(\hbox{cl ran}\, H_{\overline{z} \breve{[g]}}\bigr)^\perp
=\bigl(E_g^*\bigr)^\perp =\{0\},
$$
which is a contradiction. \ This completes the proof.
\end{proof}

\medskip

If $\Delta$ is an $n \times r$ inner matrix function, then we can
find a solution $\Phi\in H^{\infty}_{M_{n\times m}}$ (with $m \leq
r+1$) of the equation $\hbox{ker}\, H_{\breve{\Phi}}^*=\Delta
H^2_{\mathbb C^r}$.

\medskip

\begin{cor}\label{existrem}
For a given $n\times r$ inner matrix function $\Delta$, there exists
at least a solution $\Phi\in H^{\infty}_{M_{n\times m}}$ (with $m
\leq r+1$) of the equation $\hbox{ker}\, H_{\breve{\Phi}}^*=\Delta
H^2_{\mathbb C^r}$.
\end{cor}

\begin{proof}
If $\hbox{ker}\, \Delta^* = \{0\}$, then this is obvious. Let
$\hbox{ker}\, \Delta^*\neq \{0\}$ and $\Delta_c \in H^{\infty}_{M_{n
\times p}}$ be the complementary factor of $\Delta$. \ Then by Lemma
\ref{thm2.9},  $1\leq p\leq n-r$. For $j=1,2,\cdots, p$, put
$$
g_j:=e^\frac{1}{z-\alpha_j},
$$
where $\alpha_j$ are distinct points in the interval $[2, 3]$. \
Then it is known that (cf. \cite[P.~55]{Ni1})
$$
g:=\begin{bmatrix} g_1\\g_2\\ \vdots\\g_p\end{bmatrix} \in
H^{\infty}_{\mathbb C^p}
$$
is a cyclic vector of $S^*_{\mathbb C^p}$. \ Put $
\Phi:=\bigl[[z\Delta_c g], \Delta\bigr]$. \ Then by Lemma
\ref{dhfbgbgbgbg}, we have $\Phi \in H^{\infty}_{M_{n\times
(r+1)}}$. \ The same argument as the proof of Theorem \ref{kkknvnv}
gives the result.
\end{proof}

\medskip

\begin{cor}\label{cor83}
If $\Delta$ is an inner function with values in $\mathcal
B(E^\prime, E)$, then there exists a function $\Phi\in L^2_s
(\mathcal B(D,E))$ (with $D=E^\prime$ or $D=\mathbb C\oplus
E^\prime$) such that $\Phi\equiv\Delta A^*+B$ is the BLH-canonical
decomposition of $\Phi$.
\end{cor}

\begin{proof}
By Theorem \ref{kkknvnv}, there exists a function $\Phi\in
H^2_s(\mathcal B(D,E))$ such that $\hbox{ker}\,
H_{\breve{\Phi}}^*=\Delta H^2_{E^\prime}$, with $D=E^\prime$ or
$D=\mathbb C\oplus E^\prime$. \ If we put $A:=\Phi^*\Delta$ and
$B:=\Phi-\Delta A^*$, then by the proof of the first assertion of
Theorem \ref{vectormaintheorem}, $\Phi=\Delta A^*+B$ is the
BLH-canonical decomposition of $\Phi$.
\end{proof}

\bigskip

\begin{rem}\label{remark7.9}
In view of Corollary \ref{existrem}, it is reasonable to ask whether
 such a solution $\Phi\in L^{2}_{M_{n\times m}}$
of the equation $\hbox{ker}\, H_{\breve{\Phi}}^*=\Delta H^2_{\mathbb
C^r}$ ($\Delta$ an $n\times r$ inner matrix function) exists for
each $m=1,2,\cdots$ even though it exists for some $m$. For example,
let
\begin{equation}\label{ex444}
\Delta:=\frac{1}{\sqrt{2}}\begin{bmatrix} z\\1\end{bmatrix}.
\end{equation}
Then, by Corollary \ref{existrem}, there exists  a solution $\Phi\in
L^2_{M_{2\times m}}$ ($m=1$ or $2$) of the equation $\hbox{ker}\,
H_{\breve{\Phi}}^*=\Delta H^2$. \ For $m=2$, let
\begin{equation}\label{ex8712}
\Phi: =\begin{bmatrix} z&za\\1&-a\end{bmatrix}\in H^\infty_{M_{2}},
\end{equation}
where $a \in H^{\infty}$ is such that $\overline{a}$ is not of
bounded type. \ Then a direct calculation shows that $\ker H_{\breve
\Phi}^*=\hbox{ker}\, H_{\Phi^*}=\Delta H^2$. \ We may then ask how
about the case $m=1$. In this case, the answer is affirmative. To
see this, let
$$
\Psi:=\begin{bmatrix}z+za\\1-a\end{bmatrix}\in H^\infty_{M_{2\times
1}},
$$
where $a \in H^{\infty}$ is such that $\overline{a}$ is not of
bounded type. \ Then a direct calculation shows that $\hbox{ker}\,
H_{\Psi^*}=\Delta H^2$. \ Therefore, if $\Delta$ is given by
(\ref{ex444}), then we may assert that there exists a solution
$\Phi\in L^2_{M_{n\times m}}$ of the equation $\hbox{ker}\,
H_{\breve{\Phi}}^*=\Delta H^2$ for each $m=1,2$. \ However, this
assertion is not true in general, i.e., a solution exists for some
$m$, but may not exist for another $m_0<m$. \ To see this, let
$$
\Delta:=\begin{bmatrix} z&0&0\\0&z&0\\0&0&1\\0&0&0\end{bmatrix}\in
H^{\infty}_{M_{4\times 3}}.
$$
Then $\Delta$ is inner. We will show that there exists no solution
$\Phi\in L^2_{M_{4\times 1}}$ (i.e., the case $m=1$) of the equation
$\ker H_{\breve\Phi}^*=\Delta H^2_{\mathbb C^3}$. \ Assume to the
contrary that $\Phi \in L^2_{M_{4\times 1}}$ is a solution of the
equation $\ker H_{\breve\Phi}^*=\Delta H^2_{\mathbb C^3}$. \ By
Theorem \ref{vectormaintheorem}, $\Phi$ can be written as
$$
\Phi=\Delta A^*+B,
$$
where $A \in H^2_{M_{1\times 3}}$ is such that $\Delta$ and $A$ are
right coprime. \ But since $\widetilde{\Delta}H^2_{\mathbb
C^4}=zH^2\oplus zH^2\oplus H^2$, it follows that
$$
\widetilde{\Delta}H^2_{\mathbb C^4} \bigvee \widetilde{A}H^2 \neq
H^2_{\mathbb C^3},
$$
which implies that $\Delta$ and $A$ are not right coprime, a
contradiction. \ Therefore we cannot find any solution $\Phi$, in
$L^2_{M_{4\times 1}}$ (the case $m=1$), of the equation $\ker\,
H_{\breve\Phi}^* = \Delta H^2_{\mathbb C^3}$. \ By contrast, if
$m=2$, then we can find a solution $\Phi\in L^2_{M_{4\times 2}}$. \
Indeed, let
$$
\Phi:=\begin{bmatrix} z&0\\ 0&z\\ 0&0\\ a&0\end{bmatrix},
$$
where $a\in H^\infty$ is such that $\overline a$ is not of bounded
type. \
Then $\ker H_{\Phi^*}=zH^2\oplus zH^2\oplus H^2\oplus \{0\} =\Delta
H^2_{\mathbb C^3}$. \ Thus we obtain a solution for $m=2$ although
there exists no solution for $m=1$.
\end{rem}

\bigskip

Let $\Delta$ be an inner function with values in $\mathcal
B(E^\prime, E)$. \ In view of Remark \ref{remark7.9}, we may ask how
to determine a possible dimension of $D$ for which there exists a
solution $\Phi\in L^2_s(\mathcal B(D, E))$ of the equation $\ker
H_{\breve\Phi}^* = \Delta H^2_{E^{\prime}}$. \ In fact, if we have a
solution $\Phi\in L^2_s(\mathcal B( D, E))$  of the equation $\ker
H_{\breve\Phi}^* = \Delta H^2_{E^{\prime}}$, then a solution
$\Psi\in L^2_s(D^\prime, E))$ also exists if $D^\prime$ is a
separable complex Hilbert space containing $D$: indeed, if ${\mathbf
0}$ denotes the zero operator in $\mathcal B(D^\prime \ominus D, E)$
and $\Psi:=[\Phi, {\mathbf 0}]$, then it follows from Lemma
\ref{thm4.2} that $\ker H_{\breve\Phi}^*=\ker H_{\breve{\Psi}}^*$. \
Thus we would like to ask what is the infimum of $\dim D$ such that
there exists a solution $\Phi\in L^2_s(\mathcal B(D, E))$ of the
equation $\ker H_{\breve\Phi}^* = \Delta H^2_{E^{\prime}}$. \ To
answer this question, we introduce a notion of the ``Beurling
degree" for an inner function, by employing a canonical
decomposition of strong $L^2$-functions induced by the given inner
function.

\medskip

\begin{df}\label{df of degree}
Let $\Delta$ be an  inner function with values in $\mathcal
B(E^{\prime}, E)$. \ Then the {\it Beurling degree} of $\Delta$,
denoted by $\hbox{deg}_B (\Delta)$, is defined by
$$
\begin{aligned}
\hbox{\rm deg}_B(\Delta):=\inf &\Bigl\{ \dim D\in \mathbb
Z_+\cup\{\infty\}:
\exists\ \hbox{a pair $(A,B)$ such that }\\
&\hbox{$\Phi=\Delta A^*+B$ is a canonical decomposition of $\Phi\in
L^2_s(B(D,E))$} \Bigr\}
\end{aligned}
$$
\end{df}

\bigskip

\noindent {\it Note}. By Corollary \ref{cor83},
$\hbox{deg}_B(\Delta)$ is well-defined: indeed,
$1\leq\hbox{deg}_{B}(\Delta)\leq 1+\dim E^\prime$. \ In particular,
if $E^{\prime}=\{0\}$, then $\hbox{deg}_{B}(\Delta)=1$. \ Also if
$\Delta$ is a unitary operator then clearly,
$\hbox{deg}_{B}(\Delta)=1$.

\bigskip

We are ready for:

\medskip

\begin{thm}\label{maintheorem_sm} (The Beurling degree and the  spectral
multiplicity) Given an inner function $\Delta$ with values in
$\mathcal B(E^\prime, E)$, with $\dim E^\prime<\infty$, let
$T:=S_{E}^*|_{\mathcal H(\Delta)}$. \ Then
\begin{equation}\label{main_2}
\mu_T=\hbox{\rm deg}_{B}(\Delta).
\end{equation}
\end{thm}

\begin{proof}
Let $T:=S_E^*\vert_{\mathcal H(\Delta)}$. \ We first claim that
\begin{equation}\label{main2k}
\begin{aligned}
\hbox{deg}_B(\Delta)=\inf \bigl\{ \dim D: \ \ker
H_{\breve\Phi}^*=\Delta H^2_{E^\prime}\ \hbox{for some} &\ \Phi\in
L^2_s(\mathcal B(D,E))\\
&\qquad\qquad\hbox{with} \ D\neq \{0\} \bigr\}.
\end{aligned}
\end{equation}
To see this, let $\Delta$ be an inner function with values in
$\mathcal B(E^\prime, E)$, with $\dim E^\prime<\infty$. \ Suppose
that $\Phi=\Delta A^*+B$ is a canonical decomposition of $\Phi$  in
$L^2_s(\mathcal B(D,E))$. \ Then by the uniqueness of $\Delta$ in
Theorem \ref{vectormaintheorem}, we have
\begin{equation}\label{5.22.1}
\ker H_{\breve\Phi}^*=\Delta H^2_{E^{\prime}},
\end{equation}
which implies
\begin{equation}\label{471}
\begin{aligned}
\hbox{deg}_B(\Delta) \geq \inf \bigl\{ \dim D: \ \ker
H_{\breve\Phi}^*=\Delta H^2_{E^\prime}\ \hbox{for some} &\ \Phi\in
L^2_s(\mathcal B(D,E))\\
&\qquad\qquad\hbox{with} \ D\neq \{0\} \bigr\}.
\end{aligned}
\end{equation}
For the reverse inequality of (\ref{471}), suppose $\Phi\in
L^2_s(\mathcal B(D, E))$ satisfies $\ker H_{\breve\Phi}^*=\Delta
H^2_{E^{\prime}}$. \ Then by the same argument as in the proof of
the first assertion of Theorem \ref{vectormaintheorem},
$$
\hbox{$\Phi=\Delta A^*+B$ \ \ ($A:=\Phi^*\Delta$ and $B:=\Phi-\Delta
A^*$)}
$$
is a canonical decomposition of $\Phi$, and hence we have the
reverse inequality of (\ref{471}). \ This proves the claim
(\ref{main2k}). \ We will next show that
\begin{equation}\label{Bbbb}
\hbox{deg}_B(\Delta)\le \mu_T.
\end{equation}
If $\mu_T=\infty$, then (\ref{Bbbb}) is trivial. Suppose
$p\equiv\mu_T<\infty$. \ Then there exists a subset $G=\{g_1,
g_2,\cdots g_p\}\subseteq H^2_{E}$ such that $E_{G}^*=\mathcal
H(\Delta)$. \ Put
$$
\Psi:=z [G].
$$
Then by Lemma \ref{dhfbgbgbgbg}, $\Psi \in H^2_s(\mathcal B(\mathbb
C^p, E))$. \ It thus follows from Lemma \ref{thm327} that
$$
\mathcal H(\Delta)=E_{G}^*=\hbox{cl ran}\,
H_{\overline{z}[\breve{G}]} =\hbox{cl ran}\, H_{\breve{\Psi}},
$$
which implies $\hbox{ker}\,H_{\breve{\Psi}}^*= \Delta
H^2_{E^{\prime}}$. \ Thus by (\ref{main2k}),
$\hbox{deg}_B(\Delta)\le p=\mu_T$, which proves (\ref{Bbbb}). \ For
the reverse inequality of (\ref{Bbbb}), suppose that $r\equiv\dim
E^{\prime}<\infty$, Write $m_0\equiv \hbox{deg}_B(\Delta)$. \ Then
it follows from Theorem \ref{kkknvnv} and (\ref{main2k}) that
$m_0\leq r+1<\infty$ and there exists  a function $\Phi\in
L^2_s(\mathcal B(\mathbb C^{m_0},E))$ such that
\begin{equation}
\hbox{ker}\, H_{\breve{\Phi}}^*=\Delta H^2_{\mathbb C^r}.
\end{equation}
Now let
$$
G:=\Phi_+-\widehat{\Phi}(0).
$$
Thus we may write $G=z F$ for some $F\in H^2_s(\mathcal B(\mathbb
C^{m_0},E))$. \ Then by Lemma \ref{boundedhankel} and Lemma
\ref{thm327}, we have that
$$
E_{\{F\}}^*=\hbox{cl ran}\, H_{\breve{G}}=\bigl(\hbox{ker}\,
H_{\breve{\Phi}}^*\bigr)^{\perp}=\mathcal H(\Delta),
$$
which implies $\mu_T \leq m_0=\hbox{deg}_B(\Delta)$. \ This
completes the proof.
\end{proof}

\bigskip

\begin{cor}
Let $T:=S_{E}^*|_{\mathcal H(\Delta)}$. \ If
$\hbox{rank}\,(I-T^*T)<\infty$, then
$$
\mu_T=\deg_B(\Delta).
$$
\end{cor}

\medskip

\begin{proof}
This follows at once from Theorem \ref{maintheorem_sm} together with
the observation that if $\Delta$ is an inner function with values in
$\mathcal B(E^\prime, E)$, then $\dim\, E^\prime\le \dim\,E =
\hbox{rank}\,(I-T^*T)<\infty$, where the second equality comes from
the Model Theorem (cf. p.\pageref{MT}, paragraph containing (\ref{stronglimit})).
\end{proof}

\medskip

\begin{rem}\label{rembdbffvfv}
We conclude with some observations on Theorem \ref{maintheorem_sm}.

\begin{itemize}
\item[(a)] From a careful analysis of the proof of Theorem \ref{maintheorem_sm},
we can see that (\ref{Bbbb}) holds in general without the assumption
``$\dim\, E^\prime<\infty$'': more concretely, given an inner
function $\Delta$ with values in $\mathcal B(E^\prime,E)$, if
$T:=S_E^*\vert_{\mathcal H(\Delta)}$, then
$$
\hbox{deg}_B(\Delta)\le \mu_T.
$$

\medskip

\item[(b)] From  Remark \ref{remark7.9} and (\ref{main2k}),
we see that if
$$
\Delta:=\begin{bmatrix} z&0&0\\0&z&0\\0&0&1\\0&0&0\end{bmatrix},
$$
then $\hbox{deg}_{B}(\Delta)=2$. \ Let $T:=S_{\mathbb
C^4}^*|_{\mathcal H(\Delta)}$. \ Observe that
$$
\mathcal H(\Delta)=\mathcal H(z)\oplus \mathcal H(z)\oplus
\{0\}\oplus H^2.
$$
Since $\mathcal H(z)\oplus \mathcal H(z)$ has no cyclic vector, we
must have $\mu_T \neq 1$. \ In fact, if we put
$$
f=
\begin{bmatrix}
1\\0\\0\\a\end{bmatrix}\ \
\hbox{and } \ \
g=\begin{bmatrix} 0\\1\\0\\0\end{bmatrix},
$$
where $\overline a$ is not of bounded type, then $E^*_{\{f,
g\}}=\mathcal H(\Delta)$, which implies $\mu_T=2$. \ This
illustrates Theorem \ref{maintheorem_sm}.
\end{itemize}
\end{rem}

We now answer Question \ref{mainq}(ii) in the affirmative.

\begin{rem}
Suppose $\Delta$
is an inner function with values in $\mathcal B(E^\prime, E)$, with
$\dim E^\prime<\infty$. \ If $\Phi = \Delta A^*+B$ is a canonical
decomposition of $\Phi$ in $L^2_s(\mathcal B(D,E))$. \ Then by
Theorem \ref{vectormaintheorem}, we have
$$
\hbox{ker}\, H_{\breve{\Phi}}^*=\Delta H^2_{E^{\prime}}.
$$
It thus follows from the proof of Theorem \ref{maintheorem_sm} that
$$
E_{\{F\}}^*=\mathcal H(\Delta),
$$
where $F$ is defined by
$$
F(z):=\overline{z}\bigl(\Phi_+(z)-\hat{\Phi}(0)\bigr).
$$
This gives an answer to the problem of describing the set $\{F\}$ in
$H^2_E$ such that $\mathcal H(\Delta)=E_{\{F\}}^*$, given an inner
function $\Delta$ with values in $\mathcal B(E^\prime, E)$, with
$\dim\,E^\prime<\infty$.
\end{rem}

%
%
%
%
%
%

\chapter{The spectral multiplicity of model operators} \

In this chapter, we consider Question \ref{q444}: {\it Let
$T:=S_{E}^*|_{\mathcal H(\Delta)}$. For which inner function
$\Delta$ with values in $\mathcal B(E^\prime, E)$, does it follow
that}
$$
T \textit{ is multiplicity-free, i.e., } \mu_T=1 \textit{?}
$$
If $\dim E^\prime<\infty$, then in the viewpoint of Theorem
\ref{maintheorem_sm}, Question \ref{q444} is equivalent to the
following: if $T$ is the truncated backward shift
$S_{E}^*|_{\mathcal H(\Delta)}$, which inner function $\Delta$
guarantees that $\hbox{deg}_B(\Delta)=1$ ? \ To answer Question
\ref{q444}, in Section 7.1, we consider the notion of the
characteristic scalar inner function of operator-valued inner
functions having a meromorphic pseudo-continuation of bounded type in
$\mathbb{D}^e \equiv \{z: 1<|z|\le\infty\}$. \ In Section 7.2, we give an answer to Question \ref{q444}. \
In Section 7.3, we consider a reduction to the case of
$C_0$-contractions for the spectral multiplicity of model operators.

\vskip 1cm

%
%
%
%

\noindent {\bf \S\ 7.1. Characteristic scalar inner functions} \

\bigskip
\noindent In this section we consider the characteristic scalar
inner functions of operator-valued inner functions, by using the
results of Section 4.5.  \ The characteristic scalar inner function
of a two-sided inner matrix function has been studied in \cite{Hel},
\cite{SFBK} and \cite{CHL3}.

Let $\Delta\in H^\infty(\mathcal B(D,E))$ have a meromorphic
pseudo-continuation of bounded type in $\mathbb{D}^e$. \ Then by Lemma
\ref{thgg6688}, there exists a scalar inner function $\delta$ such
that $\delta H^2_{E} \subseteq \hbox{ker}\,H_{\Delta^*}$. \ Put
$G:=\delta\Delta^*\in H^{\infty}(\mathcal B(E, D))$. \ If further
$\Delta$ is inner then $ G \Delta=\delta I_D$, so that
$$
\hbox{g.c.d.}\,\bigl\{\delta : \ G \Delta=\delta I_D \ \hbox{for
some} \ G \in H^\infty(\mathcal B(E, D)) \bigr\}
$$
always exists. \ Thus the following definition makes sense.

\medskip

\begin{df}\label{def5.1}
Let $\Delta$ be an inner function with values in $\mathcal B(D,E)$.
\  If $\Delta$ has a meromorphic pseudo-continuation of bounded type
in $\mathbb{D}^e$, define
$$
m_{\Delta}:= \hbox{g.c.d.}\,\bigl\{\delta : \ G \Delta=\delta I_D \
\hbox{for some} \  G \in H^\infty(\mathcal B(E, D)) \bigr\},
$$
where $\delta$ is a scalar inner function. \ The inner function
$m_{\Delta}$ is called the {\it characteristic scalar inner
function} of $\Delta$.
\end{df}

\bigskip

We note that if $T\equiv P_{\mathcal H(\Delta)}S_E\vert_{\mathcal
H(\Delta)}\in C_0$, then $m_\Delta$ coincides with the minimal annihilator
$m_T$ of $T$ (cf. \cite{Ber}, \cite{SFBK}, \cite{CHL3}).

\medskip

We would like to remark that
\begin{equation}\label{rref}
\hbox{g.c.d.}\,\bigl\{\delta : \ G \Delta=\delta I_D \ \hbox{for
some} \ G \in H^\infty(\mathcal B(E, D)) \bigr\}
\end{equation}
may exist for some inner function $\Delta$ having no meromorphic
pseudo-continuation of bounded type in $\mathbb{D}^e$. \ To see this, let
\begin{equation}\label{notbtex}
\Delta:=\begin{bmatrix}f\\g \end{bmatrix} \quad (f,g\in H^\infty),
\end{equation}
where $f$ and $g$ are given in Example \ref{ex3.6}. \ Then $\Delta$
is an inner function. \ Since $\breve f$ is not of bounded type it
follows from Corollary \ref{cor512,222} that $\Delta$ has no
meromorphic pseudo-continuation of bounded type in $\mathbb{D}^e$. \ On the
other hand, since $\Delta$ is inner, by the Complementing Lemma,
there exists a function $G \in H^{\infty}_{M_{1\times 2}}$ such that
$G \Delta$ is a scalar inner function, so that (\ref{rref}) exists.

\medskip

If $\Delta$ is an $n\times n$ square inner matrix function then we
may write $\Delta\equiv [\theta_{ij}\bar b_{ij}]$, where
$\theta_{ij}$ is inner and $\theta_{ij}$ and $b_{ij}\in H^\infty$
are coprime for each $i,j=1,2,\cdots, n$. \ In Lemma 4.12 of
\cite{CHL3}, it was shown that
$$
m_{\Delta}= \hbox{\rm
l.c.m.}\,\bigl\{\theta_{ij}:i,j=1,2,\cdots,n\bigr\}.
$$
In this section, we examine the cases of general inner functions
that have meromorphic pseudo-continuations of bounded type in $\mathbb{D}^e$.

On the other hand, if $\Phi\in H^\infty(\mathcal B(D,E))$ has a
meromorphic pseudo-continuation of bounded type in $\mathbb{D}^e$, then by
Lemma \ref{thgg6688}, $\delta H^2_{E} \subseteq
\hbox{ker}\,H_{\Phi^*}$ for some scalar inner function $\delta$. \
Thus we may also define
$$
\omega_\Phi:= \hbox{g.c.d.}\,\bigl\{\delta : \ \delta H^2_{E}
\subseteq \hbox{ker}\,H_{\Phi^*} \ \hbox{for some scalar inner
function} \ \delta \bigr\}.
$$
If $\Delta$ is an inner function with values  in $\mathcal B(D, E)$
and has a meromorphic pseudo-continuation of bounded type in $\mathbb{D}^e$,
then  $\omega_{\Delta}$ is called the {\it pseudo-characteristic
scalar inner function} of $\Delta$. \ Note that $m_{\Delta}$ is an
inner divisor of $\omega_{\Delta}$. \ If further $\Delta$ is
two-sided inner, then
\begin{equation}\label{cefcvwf}
\delta H^2_E \subseteq \hbox{ker}\, H_{\Delta^*} \Longleftrightarrow
G\equiv\delta \Delta^*\in H^{\infty}(\mathcal B(E))
\Longleftrightarrow G \Delta = \Delta G=\delta I_E,
\end{equation}
which implies $m_{\Delta}=\omega_{\Delta}$.

\medskip

The following lemma shows a way to determine $\omega_\Phi$ more
easily.

\medskip

\begin{lem}\label{thnfhfhbbbbskl}
Let $D$ and $E$ be separable complex Hilbert spaces and let
$\{d_j\}$ and $\{e_i\}$ be orthonormal bases of $D$ and $E$,
respectively. Suppose $\Phi\in H^\infty(\mathcal B(D,E))$ has a
meromorphic pseudo-continuation of bounded type in $\mathbb{D}^e$.  \ In view
of Proposition \ref{corapp-thm1}, we may write
$$
\phi_{ij}\equiv\langle \Phi d_j, \
e_i\rangle_E=\theta_{ij}\overline{a}_{ij},
$$
where $\theta_{ij}$ is inner and $\theta_{ij}$ and $a_{ij}\in
H^\infty$ are coprime. \ Then we have
$$
\omega_{\Phi}=\hbox{\rm l.c.m.}\,\bigl\{\theta_{ij}:i,j=1,2,\cdots,
\bigr\}.
$$
\end{lem}

\begin{proof}  Let $\Phi\in H^\infty(\mathcal B(D,E))$
have a meromorphic pseudo-continuation of bounded type in $\mathbb{D}^e$. \ By
Lemma \ref{thgg6688}, we may write $\Phi=\theta A^*$ for some $A \in
H^{\infty}(\mathcal B(E,D))$ and a scalar inner function $\theta$. \
Also by an analysis of the proof of Proposition \ref{corapp-thm1},
we can see that $\theta_0\equiv\hbox{\rm
l.c.m.}\,\bigl\{\theta_{ij}:i,j=1,2,\cdots, \bigr\}$ is an inner
divisor of $\theta$. \ Thus by Lemma \ref{thgg6688}, $\theta_0$ is
an inner divisor of $\omega_{\Phi}$. \ Since $\Phi\in
H^\infty(\mathcal B(D,E))$, it follows that for all $f \in H^2_E$
and $j,n \geq 1$,
\begin{equation}\label{kkkgjfd} \langle
\Phi(z)d_j, \ z^n\theta_0(z)f(z)\rangle_E \in L^2.
\end{equation}
On the other hand, for all $f \in H^2_E$,
\begin{equation}\label{49jfgd}
f(z)=\sum_{i\geq 1}\langle f(z), \ e_i\rangle e_i\equiv\sum_{i\geq
1}f_i(z) e_i  \quad \hbox{for almost all} \ z\in \mathbb T \quad(f_i
\in H^2).
\end{equation}
Since $\theta_0=\hbox{\rm
l.c.m.}\,\bigl\{\theta_{ij}:i,j=1,2,\cdots, \bigr\}$, it follows
from (\ref{kkkgjfd}) and (\ref{49jfgd}) that for all $j,n \geq 1$,
$$
\aligned \int_{\mathbb T}\langle\Phi(z)d_j, \
z^n\theta_0(z)f(z)\rangle_Edm(z) &=\int_{\mathbb T}
\overline{z}^n\sum_{i\geq1}\overline{f_i}(z)\overline{\theta_0}(z)
\theta_{ij}(z)\overline{a}_{ij}(z)dm(z)\\
&=0,
\endaligned
$$
where the last equality follows from the fact that
$\overline{z}^n\sum_{i\geq1}\overline{f_i}(z)\overline{\theta_0}(z)
\theta_{ij}(z)\overline{a}_{ij}(z) \in L^2 \ominus H^2$. \ Since
$\{d_i\}$ is an orthonormal basis for $D$, it follows from Fatou's Lemma that for all $x \in D$ and $n=1,2,3,\cdots$,
$$
\int_{\mathbb T}\langle \Phi(z)x, \
z^n\theta_0(z)f(z)\rangle_{E}dm(z)=0.
$$
Thus by Lemma \ref{thm4.2}, $\theta_0 H^2_E \subseteq \hbox{ker}\,
H_{\Phi^*}$, so that $\omega_{\Phi}$ is an inner divisor of
$\theta_0$, and therefore $\theta_0=\omega_{\Phi}$. \ This complete
the proof.
\end{proof}

\medskip

\begin{cor}\label{cor720jfjg}  Let $\Delta$ be a two-sided inner matrix
function. \ Thus, in view of Corollary \ref{cor512,222}, we may
write $\Delta\equiv \bigl[\theta_{ij}\overline{b}_{ij}\bigr]$, where
$\theta_{ij}$ is an inner function and $\theta_{ij}$ and $b_{ij}\in
H^\infty$ are coprime for each $i,j=1,2,\cdots$. \ Then
$$
\omega_{\Delta}=m_\Delta=\hbox{\rm
l.c.m.}\,\bigl\{\theta_{ij}:i,j=1,2,\cdots, \bigr\}.
$$
\end{cor}
\begin{proof} Immediate from Lemma \ref{thnfhfhbbbbskl}.
\end{proof}

\medskip

\begin{rem}\label{exlema444}
If $\Delta$ is not two-sided inner then Corollary \ref{cor720jfjg}
may fail.  \ To see this, let
$$
\Delta:=\frac{1}{\sqrt{2}}\begin{bmatrix}1\\z\end{bmatrix}.
$$
Then by Corollary \ref{cor512,222}, $\Delta$ has a meromorphic
pseudo-continuation of bounded type in $\mathbb{D}^e$. \ It thus follows from
Lemma \ref{thnfhfhbbbbskl} that $\omega_{\Delta}=z$. \ On the other
hand, let $G:=\begin{bmatrix}\sqrt{2} & 0
\end{bmatrix}$.
Then $G \Delta=1$, so that $m_{\Delta}=1\neq z=\omega_{\Delta}$. \
Note that, by Corollary \ref{cor720jfjg},
$$
[\Delta,
\Delta_c]=\frac{1}{\sqrt{2}}\begin{bmatrix}1&1\\z&-z\end{bmatrix}
\quad \hbox{and} \quad m_{[\Delta, \Delta_c]}= \omega_{[\Delta,
\Delta_c]}=z.
$$
\end{rem}

\medskip

The following lemma shows that Remark \ref{exlema444} is not
an accident.


\begin{lem}\label{thnfhjjjfhbbbbskl}
Let $\Delta$ be an inner function and have a meromorphic
pseudo-continuation of bounded type in $\mathbb{D}^e$. \ Then
$$
m_{[\Delta, \Delta_c]}=\omega_{[\Delta, \Delta_c]}=\omega_{\Delta}
$$
and $\Delta_c$ has a meromorphic pseudo-continuation of bounded type
in $\mathbb{D}^e$: in this case, $\omega_{\Delta_c}$ is an inner divisor of
$\omega_\Delta$.
\end{lem}

\begin{proof}  Suppose that $\Delta$ is an inner function with values
in $\mathcal B(D,E)$ and has a meromorphic pseudo-continuation of
bounded type in $\mathbb{D}^e$. \ Then it follows from Corollary
\ref{thmthjdkfjkf} and Lemma \ref{thgg668} that  $[\Delta,
\Delta_c]$ is two-sided inner. \  On the other hand, it follows from
Lemma \ref{thm2.9} that
$$
\hbox{ker}\,H_{\Delta^*}=[\Delta, \Delta_c] H^2_{D \oplus
D^{\prime}}= \hbox{ker}\,H_{[\Delta, \Delta_c]^*}.
$$
Thus by Lemma \ref{thgg6688},  $[\Delta, \Delta_c]$ has a
meromorphic pseudo-continuation of bounded type in $\mathbb{D}^e$ and
$m_{[\Delta, \Delta_c]}=\omega_{[\Delta,
\Delta_c]}=\omega_{\Delta}$.  \ This proves the first assertion. \
Since $[\Delta, \Delta_c]$ has a meromorphic pseudo-continuation of
bounded type in $\mathbb{D}^e$, it follows from Lemma \ref{thgg6688} that
$\Delta_c$ has a meromorphic pseudo-continuation of bounded type in
$\mathbb{D}^e$. \ On the other hand, by Lemma \ref{thm2.9}(b), $\Delta_c^*\Delta=0$. \ Thus, by Lemma \ref{thm2.9}(a), $\Delta H_D^2 \subseteq \ker \Delta_c^*=\Delta_{cc}H_{D^{\prime \prime}}^2$, which implies that $\Delta_{cc}$ is a left inner divisor of $\Delta$. \ Thus, $[\Delta_{cc}, \Delta_c]$ is a left inner divisor of
$[\Delta, \Delta_c]$, so that $\omega_{\Delta_c}=\omega_{[\Delta_{cc}, \Delta_c]}$ is an
inner divisor of $\omega_\Delta=\omega_{[\Delta, \Delta_c]}$. \ This
proves the second assertion.
\end{proof}

\vskip 1cm

%
%
%
%
%
%

\noindent {\bf \S\ 7.2. Multiplicity-free model operators} \

\bigskip
\noindent In this section we give an answer to Question \ref{q444}.
This is accomplished by several lemmas.

\bigskip

\begin{lem}\label{cyclicgggg}
Let $\Phi \in H^\infty(\mathcal B(D, E))$ have a meromorphic
pseudo-continuation of bounded type in $\mathbb{D}^e$. \ Then for each cyclic
vector $g$ of $S^*_{D}$,
$$
\hbox{ker}\, H_{[z\Phi g]^{\smallsmile}}^*= \hbox{ker}\, \Phi^*,
$$
where $[z\Phi g]^{\smallsmile}$ denotes the flip of  $[z\Phi g]$.
\end{lem}

\begin{proof}
Let $\Phi \in H^\infty(\mathcal B(D, E))$ have a meromorphic
pseudo-continuation of bounded type in $\mathbb{D}^e$. \ Then by Lemma
\ref{thgg6688}, there exists a scalar inner function $\delta$ such
that $\delta H^2_E \subseteq \hbox{ker}\, H_{\Phi^*}$. \ We thus
have
\begin{equation}\label{813813}
\delta\Phi^*h\in H^2_D\quad\hbox{for any} \ h\in H^2_E.
\end{equation}
Let $g$ be a cyclic vector of $S_D^*$ and $h \in \hbox{ker}\,
H_{[z\Phi g]^{\smallsmile}}^*$. \ Then it follows from Lemma
\ref{thm4.2} that for all $ n=1,2,3\cdots$,
$$
\aligned 0&=\int_{\mathbb T}\bigl\langle z\Phi(z)g(z), \
z^{n}\delta(z)h(z)
\bigr\rangle_{E}dm(z)\\
&=\int_{\mathbb T}\bigl\langle S_D^{*(n-1)} g(z), \
\delta(z)\Phi^*(z)h(z)\bigr\rangle_{D}dm(z)\\
&=\bigl\langle S_D^{*(n-1)} g(z), \
\delta(z)\Phi^*(z)h(z)\bigr\rangle_{L^2_{D}},
\endaligned
$$
which implies, by (\ref{813813}), that $\delta \Phi^* h=0$, and
hence $h\in \ker\, \Phi^*$. \ We thus have
$$
\hbox{ker}\, H_{[zg_{\Phi}]^{\smallsmile}}^*\subseteq \hbox{ker}\,
\Phi^*\,.
$$
The reverse inclusion follows  at once from Lemma \ref{thm4.2}. \
This completes the proof.
\end{proof}

\bigskip

\begin{lem}\label{corcyclicgggg}  Let
$\Phi \in H^\infty(\mathcal B(D, E))$ have a meromorphic
pseudo-continuation of bounded type in $\mathbb{D}^e$. \ Then for each cyclic
vector $g$ of $S^*_{D}$,
\begin{equation}\label{jdefwvkjer}
E^*_{\{\Phi g\}}=\mathcal H((\Phi^i)_c),
\end{equation}
where $\Phi^i$ denotes the inner part in the inner-outer
factorization of $\Phi$. \ Hence, in particular, $S_{E}^*|_{\mathcal
H((\Phi^i)_c)}$ is multiplicity-free.

\end{lem}

\begin{proof} Let $\Phi\equiv \Phi^i \Phi^e$ be the
inner-outer factorization of $\Phi$. \ Since $\Phi^e$ has dense
range, $(\Phi^{e})^*$ is one-one, so that $\hbox{ker}\,
\Phi^*=\hbox{ker}\, (\Phi^{i})^*$. \ It thus follows from Lemma
\ref{thm2.9}, Lemma \ref{thm327} and Lemma \ref{cyclicgggg} that
$$
E^*_{\{\Phi g\}}=\bigl(\hbox{ker}\, H_{[z\Phi
g]^{\smallsmile}}^*\bigr)^{\perp} =\bigl(\hbox{ker}\,
\Phi^*\bigr)^{\perp}=\mathcal H((\Phi^i)_c),
$$
which proves (\ref{jdefwvkjer}). \ This completes the
proof.\end{proof}

\medskip

The following corollary is a matrix-valued version of Lemma
\ref{corcyclicgggg}.

\begin{cor}\label{corjtyucyclicgggg}
Let $\Delta$ be an $n\times r$ inner matrix function such that
$\breve{\Delta}$ is of bounded type. \ If $g$ is a cyclic vector of
$S^*_{\mathbb C^r}$, then $E^*_{\{\Delta g\}}=\mathcal H(\Delta_c)$.
\end{cor}

\begin{proof} It follows from Corollary \ref{cor512,222} and
Lemma \ref{corcyclicgggg}.
\end{proof}

\medskip

The following lemma shows that the flip of the adjoint of an inner
function may be an outer function.

\medskip

\begin{lem}\label{qqqqqjjjj}
Let $\Delta$ be an inner function with values in $\mathcal B(D,E)$,
with its complementary factor $\Delta_c$ with values in $\mathcal
B(D^\prime, E)$. \ If $\dim D^\prime<\infty$, then
$\widetilde{\Delta_c}$ is an outer function.
\end{lem}

\begin{proof} If $D^{\prime}=\{0\}$, then this is trivial.
Suppose that  $D^\prime=\mathbb C^p$ for some $p\geq 1$. \ Write
\begin{equation}\label{fjfjfj}
\widetilde{\Delta_c}\equiv
(\widetilde{\Delta_c})^{i}(\widetilde{\Delta_c})^{e}
\quad(\hbox{inner-outer factorization}),
\end{equation}
where $(\widetilde{\Delta_c})^{i}\in H^{\infty}_{M_{p\times q}}$ and
$(\widetilde{\Delta_c})^{e} \in H^{\infty}(\mathcal B(E, \mathbb
C^q))$ for some $q\le p$. \ It thus follows that
$$
q=\hbox{Rank}\,(\widetilde{\Delta_c})^{i}\geq\hbox{Rank}\,\widetilde{\Delta_c}
=\hbox{max}_{\zeta \in \mathbb
D}\,\hbox{rank}\,\widetilde{\Delta_c}(\zeta)\widetilde{\Delta_c}(\zeta)^*
=p,
$$
which implies $p=q$. \ Since $(\widetilde{\Delta_c})^{i}\in
H^{\infty}_{M_p}$ is two-sided inner, by the Complementing Lemma,
there exists a function $G\in H^{\infty}_{M_{p}}$ and  a scalar
inner function $\theta$ such that $G
(\widetilde{\Delta_c})^{i}=\theta I_{p}$. \ Thus by (\ref{fjfjfj}),
we have $G\widetilde{\Delta_c}=\theta
I_{p}(\widetilde{\Delta_c})^{e}$, and hence we have
$$
\breve{\theta}I_E \Delta_c
\widetilde{G}=\widetilde{\overline{\theta} I_pG
\widetilde{\Delta_c}}=\widetilde{(\widetilde{\Delta_c})^e}\in
H^{\infty}(\mathcal B(\mathbb C^p, E)).
$$
Thus
we have
\begin{equation}\label{mmmzmzmmz}
\breve{\theta}I_E \Delta_c \widetilde{G} H^2_{\mathbb C^p} \subseteq
H^2_{E}.
\end{equation}
It thus follows from Lemma \ref{thm2.9} and (\ref{mmmzmzmmz}) that
$$
\Delta_c \breve{\theta}I_p \widetilde{G} H^2_{\mathbb
C^p}=\breve{\theta}I_E \Delta_c \widetilde{G}H^2_{\mathbb
C^p}\subseteq \hbox{ker}\, \Delta^*=\Delta_cH^2_{\mathbb C^{p}},
$$
which implies $\breve{\theta}I_p \widetilde{G} H^2_{\mathbb
C^p}\subseteq H^2_{\mathbb C^p}$. \ We thus have $\breve{\theta}I_p
\widetilde{G} \in H^{\infty}_{M_{p}}$, so that $\overline{\theta}I_p
G\in H^{\infty}_{M_{p}}$. \ Therefore we may write $G=\theta I_p
G_1$ for some $G_1\in H^{\infty}_{M_{p}}$. \  It thus follows that
$$
\theta I_p= G (\widetilde{\Delta_c})^i =\theta I_p G_1
(\widetilde{\Delta_c})^i ,
$$
which gives that $G_1 (\widetilde{\Delta_c})^i =I_p$. \ Therefore
we have
\begin{equation}\label{ggglhlhl}
H^2_{\mathbb
C^p}=\widetilde{(\widetilde{\Delta_c})^i}\widetilde{G_1}
H^2_{\mathbb C^p}\subseteq \widetilde{(\widetilde{\Delta_c})^i}
H^2_{\mathbb C^p},
\end{equation}
which implies that $\widetilde{(\widetilde{\Delta_c})^i}$ is a
unitary matrix, and so is $(\widetilde{\Delta_c})^i$. \ Thus,
$\widetilde{\Delta_c}$ is an outer function. \ This completes the
proof.
\end{proof}

\medskip

\begin{cor}\label{orqqqqqjjjj}
If $\Delta$ is an inner matrix function, then $\Delta_c^t$ is an
outer function.
\end{cor}

\begin{proof}
Immediate from Lemma \ref{qqqqqjjjj}.
\end{proof}

\medskip

\begin{rem}
Let $T:=S_{\mathbb C^n}^*|_{\mathcal H(\Delta)}$ for some non-square
inner matrix function $\Delta$. \ Then Corollary
\ref{corjtyucyclicgggg} shows that if $\Delta=\Omega_c$ for an inner
matrix function $\Omega$ such that $\breve\Omega$ is of bounded
type, then $T$ is multiplicity-free. \ However, the converse is not
true in general, i.e., the condition ``multiplicity-free" does not
guarantee that $\Delta=\Omega_c$. \ To see this, let $\Delta:=[0 \
z]^t$. Then $\Delta$ is inner and $\breve{\Delta}$ is of bounded
type. \ Since $\Delta^t=[0  \ z]$ is not an outer function, it
follows from Corollary \ref{orqqqqqjjjj} that $\Delta \neq \Omega_c$
for any inner matrix function. \ Let $f:=(a \ 1)^t$ ($\overline a$
is not of bounded type). \ Then $E_f^*=\mathcal H(\Delta)$, so that
$T$ is multiplicity-free.
\end{rem}

\medskip

\begin{lem}\label{hfhnbnnn} Let $\Delta$ be an inner function
and have a meromorphic pseudo-continuation of bounded type in $\mathbb{D}^e$.
If $\widetilde{\Delta}$ is an outer function and $\hbox{ker}\,
\Delta^*=\{0\}$, then $\Delta$ is a unitary operator.
\end{lem}
\begin{proof}  Let $\Delta$ be an inner function with values in
$\mathcal B(D,E)$ and have a meromorphic pseudo-continuation of
bounded type in $\mathbb{D}^e$. \ Then by Lemma \ref{thgg668},
$\breve{\Delta}$ is of bounded type. \ Suppose that
$\widetilde{\Delta}$ is an outer function and $\hbox{ker}\,
\Delta^*=\{0\}$. \ Then by Lemma \ref{thm2.9}, Corollary
\ref{thmthjdkfjkf} and Lemma \ref{thgg668}, $\Delta$ is two-sided
inner, and so is $\widetilde{\Delta}$. \ Thus $\Delta$ is a unitary
operator, as desired.
\end{proof}

\medskip

The following lemma is a key idea for an answer to Question
\ref{q444}.

\begin{lem}\label{ffbfbbfbfbfvvvbbfLll}
Let $\Delta$ be an inner function and have a meromorphic
pseudo-continuation of bounded type in $\mathbb{D}^e$. \ If
$\widetilde{\Delta}$ is an outer function, then
$$
\Delta_{cc}=\Delta.
$$
\end{lem}

\begin{proof} Let $\Delta$ be an inner function with values in
$\mathcal B(D,E)$ and have a meromorphic pseudo-continuation of
bounded type in $\mathbb{D}^e$. \ Also, suppose $\widetilde{\Delta}$ is an
outer function. \ If $\hbox{ker}\, \Delta^*=\{0\}$, then the result
follows at one from Lemma \ref{hfhnbnnn}. \ Assume that
$\hbox{ker}\, \Delta^*\neq \{0\}$. \ By Lemma \ref{thgg668},
$\breve{\Delta}$ is of bounded type, so that by Corollary
\ref{thmthjdkfjkf}, $[\Delta, \Delta_c]$ is a two-sided inner
function with values in $\mathcal B(D\oplus D^{\prime},E)$ for some
nonzero Hilbert space $D^\prime$. \ We now claim that
\begin{equation}\label{615}
\Delta=\Delta_{cc}\Omega \quad \hbox{for a two-sided inner function
$\Omega$ with values in $\mathcal B(D)$}.
\end{equation}
Since $\Delta_{cc}$ is a left inner divisor of $\Delta$ (cf. the Proof of Lemma \ref{thnfhjjjfhbbbbskl}), we may write
\begin{equation}\label{dddhfhfh}
\Delta=\Delta_{cc}\Omega
\end{equation}
for an inner function $\Omega$ with values in $\mathcal
B(D,D^{\prime\prime})$. \ Assume to the contrary that $\Omega$ is
not two-sided inner. \  Since $\Delta$ has a meromorphic
pseudo-continuation of bounded type in $\mathbb{D}^e$, it follows from Lemma
\ref{thgg6688} that
$$
\theta H^2_E \subseteq \hbox{ker}\, H_{\Delta^*}=\hbox{ker}\,
H_{\Omega^*\Delta_{cc}^*}
$$
for some scalar inner function $\theta$. \ Thus
$\Omega^*\Delta_{cc}^*\theta H^2_E\subseteq H^2_D$.  \ In
particular, we have
$$
\Omega^*\theta H^2_{D^{\prime\prime}}=\Omega^*\Delta_{cc}^*\theta
\Delta_{cc}H^2_{D^{\prime\prime}}\subseteq H^2_D,
$$
and hence $\theta H^2_{D^{\prime\prime}} \subseteq \hbox{ker}\,
H_{\Omega^*}$, which implies, by Lemma \ref{thgg6688},  that
$\Omega$ has a meromorphic pseudo-continuation of bounded type in
$\mathbb{D}^e$. \ Thus by Lemma \ref{thgg668}, $\breve{\Omega}$ is of bounded
type.  \ It thus follows from Lemma \ref{thm2.9} that
$$
[\Omega, \Omega_c] \ \hbox{is two-sided inner},
$$
where $\Omega_c$ is the complementary factor of $\Omega$, with
values in $\mathcal B(D_1, D^{\prime\prime})$ for some nonzero
Hilbert space $D_1$. \ On the other hand, it follows from
(\ref{dddhfhfh}) that for all $f \in H^2_{D_1}$,
$$
[\Delta, \Delta_c]^*\Delta_{cc}\Omega_cf =
\begin{bmatrix}\Omega^*\Omega_cf\\
\Delta_c^*\Delta_{cc}\Omega_cf\end{bmatrix}=0,
$$
which implies that $D_1=\{0\}$, a contradiction.  \ This proves
(\ref{615}). Thus we may write
\begin{equation}\label{vvfffffvvsvsvs}
\widetilde{\Delta}=\widetilde{\Omega}\widetilde{\Delta_{cc}}
\end{equation}
for a two-sided inner function $\widetilde{\Omega}$ with values in
$\mathcal B(D)$. \ Since $\widetilde{\Delta}$ is an outer function
and $\widetilde{\Omega}$ is two-sided inner, it follows from
(\ref{vvfffffvvsvsvs}) that $\widetilde{\Omega}$ is a  unitary
operator, and so is $\Omega$. \ This completes the proof.
\end{proof}

\medskip

Lemma \ref{ffbfbbfbfbfvvvbbfLll} may fail if the condition
``$\Delta$ has a meromorphic pseudo-continuation of bounded type in
$\mathbb{D}^e$" is dropped. \ To see this. let
$$
\Delta:=\begin{bmatrix}f\\g\\0\end{bmatrix},
$$
where $f$ and $g$ are given in Example \ref{ex3.6}. \ Then
$\widetilde{\Delta}$ is an outer function. \ A straightforward
calculation shows that
$$
\Delta_c=\begin{bmatrix}0\\0\\1\end{bmatrix}  \quad \hbox{and} \quad
\Delta_{cc}=\begin{bmatrix}1&0\\0&1\\0&0\end{bmatrix}\neq \Delta.
$$
Note that $\breve{\Delta}$ is not of bounded type. \ Thus, by
Corollary \ref{cor512,222}, $\Delta$ has no meromorphic
pseudo-continuation of bounded type in $\mathbb{D}^e$.

\bigskip

We are ready to give an answer to Question \ref{q444}.

\medskip

\begin{thm} (Multiplicity-free model operators)\label{wwwkslgjho}
Let $T:= S_{E}^*\vert_{\mathcal H(\Delta)}$. \ If $\Delta$ has a
meromorphic pseudo-continuation of bounded type in $\mathbb{D}^e$ and
$\widetilde{\Delta}$ is an outer function, then $T$ is
multiplicity-free.
\end{thm}

\begin{proof}
Let $T:= S_{E}^*\vert_{\mathcal H(\Delta)}$. \ Suppose that $\Delta$
has a meromorphic pseudo-continuation of bounded type in $\mathbb{D}^e$ and
$\widetilde{\Delta}$ is an outer function. \ If $\ker
\Delta^*=\{0\}$, then by Lemma \ref{hfhnbnnn}, $\Delta$ is a unitary
operator, so that
 $T$ is multiplicity-free. If instead $\hbox{ker}\,
\Delta^* \neq \{0\}$, then by Lemma \ref{thnfhjjjfhbbbbskl},
$\Delta_c$ has a meromorphic pseudo-continuation of bounded type in
$\mathbb{D}^e$.  \ Since $\widetilde{\Delta}$ is an outer function it follows
from Lemma \ref{ffbfbbfbfbfvvvbbfLll} that $\Delta=\Delta_{cc}$.\
Applying Lemma \ref{corcyclicgggg} with $\Phi\equiv \Delta_c$, we
can see that $T$ has a cyclic vector, i.e., $T$ is
multiplicity-free.
\end{proof}

\medskip

The following corollary is an immediate result of Theorem
\ref{wwwkslgjho}.


\begin{cor} \label{cor5.16}
Let $T:=S_{\mathbb C^n}^*|_{\mathcal H(\Delta)}$ for an inner matrix
function $\Delta$ whose flip $\breve{\Delta}$ is of bounded type. \
If $\Delta^t$ is an outer function, then $T$ is multiplicity-free.
\end{cor}

\begin{proof}
This follows from Theorem \ref{wwwkslgjho} and Corollary
\ref{cor512,222}.
\end{proof}

\medskip

If $\Delta$ is an inner matrix function then the converse of Lemma
\ref{ffbfbbfbfbfvvvbbfLll} is also true.

\begin{cor}\label{cor33fbfbbfLll} \
Let $\Delta$ be an inner matrix function whose flip $\breve\Delta$
is of bounded type. \ Then the following are equivalent:

\begin{itemize}
\item[(a)]  $\Delta^t$ is an outer function;
\item[(b)] $\widetilde{\Delta}$ is an outer function;
\item[(c)] $\Delta_{cc}=\Delta$;
\item[(d)] $\Delta=\Omega_{c}$ for some inner
  matrix function $\Omega$.
\end{itemize}
Hence, in particular, $\Delta_{ccc}=\Delta_c$.
\end{cor}

\begin{proof}
The implication (a)$\Rightarrow$(b) is clear and the implication
(b)$\Rightarrow$(c) follows from Corollary \ref{cor512,222} and
Lemma \ref{ffbfbbfbfbfvvvbbfLll}. \ Also the implication
(c)$\Rightarrow$(d) is clear and the implication (d)$\Rightarrow$(a)
follows from Corollary \ref{orqqqqqjjjj}. \ The second assertion
follows from the first assertion together with Corollary
\ref{thmthjdkfjkf} and Corollary \ref{orqqqqqjjjj}.
\end{proof}

\vskip 1cm

%
%
%
%

\noindent {\bf \S\ 7.3. A reduction to the case of
$C_0$-contractions} \

\bigskip
\noindent On the other hand, the theory of spectral multiplicity for
$C_0$-operators has been well developed in terms of their
characteristic functions (cf. \cite[Appendix 1]{Ni1}). \ However
this theory is not applied directly to $C_{0\, \bigcdot}$-operators, in
which cases their characteristic functions need not be two-sided
inner. \ The object of this section is to show that if the
characteristic function of a $C_{0\, \bigcdot}$-operator $T$ has a
meromorphic pseudo-continuation of bounded type in $\mathbb{D}^e$, then its
spectral multiplicity can be computed by that of the $C_0$-operator
induced by $T$.

\bigskip

We first observe:

\medskip

\begin{lem}\label{nclf;gop} If $\Phi \in L^2_{\mathcal{B}(D,E)}$ and
$f\in H^{\infty}_D$, then $\Phi f \in L^2_E$.
\end{lem}
\begin{proof} Suppose $\Phi \in L^2_{\mathcal{B}(D,E)}$ and
$f\in H^{\infty}_D$. \ Since $f$ is strongly measurable, there exist
countable valued functions $f_n=\sum_{k= 1}^{\infty}d_k^{(n)}
\chi_{\sigma_k^{(n)}}$ such that $f(z)=\lim_n f_n(z)$ for almost all
$z \in \mathbb T$. \ For all $e \in E$ and $n=1,2,3,\cdots$,
\begin{equation}\label{fveboll}
\bigl \langle \Phi(z)f_n(z), \ e \bigr\rangle_E=\sum_{k= 1}^{\infty}
\chi_{\sigma_k^{(n)}}(z) \cdot \bigl \langle \Phi(z)d_k^{(n)} , \ e
\bigr\rangle_D.
\end{equation}
But since $\Phi$ is WOT measurable, by (\ref{fveboll}), $\Phi f_n$
is weakly measurable and in turn, $\Phi f:\mathbb T\to E$ is weakly
measurable, and hence it is strongly measurable. \ Observe that
$$
\int_{\mathbb T}||\Phi(z)f(z)||^2_Edm(z)\leq
||f||_{\infty}\int_{\mathbb T}||\Phi(z)||^2 dm(z)<\infty,
$$
which implies that $\Phi f \in L^2_E$. \ This completes the proof.
\end{proof}

\medskip

\begin{lem}\label{bnmmdjfjjf} Let $\Phi \in L^2_{\mathcal B(D,
E)}$ and let $A:H^2_D \to H^2_E$ be a densely defined operator, with
domain $H^\infty_D\subset H^2_D$, defined by
$$
Af:=JP_-(\Phi f) \quad(f \in H^{\infty}_D).
$$
Then
$$
\hbox{ker}\, A^*=\hbox{ker}\, H_{\Phi}^*.
$$
\end{lem}
\begin{proof}  Let $\Phi \in L^2_{\mathcal B(D,
E)}\subseteq L^2_s(\mathcal B(D, E))$. \ Since the domain of
$H_{\Phi}$ is a subset of the domain of $A$, it follows that the
domain of $A^*$ is a subset of the domain of $H_{\Phi}^*$, so that
$\hbox{ker}\, A^*\subseteq \hbox{ker}\, H_{\Phi}^*$. \ For the
reverse inclusion, suppose $g \in \hbox{ker}\, H_{\Phi}^*$. \ Then
\begin{equation}\label{nnswndcnvnk}
\bigl\langle H_{\Phi} p, \ g \bigr\rangle_{L^2_E}=0 \quad \hbox{for
all} \ p \in \mathcal P_{D}.
\end{equation}
Let $f\in H^{\infty}_D$ be arbitrary. \ Then we may write
$$
f(z)=\sum_{k=0}^{\infty}a_k z^k \quad(a_k \in D).
$$
Let
$$
p_n(z):=\sum_{k=0}^{n}a_k z^k \in \mathcal P_D.
$$
Then  it follows from (\ref{nnswndcnvnk}) that
$$
0=\lim_{n\rightarrow \infty}\bigl\langle  H_{\Phi}p_n, \ g
\bigr\rangle_{L^2_E}=\lim_{n\rightarrow \infty}\bigl\langle  p_n, \
\Phi^*Jg \bigr\rangle_{L^2_D}=\bigl\langle  \Phi f, \ Jg
\bigr\rangle_{L^2_E}=\bigl\langle  A f, \ g \bigr\rangle_{L^2_E},
$$
which implies that $g \in \hbox{ker}\, A^*$, so that  $\hbox{ker}\,
H_{\Phi}^*\subseteq \hbox{ker}\, A^*$. \ This completes the proof.
\end{proof}

\medskip

\begin{cor}\label{llllamdmnd} If $\Phi \in H^2_{\mathcal B(D, E)}$, then
$$
E^*_{\{\Phi\}}=\hbox{cl}\,
\Bigl\{JP_-\bigl(\overline{z}\breve{\Phi}h\bigr): h \in
H^{\infty}_D\Bigr\}.
$$
\end{cor}

\begin{proof} Define $A:H^2_D \to H^2_E$ by
$Af:=JP_-(\overline{z}\breve{\Phi} h)$ ($h \in H^{\infty}_D$). \ By
Lemma \ref{bnmmdjfjjf}, $\hbox{ker}\,
H_{\overline{z}\breve{\Phi}}^*=\hbox{ker}\, A^*$. By
(\ref{ephistar}), we have
$$
E^*_{\{\Phi\}}=\hbox{cl ran}\, H_{\overline{z}\breve{\Phi}}=\hbox{cl
ran}\, A =\hbox{cl}\,
\Bigl\{JP_-\bigl(\overline{z}\breve{\Phi}h\bigr): h \in
H^{\infty}_D\Bigr\}.
$$
\end{proof}

\medskip

We thus have:

\medskip

\begin{lem}\label{dcebsjdhvkrjw} Suppose $\Delta$ is a
two-sided inner function and  has a meromorphic pseudo-continuation
of bounded type in $\mathbb{D}^e$. \ Let $F\equiv \{f_1, f_2, \cdots,
f_p\}\subseteq \mathcal H(\Delta)$ . \ Then
$$
E_F^*=\bigvee\Bigl\{P_+(\breve{h}_jf_j): h_j \in H^{\infty}\cap
\mathcal H(\widetilde{\omega_{\Delta}}),  \ j=1,2,\cdots p\Bigr\},
$$
where $\omega_\Delta$ is the pseudo-characteristic scalar inner
function of $\Delta$.
\end{lem}
\begin{proof} Suppose $\Delta$ is a two-sided inner function with values
in $\mathcal B(E)$ and  has a meromorphic pseudo-continuation of
bounded type in $\mathbb{D}^e$. \ Let $F\equiv \{f_1, f_2, \cdots,
f_p\}\subseteq \mathcal H(\Delta)$. \ Write
$[F]:=\begin{bmatrix}[f_1], [f_2], \cdots, [f_p]\end{bmatrix}$ and
$\theta:= \omega_{\Delta}$. \ Since $[f_j] \in H^2_{\mathcal
{B}(\mathbb C, E)}$ for each $j=1,2,\cdots, p$, it is easy to see
that $F \in  H^{2}_{\mathcal B(\mathbb C^p, E)}$. \ We first claim
that
\begin{equation}\label{766}
E_F^* = \hbox{cl}\,\Bigl\{JP_-(\overline{z}\breve{[F]}h): h \in
H^{\infty}_{\mathbb C^p}\cap \mathcal
H(\widetilde{\theta}I_p)\Bigr\}.
\end{equation}
By Corollary \ref{llllamdmnd} we have
$$
E_F^* =\hbox{cl}\,\Bigl\{JP_-(\overline{z}\breve{[F]}h): h \in
H^{\infty}_{\mathbb C^p}\Bigr\} \supseteq
\hbox{cl}\,\Bigl\{JP_-(\overline{z}\breve{[F]}h): h \in
H^{\infty}_{\mathbb C^p}\cap \mathcal
H(\widetilde{\theta}I_p)\Bigr\}.
$$
For the reverse inclusion, it suffices to show that
\begin{equation}\label{vbfjdvsfj}
P_-(\overline{z}\breve{[F]}\widetilde{\theta}h) =0 \quad \hbox{for
all} \ h \in H^{\infty}_{\mathbb C^p}.
\end{equation}
By Lemma \ref{thgg6688}, we may write
$$
\Delta=\theta A^* \quad\hbox{for some $A\in H^{\infty}(\mathcal
B(E))$.}
$$
Since $\Delta$ is two-sided inner, it follows that
$I_E=\Delta\Delta^*=A^* A$, so that $\theta H^2_{E}=\Delta A
H^2_{E}\subseteq \Delta H^2_{E}$. \ We thus have
$$
\mathcal H(\Delta)\subseteq \mathcal H(\theta I_E).
$$
Thus $f_j\in \mathcal H(\theta I_E)$ ($j=1,\cdots, p$), so that $
\overline{\theta}f_j \in L^2_{E}\ominus H^2_{E}$. \ Hence for all $h
\in H^{\infty}_{\mathbb C^p}$, by Lemma \ref{nclf;gop}, we have
$\overline{\theta}[F] \breve{h} \in L^2_{E} \ominus H^2_{E}$, so
that $\overline{z}\breve{[F]}\widetilde{\theta}h \in H^2_{E}$, and
hence $P_-(\overline{z}\breve{[F]}\widetilde{\theta}h)=0$, which
gives (\ref{vbfjdvsfj}). \ This proves (\ref{766}). \
Write
$h=(h_1,h_2, \cdots h_p)^t\in H^{\infty}_{\mathbb C^p}\,\cap\,
\mathcal H(\widetilde{\theta}I_p)$, and hence $h_j \in
H^{\infty}\cap \mathcal H(\widetilde{\theta})$. \ Thus it follows
from  (\ref{766}) that
$$
\aligned E_F^*&=\hbox{cl}\,\Bigl\{JP_-(\overline{z}\breve{[F]}h): h
\in H^{\infty}_{\mathbb C^p}\cap \mathcal
H(\widetilde{\theta}I_p)\Bigr\}\\
&=\bigvee\Bigl\{JP_-(\overline{z}\breve{[f_j]}h_j): h_j \in
H^{\infty} \cap \mathcal
H(\widetilde{\theta}), \ j=1,2,\cdots p\Bigr\}\\
&=\bigvee \Bigl\{P_+(\breve{h}_jf_j): h_j \in  H^{\infty} \cap
\mathcal H(\widetilde{\theta}), \ j=1,2,\cdots p\Bigr\}.
\endaligned
$$
This completes the proof.
\end{proof}

\medskip

\begin{lem}\label{lemdchgvfervbhyu} \
Let $\Delta$ be an inner function with values in $\mathcal
B(E^\prime, E)$, with $\dim E^\prime<\infty$. \ If
$\widetilde{\Delta}= (\widetilde{\Delta})^i(\widetilde{\Delta})^e$
is the inner-outer factorization of $\widetilde \Delta$, then we
have:
\begin{itemize}
\item[(a)] $(\widetilde{\Delta})^i$ is a
two-sided inner function with values in  $\mathcal B(E^\prime)$;
\item[(b)] $\widetilde{(\widetilde{\Delta})^e}$ is an inner function
with values in $\mathcal B(E^\prime,E)$.
\end{itemize}
\end{lem}

\begin{proof}
Let $\dim E^\prime=r$. Then the inner part $(\widetilde{\Delta})^i$
is an $r\times p$ inner matrix function for some $p\leq r$. \ Thus
we have
$$
p=\hbox{Rank}\, (\widetilde{\Delta})^i \ge \hbox{Rank}\,
\widetilde{\Delta}=\hbox{Rank}\, \Delta=r,
$$
which proves (a). \ For (b), observe
$\Delta=\widetilde{(\widetilde{\Delta})^e}\widetilde{(\widetilde{\Delta})^i}$.
\  Since $\Delta$ is inner, we have that
$$
I_r=\Delta^*\Delta=\widetilde{(\widetilde{\Delta})^i}^*
\widetilde{(\widetilde{\Delta})^e}^*\widetilde{(\widetilde{\Delta})^e}
\widetilde{(\widetilde{\Delta})^i}.
$$
But since $(\widetilde{\Delta})^i$ is two-sided inner, so is
$\widetilde{(\widetilde{\Delta})^i}$. \ Thus it follows that
$$
\widetilde{(\widetilde{\Delta})^e}^*
\widetilde{(\widetilde{\Delta})^e}=\widetilde{(\widetilde{\Delta})^i}\widetilde{(\widetilde{\Delta})^i}^*
=I_r,
$$
which implies that $\widetilde{(\widetilde{\Delta})^e}$ is an inner
function. \ This proves (b).
\end{proof}

\medskip

\begin{lem}\label{thmfgdbkkknvttttnvp}
Suppose $\Delta$ is an inner function with values in $\mathcal
B(E^\prime, E))$, with $\dim E^\prime<\infty$ and  has a meromorphic
pseudo-continuation of bounded type in $\mathbb{D}^e$. \ Write
$$
\Delta_1:=\widetilde{(\widetilde{\Delta})^e}.
$$
Then $\Delta_{1}$ has a meromorphic pseudo-continuation of bounded
type in $\mathbb{D}^e$.
\end{lem}

\begin{proof}
Let $\dim E^\prime =r$ and let
$\widetilde{\Delta}=(\widetilde{\Delta})^i(\widetilde{\Delta})^e$ be
the inner-outer factorization of $\widetilde\Delta$. \ Then
$$
\Delta=\widetilde{(\widetilde{\Delta})^e}
\widetilde{(\widetilde{\Delta})^i}\equiv\Delta_1\Delta_{\rm s} \quad
\Bigl( \hbox{where}
 \  \Delta_1\equiv\widetilde{(\widetilde{\Delta})^e} \
\hbox{and} \ \Delta_{\rm s}\equiv
\widetilde{(\widetilde{\Delta})^i}\Bigr).
$$
By Lemma \ref{lemdchgvfervbhyu}, $\Delta_{\rm s}\in
H^{\infty}_{M_r}$ is square inner and $\Delta_1\in H^\infty(\mathcal
B(\mathbb C^r, E))$ is inner. \ Since $\Delta$ has a meromorphic
pseudo-continuation of bounded type in $\mathbb{D}^e$, it follows from Lemma
\ref{thgg6688} that there exists a scalar inner function $\theta$
such that $ \theta H^2_{E} \subseteq
\hbox{ker}\,H_{\Delta^*}=\hbox{ker}\,H_{\Delta_{\rm
s}^*\Delta_1^*}$. Thus we have
\begin{equation}\label{nvnnvnvnv}
\Delta_{\rm s}^*\Delta_1^*\theta H^2_{E}\subseteq H^2_{\mathbb C^r}.
\end{equation}
Since $\Delta_{\rm s}$ is square inner, it follows from
(\ref{nvnnvnvnv}) that $\Delta_1^*\theta H^2_{E}\subseteq
\Delta_{\rm s} H^2_{\mathbb C^r}\subseteq H^2_{\mathbb C^r}$, so
that $\theta H^2_E\subseteq \hbox{ker}\,H_{\Delta_1^*}$, which
implies, by  Lemma \ref{thgg6688}, that $\Delta_1$ has a meromorphic
pseudo-continuation of bounded type in $\mathbb{D}^e$. \ This completes the
proof.
\end{proof}

\medskip

\begin{lem}\label{lemhbcsfvhs}
Let $\Delta_1$ be an inner function with values in $\mathcal B(D,
E)$ and $\Delta_2$ be a two-sided inner function with values in
$\mathcal B(D)$. \ Then,
$$
\mathcal H(\Delta_1\Delta_2)=\mathcal H(\Delta_1)\bigoplus \Delta_1
\mathcal H(\Delta_2).
$$
\end{lem}

\begin{proof}
The inclusion $\mathcal H(\Delta_1)\subseteq \mathcal
H(\Delta_1\Delta_2)$ is clear and by Corollary \ref{thm2.4edfrgt},
$\Delta_1\mathcal H (\Delta_2)\subseteq \mathcal
H(\Delta_1\Delta_2)$, which gives $\mathcal H(\Delta_1)\bigoplus
\Delta_1 \mathcal H(\Delta_2)\subseteq \mathcal
H(\Delta_1\Delta_2)$. \ For the reverse inclusion, suppose $f\in
\mathcal H(\Delta_1\Delta_2)$. \ Let $f_1:= P_{\mathcal
H(\Delta_1)}(f)$ and $f_2:= f-f_1$. \ Then $f_2=\Delta_1 g$ for some
$g \in H^2_{D}$. \ Since $f \in \mathcal H(\Delta_1\Delta_2)$, it
follows from Corollary \ref{thm2.4edfrgt} that
\begin{equation}\label{cdhbjbewvhvjk}
\Delta_2^*\Delta_1^*(f_1+\Delta_1g) \in L^2_D \ominus H^2_D.
\end{equation}
Since $f_1 \in \mathcal H(\Delta_1)$, it follows from Corollary
\ref{thm2.4edfrgt} that $\Delta_1^*f_1 \in  L^2_{D} \ominus
H^2_{D}$. \ Thus by Corollary \ref{lemma4.100000}, for all $h \in
H^2_D$,
$$
\bigl \langle \Delta_2^*\Delta_1^*f_1, \ h
\bigr\rangle_{L^2_D}=\bigl \langle \Delta_1^*f_1, \ \Delta_2 h
\bigr\rangle_{L^2_D}=0,
$$
which implies that $\Delta_2^*\Delta_1^* f_1 \in L^2_D \ominus
H^2_D$. \ Thus by Corollary \ref{thm2.4edfrgt} and
(\ref{cdhbjbewvhvjk}), we have $g \in \mathcal H(\Delta_2)$, and
hence $f_2\in \Delta_1\mathcal H(\Delta_2)$. \ Therefore we have
$\mathcal H(\Delta_1\Delta_2)\subseteq\mathcal H(\Delta_1)\bigoplus
\Delta_1 \mathcal H(\Delta_2)$. \ This completes the proof.
\end{proof}

\bigskip

We are ready for:

\medskip

\begin{thm} (The spectral multiplicity of model operators)\label{hwysonghkkl}
Given an inner function $\Delta$ with values in $\mathcal
B(E^\prime, E)$, with $\dim\, E^\prime<\infty$, let
$T:=S_{E}^*\vert_{\mathcal H(\Delta)}$. \ If $\Delta$ has a
meromorphic pseudo-continuation of bounded  type in $\mathbb{D}^e$, then
\begin{equation}\label{maineq}
\mu_T=\mu_{T_s},
\end{equation}
where $T_s$ is a $C_0$-contraction of the form $T_s:=
S^*_{E^{\prime}}|_{\mathcal H(\Delta_s)}$ with $\Delta_s:=
\widetilde{(\widetilde\Delta)^i}$. Hence in particular, $\mu_T\le
\dim\, E^{\prime}$.
\end{thm}

\begin{proof}
Let $T:= S_{E}^*\vert_{\mathcal H(\Delta)}$. \ Suppose $\Delta$ has
a meromorphic pseudo-continuation of bounded type in $\mathbb{D}^e$. \ Let
$\Delta_{\rm s}\equiv\widetilde{(\widetilde{\Delta})^i}$ and write
$$
T_s:=S^*_{E^{\prime}}|_{\mathcal H(\Delta_{\rm s})}.
$$
If $\Delta$ is two-sided inner, then $\Delta=\Delta_{\rm s}$, so
that $\mu_{T}=\mu_{T_{\rm s}}$. \ Suppose that $\Delta$ is not
two-sided inner. \ Without loss of generality, we may assume that
$E^{\prime}=\mathbb C^r$. \ By Lemma \ref{lemdchgvfervbhyu},
$\Delta_{\rm s}\in H^{\infty}_{M_r}$ is square inner. \ Thus by
(\ref{C00}) and Proposition \ref{C0condition}, we have that $T_s\in
C_0$. \ We will prove that
\begin{equation}\label{mainequality}
\mu_T=\mu_{T_s}.
\end{equation}
Write
$$
\Delta_1\equiv \widetilde{(\widetilde{\Delta})^e}.
$$
Then it follows from Lemma \ref{lemdchgvfervbhyu} and Lemma
\ref{thmfgdbkkknvttttnvp} that $\Delta_1$ is an inner function
having a meromorphic pseudo-continuation of bounded type in $\mathbb{D}^e$. \
Let
\begin{equation}\label{jdsckjvll99789}
\theta:= \omega_{\Delta_1} \omega_{\Delta_{\rm s}}.
\end{equation}
Let $p:=\mu_{T_s}$. \ In view of (\ref{muTTTTe}), we have $p\leq r$.
Then there exists a set $F\equiv \{f_1, f_2, \cdots, f_p\}\subseteq
\mathcal H(\Delta_{\rm s})$ such that $E_F^*=\mathcal H(\Delta_{\rm
s})$. \ Since by (\ref{jdsckjvll99789}), $\mathcal
H(\widetilde{\omega_{\Delta_{\rm s}}}) \subseteq \mathcal
H(\widetilde\theta)$, it follows from Lemma \ref{dcebsjdhvkrjw} that
\begin{equation}\label{nnnsssbsbbs}
\mathcal H(\Delta_{\rm s})=\bigvee\Bigl\{P_+(\breve{h}_jf_j): h_j
\in H^{\infty} \cap \mathcal H(\widetilde{\theta}), \ j=1,2,\cdots
p\Bigr\}.
\end{equation}
Write
$$
\Omega:= (\Delta_1)_c \in H^{\infty}(E^{\prime\prime}, E) \quad
(E^{\prime\prime} \ \hbox{is a  subspace of} \ E).
$$
Since $\widetilde{\Delta_1}$ is outer, it follows from Lemma
\ref{ffbfbbfbfbfvvvbbfLll} that $\Delta_1=\Omega_c$. Choose a cyclic
vector $g$ of $S^*_{E^{\prime\prime}}$. \ Then it follows from
Lemma \ref{thnfhjjjfhbbbbskl}, Lemma \ref{corcyclicgggg} and Lemma
\ref{dcebsjdhvkrjw} that
\begin{equation}\label{fvnjrekjgjkbklbklb}
\mathcal
H(\Delta_1)=E^*_{\Omega g}
=\hbox{cl}\,\bigl\{P_+\bigl(\breve{h}\Omega g\bigr): h \in
H^{\infty}\bigr\}.
\end{equation}
Let
$$
\gamma_1:=\theta \Omega g+\Delta_1f_1 \quad \hbox{and} \quad
\gamma_j:=\Delta_1f_j \quad(j=2,3,\cdots,p).
$$
Now we will show that
\begin{equation}\label{ddnnmbltlh}
\mathcal H(\Delta)=\bigvee \Bigl\{P_+(\breve{\eta}_j\gamma_j):
\eta_j \in H^{\infty}, \ j=1,2,\cdots p\Bigr\}.
\end{equation}
Let $\xi \in \mathcal H(\Delta)$ and $\epsilon>0$ be arbitrary. \
Then, by Lemma \ref{lemhbcsfvhs}, we may write
$$
\xi=\xi_1+\Delta_1\xi_2 \quad \bigl(\xi_1\in \mathcal H(\Delta_1), \
\xi_2 \in \mathcal H(\Delta_{\rm s})\bigr).
$$
By (\ref{nnnsssbsbbs}), there exist $h_j \in H^{\infty}\cap \mathcal
H(\widetilde{\theta})$ ($j=1,2,\cdots,p$) such that
\begin{equation}\label{NJFVJGVKL}
\Biggl|\Biggl|~\sum_{j=1}^p
P_+(\breve{h}_jf_j)-\xi_2\Biggr|\Biggr|_{L^2_{\mathbb
C^r}}<\frac{\epsilon}{2}.
\end{equation}
For each $j=1,2,\cdots, p$, observe that
\begin{equation}\label{jjjflsls}
\aligned P_+(\breve{h}_j\Delta_1f_j)&=P_+(\Delta_1\breve{h}_jf_j)\\
&=\Delta_1
P_+(\breve{h}_jf_j)+P_+\bigl(\Delta_1P_-(\breve{h}_jf_j)\bigr),
\endaligned
\end{equation}
and
\begin{equation}\label{vbvbfbfbfb}
\Delta_1P_+(\breve{h}_jf_j)\in \Delta_1\mathcal H(\Delta_{\rm s})
\quad \hbox{and} \quad P_+\bigl(\Delta_1P_-(\breve{h}_jf_j)\bigr)
\in \mathcal H(\Delta_1).
\end{equation}
Since $\ker (\theta\Omega)^*=\ker \Omega^*$, we have
$(\theta\Omega)_c=\Omega_c$. \ Thus by (\ref{fvnjrekjgjkbklbklb}),
$P_+(\breve{h}_1\theta \Omega g)$ belongs to $\mathcal H(\Delta_1)$.
\  Thus it follows from (\ref{vbvbfbfbfb}) that
$$
\xi_0\equiv\xi_1-\sum_{j=1}^pP_+\bigl(\Delta_1P_-(\breve{h}_jf_j)\bigr)-P_+(\breve{h}_1\theta
\Omega g) \in \mathcal H(\Delta_1).
$$
Thus by (\ref{fvnjrekjgjkbklbklb}), there exists $h_0 \in
H^{\infty}$ such that
\begin{equation}\label{cvfvfvfvfffffff}
\Bigl|\Bigl|P_+\bigl(\breve{h}_0\Omega
g\bigl)-\xi_0\Bigr|\Bigr|_{L^2_E}<\frac{\epsilon}{2}.
\end{equation}
Let
$$
\eta_1:=\widetilde{\theta}h_0+h_1 \quad \hbox{and} \quad \eta_j:=h_j
\quad(j=2,3,\cdots,p).
$$
It follows from Lemma \ref{thgg6688} that
$$
\Delta_{\rm s}=\omega_{\Delta_{\rm s}}A^* \quad (A \in
H^{\infty}_{M_r}).
$$
It thus follows that $\overline{\omega_{\Delta_{\rm
s}}}f_1=A^*\Delta_{\rm s}^*f_1\in L^2_{\mathbb C^r}\ominus
H^2_{\mathbb C^r}$. \ Thus we have
$$
\overline{\theta}\breve{h}_0\Delta_1f_1=\breve{h}_0
\overline{\omega_{\Delta_1}}\Delta_1\overline{\omega_{\Delta_{\rm
s}}}f_1 \in L^2_{E}\ominus H^2_{E},
$$
which implies $P_+(\overline{\theta}\breve{h}_0\Delta_1f_1)=0$. \
Therefore,
$$
\aligned \sum_{j=1}^p
P_+(\breve{\eta}_j\gamma_j)&=P_+\bigl((\overline{\theta}\breve{h}_0+\breve{h}_1)(\theta \Omega g+\Delta_1f_1)\bigr)+\sum_{j=2}^p P_+(\breve{h}_j\Delta_1f_j)\\
&=P_+\bigl(\breve{h}_0 \Omega g\bigr)+P_+(\breve{h}_1\theta \Omega
g)+\sum_{j=1}^p P_+(\breve{h}_j\Delta_1f_j).
\endaligned
$$
Since $\Delta_1$ is inner,  it follows from (\ref{NJFVJGVKL}),
(\ref{jjjflsls}) and (\ref{cvfvfvfvfffffff})  that
$$
\aligned \Biggl|\Biggl|\sum_{j=1}^p P_+(\breve{\eta}_j\gamma_j)-\xi
\Biggr|\Biggr|_{L^2_{E}}&\leq\Biggl|\Biggl|P_+\bigl(\breve{h}_0\Omega
g\bigl)-\xi_0\Biggr|\Biggr|_{L^2_E}+\Biggl|\Biggl|~\sum_{j=1}^p
\Delta_1 P_+(\breve{h}_jf_j)-\Delta_1\xi_2\Biggr|\Biggr|_{L^2_E}\\
&<\epsilon.
\endaligned
$$
This proves (\ref{ddnnmbltlh}). \ Let $\Gamma:=\{\gamma_1, \gamma_2,
\cdots, \gamma_p\}$. \ It thus follows from Lemma
\ref{dcebsjdhvkrjw} and (\ref{ddnnmbltlh}) that
$$
 E^*_{\Gamma}=\bigvee\Bigl\{P_+(\breve{\eta}_j\gamma_j): \eta_j \in
H^{\infty}, \ j=1,2,\cdots p\Bigr\}=\mathcal H(\Delta),
$$
which implies that  $\mu_T\leq \mu_{T_s}$. \ For the reverse
inequality, let $q\equiv \mu_{T}<\infty$. \ Then there exists a set
$F\equiv \{f_1, f_2, \cdots, f_q\}\subseteq \mathcal H(\Delta)$ such
that $E_F^*=\mathcal H(\Delta)$. \ For each $j=1,2 \cdots q$, by
Lemma \ref{lemhbcsfvhs}, we can write
$$
f_j=g_j + \Delta_1 \gamma_j \qquad (g_j \in \mathcal H(\Delta_1), \
\gamma_j\in \mathcal H(\Delta_{\rm s})).
$$
Now we will show that
\begin{equation}\label{jefsnvdjb}
E_{\Gamma}^*=\mathcal H(\Delta_{\rm s}) \quad
(\Gamma\equiv\{\gamma_j: j=1,2,\cdots, q\}).
\end{equation}
Clearly, $E_{\Gamma}^* \subseteq \mathcal H(\Delta_{\rm s})$. \ On
the other hand, since $E_F^*=\mathcal H(\Delta)$ and $\mathcal
H(\Delta_1)$ is an invariant subspace for $S_E^*$, it follows from
Lemma \ref{dcebsjdhvkrjw}, (\ref{jjjflsls}) and (\ref{vbvbfbfbfb})
that
$$
\aligned \Delta_1 \mathcal H(\Delta_{\rm
s})&= \bigvee\Bigl\{P_{\Delta_1\mathcal H(\Delta_s)}\bigl(S_E^{*n}\Delta_1 \gamma_j\bigr): j=1,2,\cdots q, \ n=0,1,2\cdots\Bigr\}\\
&=\bigvee\Bigl\{P_{\Delta_1\mathcal H(\Delta_s)}(\breve{h}_j\Delta_1 \gamma_j): h_j \in H^{\infty},  \ j=1,2,\cdots p\Bigr\}\\
&=\bigvee\Bigl\{\Delta_1P_+(\breve{h}_j\gamma_j): h_j \in
H^{\infty},  \ j=1,2,\cdots p\Bigr\}\\
&= \Delta_1 E_\Gamma^*.
\endaligned
$$
This proves (\ref{jefsnvdjb}). \ Thus we have that $\mu_{T_s} \leq
q= \mu_{T}$. \ This proves (\ref{mainequality}). \ The last
assertion follows at once from (\ref{muTTTTe}) since $\Delta_{\rm
s}$ is square-inner. \ This completes the proof.
\end{proof}

\medskip

\medskip

\begin{cor}\label{cor110}
Suppose $\Delta$ is an $n\times r$ inner matrix function whose flip
$\breve\Delta$ is of bounded type. \ If  $T:=S^*_E|_{\mathcal
H(\Delta)}$, then $\mu_T\le r$.
\end{cor}

\begin{proof} It follows from Corollary \ref{cor512,222} and
Theorem \ref{hwysonghkkl}.
\end{proof}

%
%
%
%
%

\chapter{Miscellanea} \

In this chapter, by using the preceding results, we analyze left
and right coprimeness, the model operator, and an interpolation
problem for operator-valued functions.

\vskip 1cm

%
%
%
%

\noindent
{\bf \S\ 8.1. Left and right coprimeness} \

\bigskip
\noindent
In this section we consider conditions for the equivalence of left
coprime-ness and right coprime-ness.

If $\delta$ is a scalar inner function, a function $A \in
H^{\infty}(\mathcal B(E))$ is said to have a {\it
 scalar inner multiple} $\delta$  if there
exists a function $G \in H^\infty(\mathcal B(E))$ such that
$$
G A=A G=\delta I_E.
$$
We write $\hbox{mul}\,(A)$ for the set of all scalar inner multiples of $A$, and we define
\begin{equation}\label{81111}
m_{A}:=\hbox{g.c.d.}\, \bigl\{\delta: \delta \in
\hbox{mul}\,(A)\bigr\}.
\end{equation}
We note that if $\Delta$ is a two-sided inner function then by Lemma
\ref{thgg6688} and (\ref{cefcvwf}), the following are equivalent:
\medskip

\begin{itemize}
\item[(a)] $\Delta$ has a meromorphic pseudo-continuation of bounded type in $\mathbb{D}^e$;
\item[(b)] $\Delta$ has a scalar inner multiple.
\end{itemize}
Thus if $\Delta \in H^\infty(\mathcal B(E))$ is two-sided inner and
has a scalar multiple, then $m_{\Delta}$ defined in (\ref{81111})
coincides with the characteristic function of $\Delta$. \ This
justifies the use of the notation $m_A$ for (\ref{81111}).
\bigskip

On the other hand, we may ask:

\begin{q}\label{82222}
If $A\in H^\infty(\mathcal B(D,E))$ has a scalar inner multiple,
does it follow that $m_{A} \in \hbox{mul}\,(A)$ ?
\end{q}

\bigskip

If $A$ is two-sided inner with values in $\mathcal B(E)$, then the
answer to Question \ref{82222} is affirmative: indeed, by
(\ref{cefcvwf}),
$$
m_A H^2_E =\bigvee \bigl\{\delta H^2_E: \delta \in
\hbox{mul}\,(A)\bigr\} \subseteq \hbox{ker}\, H_{A^*},
$$
which implies, again by (\ref{cefcvwf}), that
\begin{equation}\label{83333}
m_A \in \hbox{mul}\,(A).
\end{equation}

\bigskip

\begin{lem}\label{lemdcwefvedf}
If $A \in H^{\infty}(\mathcal B(E))$ is an outer function having a
scalar inner multiple, then $1 \in \hbox{mul}\,(A)$, i.e., $A$ is
invertible in $H^{\infty}(\mathcal B(E))$.
\end{lem}

\begin{proof}
Suppose that $A \in H^{\infty}(\mathcal B(E))$ is an outer function
having a scalar inner multiple $\delta$. \ Then
\begin{equation}\label{vefgvtrhnb}
AG=GA=\delta I_E \quad \hbox{for some} \ G \in H^{\infty}(\mathcal
B(E)).
\end{equation}
We claim that
\begin{equation}\label{bghtnjh}
A H^2_E=\hbox{cl}\, AH^2_E.
\end{equation}
To see this, suppose $f\in \hbox{cl}\, A H^2_E$. \ Then there exists
a sequence $(g_n)$ in $H^2_E$ such that $||Ag_n-f||_{L^2_E}
\longrightarrow 0$. \ Thus we have that
\begin{equation}\label{84445}
||GAg_n-Gf||_{L^2_E}\leq ||G||_{\infty}||Ag_n - f||_{L^2_E}
\longrightarrow 0.
\end{equation}
It thus follows from (\ref{vefgvtrhnb}) and (\ref{84445}) that
$$
||g_n - \overline{\delta}Gf||_{L^2_E}=||\delta g_n-Gf||_{L^2_E}
=||GAg_n-Gf||_{L^2_E} \longrightarrow 0
$$
But since $H^2_E$ is a closed subspace of $L^2_E$, we have
$g\equiv\overline{\delta}Gf \in H^2_E$. \ Since $A\in
H^{\infty}(\mathcal B(E))$, it follows that
$$
||Ag_n-Ag||_{L^2_E}\leq ||A||_{\infty}||g_n -g||_{L^2_E}
\longrightarrow 0,
$$
which implies that $f=Ag \in A H^2_E$. \ This proves
(\ref{bghtnjh}). \ Since $A$ is an outer function, it follows from
(\ref{bghtnjh}) that
$$
AH^2_E=\hbox{cl}\, A H^2_E \supseteq \hbox{cl}\, A\mathcal
P_E=H^2_E,
$$
so that
\begin{equation}\label{8555}
A H^2_E=H^2_E.
\end{equation}
We thus have that
$$
H^2_E= (\overline{\delta}G)A H^2_E=\overline{\delta}G H^2_E,
$$
which implies that $G_1:=\overline{\delta}G \in H^{\infty}(\mathcal
B(E))$. \ It thus follows from (\ref{vefgvtrhnb}) that
$$
A G_1=G_1 A=I_E,
$$
which gives the result.
\end{proof}

\bigskip

We are tempted to guess that (\ref{8555}) holds for every outer
function $A$ in $H^\infty(\mathcal B(E))$. \ However, the following
example shows that this is not such a case.

\medskip

\begin{ex}
Let $A:=\hbox{diag}(\frac{1}{n}) \in H^{\infty}(\mathcal
B(\ell^2))$. \ Then $(1, \frac{1}{2}, \frac{1}{3}, \cdots)^t \notin
AH^2_{\ell^2}$, so that
$$
AH^2_{\ell^2} \neq H^2_{\ell^2}.
$$
Now we will show that $A$ is an outer function. \ Let $f \in
H^2_{\ell^2}$ and $\epsilon>0$ be arbitrary. \ Then we may write
$$
f(z)=\sum_{n=0}^{\infty}c_n z^n \quad  (c_n \in \ell^2).
$$
Thus there exists $M>0$ such that
\begin{equation}\label{efvgrthb}
\Biggl|\Biggl| \sum_{n=M}^{\infty}c_n
z^n\Biggr|\Biggr|_{L^2_{\ell^2}} <\frac{\epsilon}{2}.
\end{equation}
Write
$$
f_1(z):=\sum_{n=0}^{M-1}c_n z^n \quad \hbox{and} \quad
f_2(z):=\sum_{n=M}^{\infty}c_n z^n.
$$
For each $n=0,1,2,\cdots, M-1$, write
$$
c_n=\bigl(a_n^{(1)}, a_n^{(2)}, a_n^{(3)}, \cdots \bigr)^t \quad
(a_n \in \mathbb C).
$$
Then there exists $N>0$ such that
\begin{equation}\label{dcfvg}
\Biggl(\sum_{k=N+1}^{\infty}\bigl|a_n^{(k)}\bigr|^2\Biggr)^{\frac{1}{2}}
<\frac{\epsilon}{2M} \quad \hbox{for each} \ n=0,1,2,\cdots, M-1.
\end{equation}
Let
$$
p(z):=\sum_{n=0}^{M-1}b_n z^n, \quad \Bigl(b_n:=\bigl(a_n^{(1)},
2a_n^{(2)}, 3a_n^{(3)}, \cdots Na_n^{(N)}, 0 ,0,
\cdots\bigr)^t\Bigr).
$$
Then it follows from (\ref{efvgrthb}) and (\ref{dcfvg}) that
$$
\begin{aligned}
\bigl|\bigl|f(z)-(Ap)(z)\bigr|\bigr|_{L^2_{\ell^2}}&=\bigl|\bigl|f_1(z)+f_2(z)
-(Ap)(z)\bigr|\bigr|_{L^2_{\ell^2}}\\
&=\bigl|\bigl|f_1(z)-(Ap)(z)\bigr|\bigr|_{L^2_{\ell^2}}
+\bigl|\bigl|f_2(z)\bigr|\bigr|_{L^2_{\ell^2}}\\
&<\frac{\epsilon}{2M}M+\frac{\epsilon}{2}=\epsilon,
\end{aligned}
$$
which implies that $A$ is an outer function.
\end{ex}

\bigskip

\begin{lem}\label{fvgbhbnj}
If $A \in H^{\infty}(\mathcal B(E))$ has a scalar inner multiple,
then
\begin{itemize}
\item[(a)] $A^i$ is two-sided inner and has a scalar inner multiple with
$\hbox{mul}\,(A)\subseteq \hbox{mul}\,(A^i)$;
\item[(b)] $1 \in \hbox{mul}\,(A^e)$.
\end{itemize}
\end{lem}
\begin{proof}
Suppose that $A \in H^{\infty}(\mathcal B(E))$ has a scalar inner
multiple $\delta$, i.e., $\delta\in \hbox{mul}\,(A)$. \ Then there
exist a function  $G \in H^{\infty}(\mathcal B(E))$ such that
\begin{equation}\label{vefdcsvfnb}
AG=GA=\delta I_E.
\end{equation}
Thus $A(z)$ and $G(z)$ are invertible for almost all $z \in \mathbb
T$. \  Write
$$
A=A^iA^e \quad(\hbox{inner-outer factorization}).
$$
Since $A(z)$ is invertible for almost all $z \in \mathbb T$,
$A^i(z)$ is onto for almost all $z \in \mathbb T$, so that $A^i$ is
two-sided inner. \ Also  $A^e(z)$ is injective for almost all $z \in
\mathbb T$. \ By (\ref{vefdcsvfnb}),
$$
A^e(z)G(z)=\delta(z)(A^i(z))^*,
$$
which implies that $A^e(z)$ is onto, and hence invertible for almost
all $z \in \mathbb T$. \ Thus $(A^eG)(z)$ is invertible for almost
all $z \in \mathbb T$. \ We thus have that
$$
A^i(A^eG)=AG=\delta I_E=(A^eG)A^i,
$$
which implies that $A^i$ has a scalar inner multiple $\delta$, i.e.,
$\delta\in \hbox{mul}\,(A^i)$. \ This proves (a). \ Also observe
that
$$
(GA^i)A^e=\delta I_E=A^e(GA^i),
$$
which implies that $A^e$ has a scalar inner multiple. \ Thus by
Lemma \ref{lemdcwefvedf}, $1 \in \hbox{mul}\,(A^e)$. \ This proves
(b).
\end{proof}

\bigskip

\begin{lem} \label{hnjmkngbfvdc}
If $A \in H^{\infty}(\mathcal B(E))$ has a scalar inner multiple,
then
$$
\hbox{mul}\,(A)=\hbox{mul}\,(A^i).
$$
\end{lem}

\begin{proof}
In view of Lemma \ref{fvgbhbnj} (a), it suffices to show that
$\hbox{mul}\,(A^i)\subseteq \hbox{mul}\,(A)$. \ To see this, let
$\delta \in \hbox{mul}\,(A^i)$. \ Then
$$
A^iG=GA^i=\delta I_E \quad \hbox{for some} \ G \in
H^{\infty}(\mathcal B(E)).
$$
Put $G_0:=(A^e)^{-1}G$. \ Then by Lemma \ref{fvgbhbnj}, $G_0 \in
H^{\infty}(\mathcal B(E))$ and
$$
AG_0=A^iA^e(A^e)^{-1}G=\delta I_E.
$$
But since $A$ has a scalar inner multiple, $A(z)$ is invertible for
almost all $z \in \mathbb T$. \ Thus we have $\delta \in
\hbox{mul}\,(A)$. \ This proves $\hbox{mul}\,(A^i)\subseteq
\hbox{mul}\,(A)$. \ This completes the proof.
\end{proof}

\bigskip

The following corollary gives an affirmative answer to Question
\ref{82222}.

\medskip

\begin{cor}\label{fvgbhgj}
If $A \in H^{\infty}(\mathcal B(E))$ has a scalar inner multiple
then
$$
m_{A} \in \hbox{mul}\,(A).
$$
\end{cor}

\begin{proof}
By Lemma \ref{fvgbhbnj}, $A^i$ is two-sided inner. \ By (\ref{83333}),
$m_{A^i}\in\hbox{mul}\,(A^i)$. \ Thus it follows from Lemma
\ref{hnjmkngbfvdc} that
$$
m_{A}=m_{A^i} \in \hbox{mul}\,(A^i)=\hbox{mul}\,(A).
$$
\end{proof}

\medskip

The following lemma is elementary.

\medskip

\begin{lem}\label{lem2fgt.1}
Let $E$ be a  complex Hilbert space. \ If $\theta$ and $\delta$ are
scalar inner functions, then
$$
\hbox{left-g.c.d.}\,\{\theta I_E,  \delta I_E\}=
\hbox{g.c.d.}\,\{\theta, \delta\}I_E.
$$
\end{lem}

\begin{proof} Let
$$
\Omega:=\hbox{left-g.c.d.}\,\{\theta I_E,  \delta I_E\} \quad
\hbox{and} \quad \omega:= \hbox{g.c.d.}\,\{\theta, \delta\}.
$$
Then we can write
$$
\theta=\omega \theta_1 \quad \hbox{and} \quad \delta=\omega
\delta_1,
$$
where $\theta_1$ and $\delta_1$ are coprime inner functions. \ Thus
we have
$$
\Omega H^2_E=\theta H^2_E \bigvee \delta H^2_E=\omega  \theta_1
H^2_E \bigvee \omega \delta_1  H^2_E=\omega \Bigl( \theta_1 H^2_E
\bigvee\delta_1 H^2_E\Bigr) = \omega H^2_E,
$$
which implies that $\Omega= \omega I_E$. \ This completes the proof.
\end{proof}

\bigskip

\begin{lem}\label{lemedwqx5.2}
Let $A \in H^{\infty}(\mathcal B(E))$ have a scalar inner multiple
and $\theta$ be a scalar inner function. \ Suppose that $m_A$ is not
an inner divisor of $\theta$. \ If $\delta_0 \in \hbox{mul}\,(A)$ is
such that $A$ and $\omega I_E\equiv \hbox{g.c.d.}\{\theta,
\delta_0\}I_E$ are left coprime, then $\delta_0 \overline{\omega}
\in \hbox{mul}\,(A)$.
\end{lem}

\begin{proof}
Let $A \in H^{\infty}(\mathcal B(E))$ have a scalar inner multiple
and $\theta$ be a scalar inner function. \ Suppose that $m_A$ is not
an inner divisor of $\theta$. \ Then we should have $1 \notin
\hbox{mul}\,(A)$. \ Thus, by Lemma \ref{lemdcwefvedf}, $A$ is not an
outer function, so that $A^i$ is not a unitary operator. \ Let
$\delta_0 \in \hbox{mul}\,(A)$ be such that $A$ and $\omega
I_E\equiv \hbox{g.c.d.}\{\theta, \delta_0\}I_E$ are left coprime.\
Then, by Lemma \ref{lem2fgt.1}, we may write
\begin{equation}\label{jdfsfgj}
\theta=\omega \theta_1 \quad \hbox{and} \quad
\delta_0=\omega\delta_1,
\end{equation}
where $\theta_1$ and $\delta_1$ are coprime scalar inner functions.
\  On the other hand, since $\delta_0 \in \hbox{mul}\, (A)$, we have
that
\begin{equation}\label{vyhtcfrdsvuvfv}
\delta_0 I_E=G A =A G \quad\hbox{for some} \ G \in
H^{\infty}(\mathcal B(E)).
\end{equation}
Thus by (\ref{jdfsfgj}) and (\ref{vyhtcfrdsvuvfv}), we have that
$$
G (\overline{\omega} I_E) A=(\overline{\omega} I_E) GA=\delta_1 I_E
\in H^{\infty}(\mathcal B(E)),
$$
which implies that
\begin{equation}\label{dhfjgjkdxsj}
AH^2_E\subseteq \hbox{ker}\,H_{G (\overline{\omega} I_E)}\equiv
\Theta H^2_{E^{\prime}}.
\end{equation}
Thus $\Theta$ is a left inner divisor of $A$. \ Since also $\omega
H^2_E\subseteq \hbox{ker}\,H_{G \overline{\omega} I_E}= \Theta
H^2_{E^{\prime}}$, $\Theta$ is a left inner divisor of $\omega I_E$.
\  Thus $\Theta$ is a common left inner divisor of $A$ and $\omega
I_E$, so that, by our assumption, $\Theta$ is a unitary operator. \
Thus
$$
\hbox{ker}\, {H_{G\overline{\omega}I_E}}=\Theta
H^2_{E^{\prime}}=H^2_E,
$$
which implies that $\overline{\omega}I_E G \in H^{\infty}(\mathcal
B(E))$. \ On the other hand, by (\ref{jdfsfgj}) and
(\ref{vyhtcfrdsvuvfv}), we have
$$
\delta_1 I_E=(\overline{\omega}\delta_0)I_E=(\overline{\omega} I_E
G)A=A(\overline{\omega} I_E G),
$$
which implies that $\delta_1=\delta_0\overline{\omega} \in
\hbox{mul}\,(A)$. \ This completes the proof.
\end{proof}

\bigskip

We then have:

\medskip

\begin{thm}\label{lem5.2}
Let $A \in H^{\infty}(\mathcal B(E))$ and $\theta$ be a scalar inner
function. \ If $A$ has a scalar inner multiple, then the following
are equivalent:

\begin{itemize}

\item[(a)] $\theta$ and $m_{A}$ are coprime;

\item[(b)] $\theta I_E$ and $A$ are left coprime;

\item[(c)] $\theta I_E$ and $A$ are right coprime.

\end{itemize}

\end{thm}

\begin{proof} Let $A \in H^{\infty}(\mathcal B(E))$ have a scalar
inner multiple. \ Write
$$
A=A^iA^e \quad(\hbox{inner-outer factorization}).
$$

\medskip
(a) $\Rightarrow$ (b): Suppose that $\theta I_E$ and $A$ are not
left coprime. \ Then
$$
\theta H^2_E \bigvee A^i H^2_E \neq H^2_E.
$$
By Corollary \ref{fvgbhgj}, there exists $G\in H^{\infty}(\mathcal
B(E))$ such that $G A = A G=m_A I_E$. \ Thus we have that
$$
\hbox{left-g.c.d.}\, \{\theta I_E, m_A I_E\}H^2_E=\theta H^2_E
\bigvee A G H^2_E \subseteq \theta H^2_E \bigvee A^i H^2_E \neq
H^2_E,
$$
which implies that $\theta I_E$ and $m_A I_E$ are not left coprime.
\  Thus by Lemma \ref{lem2fgt.1}, $\theta$ and $m_A$ are not
coprime.

\medskip

(b) $\Rightarrow$ (a): Suppose that $\theta$ and $m_A$ are not
coprime. \ If $m_A$ is an inner divisor of $\theta$,  then by
Corollary \ref{fvgbhgj} and Lemma \ref{lem2fgt.1}, we may write
$$
\theta I_E=m_A \theta_1 I_E=A^i A^eG\theta_1 I_E \quad(G \in
H^{\infty}(\mathcal B(E)), \ \theta_1 \ \hbox{is a scalar inner}).
$$
Thus, $A^i$ is a common left inner divisor of $\theta I_E$ and $A$. \
If $A^i$ is a unitary operator, then $A$ is an outer function. \ It
thus follows from Lemma \ref{lemdcwefvedf} that $m_A=1$, so that
$\theta$ and $m_A$ are coprime, a contradiction. \ Therefore $A^i$
is not a unitary operator, and hence $\theta I_E$ and $A$ are not
left coprime. \ Suppose instead that $m_A$ is not an inner divisor of
$\theta$. \ Write $\omega\equiv \hbox{g.c.d.}\{\theta, m_A\} \neq
1$. \ We then claim that
\begin{equation}\label{5.sdssd4}
A  \ \hbox{and} \ \omega I_E \ \hbox{are not left coprime}.
\end{equation}
Towards (\ref{5.sdssd4}), we assume to the contrary that $A$ and
$\omega I_E$ are left coprime. \ Then it follows from  Corollary
\ref{fvgbhgj} and Lemma \ref{lemedwqx5.2} that $
\overline{\omega}m_A \in \hbox{mul}\, (A)$, which contradicts the
definition of $m_A$. \ This proves (\ref{5.sdssd4}). \ But since
$\omega$ is an inner divisor of $\theta$, it follows from Lemma
\ref{lem2fgt.1} that $A$ and $\theta I_E$ is not left coprime.

\medskip

(b) $\Leftrightarrow$ (c). \ Since $\delta\in \hbox{mul}\,(A)$ if and
only if $\widetilde{\delta}\in \hbox{mul}\,(\widetilde{A})$, it
follows that $\widetilde{m_{A}}=m_{\widetilde{A}}$. \ It thus
follows from (a) $\Leftrightarrow$ (b). \ This completes the proof.
\end{proof}

\bigskip

\begin{cor}\label{lemed3w5.2}
Let $\Delta$ be an inner function with values in $\mathcal B(D, E)$
and $\theta$ be a scalar inner function. \ If $\Delta$ has a
meromorphic pseudo-continuation of bounded type in $\mathbb{D}^e$, then the
following are equivalent:

\begin{itemize}

\item[(a)] $\theta$ and $\omega_{\Delta}$ are coprime;

\item[(b)] $\theta I_E$ and $[\Delta, \Delta_c]$ are left coprime;

\item[(c)] $\theta I_E$ and $[\Delta, \Delta_c]$ are right coprime.

\end{itemize}
\end{cor}

\begin{proof} Suppose that $\Delta$ has a
meromorphic pseudo-continuation of bounded type in $\mathbb{D}^e$. \ Then by
Lemma \ref{thgg668}, $\breve{\Delta}$ is of bounded type, so that by
Corollary \ref{thmthjdkfjkf}, $[\Delta, \Delta_c]$ is two-sided
inner. \ Thus the result follows from Theorem \ref{lem5.2} and Lemma
\ref{thnfhjjjfhbbbbskl}.
\end{proof}

\bigskip

\begin{ex} Let
$$
\Delta:=\begin{bmatrix} b_{\alpha}&0\\0&b_{\beta}\\0&0\end{bmatrix}
\quad (\alpha \neq 0, \beta \neq 0).
$$
Then $zI_3$ and $\Delta$ are not left coprime because $zH^2_{\mathbb
C^3}\bigvee \Delta H^2_{\mathbb C^2}\ne H^2_{\mathbb C^3}$. \ But
$zI_3$ and $[\Delta, \Delta_c]$ are left coprime, so that, by
Corollary \ref{lemed3w5.2}, $z$ and $\omega_{\Delta}$ are coprime.\
Indeed, we note that $\hbox{ker}\, H_{\Delta^*}=[\Delta,
\Delta_c]H^2_{\mathbb C^3}$, and hence
$\omega_{\Delta}=b_{\alpha}b_{\beta}$.
\end{ex}

\bigskip

The following example shows that if the condition ``$A$ has a scalar
inner multiple" is dropped in Theorem \ref{lem5.2}, then Theorem
\ref{lem5.2} may fail.

\medskip

\begin{ex} Let
$$
\Delta(z)=S_E\quad(E=\ell^2(\mathbb{Z}_+))
$$
Then $\Delta$ is an inner function (not two-sided inner, an
isometric operator) with values in $\mathcal B(E)$. \ For $f \in
H^2_E$, we can write
$$
f(z)=\sum_{n=0}^{\infty}a_n z^n \quad (a_n \in E).
$$
We thus have that
$$
(\widetilde\Delta f)(z)=S^*\Bigl(\sum_{n=0}^{\infty}a_n
z^n\Bigr)=\sum_{n=0}^{\infty}(S^*a_n)z^n.
$$
Thus $\widetilde\Delta H^2_E=H^2_E$, so that $\Delta$ and $\theta I_E$ are
right coprime for all scalar inner function $\theta$. \ Let
$\theta(z)=z \theta_1$ ($\theta_1$ a scalar inner). \ Then
$$
(\Delta f)(z)=S\Bigl(\sum_{n=0}^{\infty}a_n
z^n\Bigr)=\sum_{n=0}^{\infty}(Sa_{n})z^{n}.
$$
We thus have
$$
\Delta H^2_E \bigvee \theta H^2_E=\Delta H^2_E \bigvee z\theta_1
H^2_E\subseteq \Delta H^2_E \bigvee z H^2_E \neq H^2_E,
$$
which implies that $\theta I_E$ and $\Delta$ are not left coprime.\
Note that $\Delta$ has no scalar inner multiple.

On the other hand, since $\hbox{ker}\, H_{\Delta^*}=H^2_E$, we have
$\omega_{\Delta}=1$. \ Thus, it follows from Corollary
\ref{lemed3w5.2} that $\theta I_E$ and $[\Delta, \Delta_c]$ are
left (and right) coprime for all scalar inner function $\theta$.
\qed
\end{ex}

\bigskip

\begin{lem}\label{lem83333}
If $\Delta\in H^\infty_{M_n}$ is an inner function then
\begin{equation}\label{mmnbfdss}
\hbox{$\theta$ and $m_\Delta$ are coprime}\Longleftrightarrow
\hbox{$\theta$ and $\det \Delta$ are coprime.}
\end{equation}
\end{lem}

\begin{proof}
If $\Delta\in H^\infty_{M_n}$ is inner, then $m_{\Delta}\in
\hbox{mul}\,(\Delta)$, so that we may write
$$
\hbox{$m_{\Delta}I_n=\Delta G$ for some inner function $G\in
H^{\infty}_{M_n}$.}
$$
Thus, $\det\Delta\, \hbox{det}G = m_{\Delta}^n$. \ If $\theta$ and
$m_{\Delta}$ are coprime, then $\theta$ and $m_{\Delta}^n$ are
coprime, so that $\theta$ and $\det\Delta$ are coprime. \
Conversely, suppose that $\theta$ and $\det\Delta$ are coprime. \
Since $(\det\Delta) I_n=(\hbox{adj}\Delta) \Delta$, it follows that
$\det\Delta\in \hbox{mul}\,(\Delta)$. \ Thus, $m_{\Delta}$ is an
inner divisor of $\det\Delta$, and hence $\theta$ and $m_{\Delta}$
are coprime. \ This proves (\ref{mmnbfdss}).
\end{proof}

\bigskip

We can recapture \cite[Theorem 4.16]{CHL3}.

\medskip

\begin{cor}\label{coredcrwwww.2}
Let $A\in H^{2}_{M_n}$ and $\theta$ be a scalar inner function. \
Then the following are equivalent:
\begin{itemize}
\item[(a)] $\theta$ and $\hbox{det}A$ are coprime;
\item[(b)] $\theta I_n$ and $A$ are left coprime;
\item[(c)] $\theta I_n$ and $A$ are right coprime.
\end{itemize}
\end{cor}

\begin{proof}
If $\Delta\in H^\infty_{M_n}$ is inner then by Theorem \ref{lem5.2}
and Lemma \ref{lem83333}, we have
\begin{equation}\label{814814}
\hbox{$\theta I_n$ and $\Delta$ are left coprime}\Longleftrightarrow
\hbox{$\theta$ and $\det \Delta$ are coprime.}
\end{equation}
We now write
$$
A=A^iA^e \quad(\hbox{inner-outer factorization}).
$$
Now we will show that if (b) or (c) holds, then $A^i$ is two-sided
inner: indeed if (b) or (c) holds, then by \cite[Lemma 4.15]{CHL3},
$\hbox{det}A \neq 0$, so that $A(z)$ is invertible, and hence
$A^i(z)$ is onto for almost all $z \in \mathbb T$. \ Thus $A^i$ is
two-sided inner. \ Then by the Helson-Lowdenslager Theorem (cf.
\cite[p.22]{Ni1}) we have that
$$
\hbox{det}A=\hbox{det}{A^i}\cdot \hbox{det}A^e
\quad(\hbox{inner-outer factorization})
$$
It thus follows from (\ref{814814}) that
$$
\begin{aligned}
\theta I_n \ \hbox{and} \ A \ \hbox{are left coprime}&
\Longleftrightarrow \theta I_n \ \hbox{and} \ A^i \ \hbox{are left
coprime}\\
& \Longleftrightarrow \theta  \ \hbox{and}  \
\hbox{det}A^i \ \hbox{are coprime}\\
&\Longleftrightarrow \theta  \ \hbox{and}  \
\hbox{det}A \ \hbox{are coprime}\\
\end{aligned}
$$
For right coprime-ness, we apply the above result and the fact that $\hbox{det}\widetilde{A}=\widetilde{\hbox{det}A}$.
\end{proof}

\vskip 1cm

%
%
%
%

\noindent {\bf \S\ 8.2. The model operator} \

\bigskip
\noindent
We recall that the model theorem (p. \pageref{MT}) states that if
$T\in\mathcal {B(H)}$
is a contraction such that $\lim_{n\to\infty} T^n x=0$ for each $x\in \mathcal H$
(i.e., $T\in C_{0\, \bigcdot}$), then there exists a unitary imbedding $V:\mathcal H\to H^2_E$
with $E:=\hbox{cl\,ran}(I-TT^*)$ such that
$V\mathcal H=\mathcal H(\Delta)$ for some inner function $\Delta$ with values in
$\mathcal B(E^\prime, E)$ and
\begin{equation}\label{MT2}
T=V^*\Bigl(S^*_E|_{\mathcal H(\Delta)}\Bigr)V.
\end{equation}
We may now ask what is a necessary and sufficient condition for
$\dim\,E^\prime<\infty$ in the Model Theorem. \ In this section, we
give a necessary condition for the finite-dimensionality of
$E^\prime$.

For an inner function $\Delta$ with values in $\mathcal B(E^\prime, E)$, define
\begin{equation} \label{CHL}
H_0:=\bigl\{f\in\mathcal H(\Delta): \ \lim_{n\to\infty} P_{\mathcal
H(\Delta)} S_E^n f=0\bigr\}.
\end{equation}
Then $H_0$ is a closed subspace of $\mathcal H(\Delta)$ and in this case,
write
$$
E_0(\Delta):=\mathcal H(\Delta)\ominus H_0.
$$
Then $E_0(\Delta)$ is an invariant subspace of $S_E^*$, so that
there exists an inner function $\Delta^s\in H^\infty(\mathcal B(E_1, E))$
such that
\begin{equation}\label{821}
E_0(\Delta)=\mathcal H(\Delta^s).
\end{equation}
We then have:

\medskip

\begin{lem}\label{lem821}
Let $\Delta$ be an inner function with values in $\mathcal
B(E^\prime, E)$. \ Then
$$
\Delta=\Delta^s\Delta_1
$$
for some two-sided inner function $\Delta_1$ with values in
$\mathcal B(E^\prime, E_1)$.
\end{lem}

\begin{proof}
Observe that $H^2_E=\Delta H^2_{E^\prime}\oplus E_0(\Delta)\oplus
H_0$. \ Thus,
$$
\Delta H^2_{E^\prime} \subseteq H^2_E\ominus E_0(\Delta)=\Delta^s H^2_{E_1},
$$
which implies that $\Delta=\Delta^s \Delta_1$ for some inner
function $\Delta_1$ with values in $\mathcal B(E^\prime, E_1)$. \ We
must show that $\Delta_1$ is two-sided. \ We first claim that
\begin{equation}\label{822}
f\in \Delta^s H^2_{E_1} \Longleftrightarrow
||f||_{L^2_E}=||\Delta^*f||_{L^2_{E^\prime}}:
\end{equation}
indeed, since $\lim_{n\to\infty}||(I_E-P_+) \Delta^* S_E^n f||_{L^2_{E^{\prime}}}=0$
for each $f\in H^2_E$,
a straightforward calculation shows that
$$
\lim_{n\to\infty} ||P_{\mathcal H(\Delta)} S_E^n f||^2_{L^2_E}=||f||^2_{L^2_E}
-||\Delta^* f||^2_{L^2_{E^{\prime}}},
$$
giving (\ref{822}). \ Thus for all $x \in E_1$ with $||x||=1$,
$$
1=||\Delta^s x||_{L^2_E}=||\Delta^*\Delta^s x||_{L^2_{E^\prime}}
=||\Delta_1^* x||_{L^2_{E^\prime}},
$$
which says that
$$
\int_{\mathbb T}||\Delta_1^*(z) x||^2dm(z)=1.
$$
But since $||\Delta_1^*(z)x||\le 1$, it follows that
$||\Delta_1^*(z)x||=1$ a.e. on $\mathbb T$, so that $\Delta_1^*(z)$
is isometry for almost all $z \in \mathbb T$ and therefore
$\Delta_1$ is two-sided inner. \  This completes the proof.
\end{proof}

\bigskip

We then have:

\medskip

\begin{thm}\label{thm822}
Let $T\in\mathcal B(H)$ be a contraction such that
$\lim_{n\to\infty} T^n x=0$ for each $x\in  H$ and have a
characteristic function $\Delta$ with values in $\mathcal
B(E^\prime, E)$. \ Then,
$$
\hbox{sup}_{\zeta \in \mathbb D}\dim\bigl\{f(\zeta) : f \in
H_0\bigr\}\leq \dim E^{\prime},
$$
where $H_0$ is defined by (\ref{CHL}). \ In particular, if $\dim
E^{\prime}<\infty$, then $ \hbox{max}_{\zeta \in \mathbb D}\dim
\bigl\{f(\zeta) : f \in H_0\bigr\}$ is finite.
\end{thm}

\begin{proof} It follows from (\ref{821}) and  Lemma \ref{lem821} that
$$
H_0=\mathcal H(\Delta)\ominus E_0(\Delta)=\mathcal H(\Delta)\ominus
\mathcal H(\Delta^s) \subseteq \Delta^s H^2_{E^{\prime}}.
$$
We thus have
$$
\hbox{sup}_{\zeta \in \mathbb D}\dim\bigl\{f(\zeta) : f \in
H_0\bigr\}\leq \hbox{sup}_{\zeta \in \mathbb
D}\dim\bigl\{\Delta^s(\zeta)g(\zeta) : g \in
H^2_{E^{\prime}}\bigr\}=\dim E^{\prime}.
$$
\end{proof}

\vskip 1cm

%
%
%
%

\noindent
{\bf \S\ 8.3. An interpolation problem} \

\bigskip
\noindent In the literature, many authors have considered the
special cases of the following (scalar-valued or operator-valued)
interpolation problem (cf. \cite{Co1}, \cite{CHL2}, \cite{CHL3},
\cite{FF}, \cite{Ga}, \cite{Gu}, \cite{GHR}, \cite{HKL}, \cite{HL1},
\cite{HL2}, \cite{NT}, \cite{Zh}).

\begin{prob}
For $\Phi\in
L^\infty(\mathcal B(E))$, when does there exist a function $K\in
H^\infty(\mathcal B(E))$ with $||K||_{\infty}\leq 1$ satisfying
\begin{equation}\label{interpolation}
\Phi-K\Phi^*\in H^\infty (\mathcal B(E)) \,?
\end{equation}
\end{prob}

If $\Phi$ is a matrix-valued rational function, this question
reduces to the classical Hermite-Fej\' er interpolation problem.

For notational convenience, we write, for $\Phi\in
L^\infty(\mathcal B(E))$,
$$
\mathcal C(\Phi)\label{cphi}:=\Bigl\{ K\in H^\infty(\mathcal B(E)):
\ \Phi-K\Phi^*\in H^\infty(\mathcal B(E)) \Bigr\}.
$$
We then have:

\medskip

\begin{thm}\label{thm41}
Let $\Phi\equiv \breve{\Phi}_-+\Phi_+\in L^\infty(\mathcal B(E))$. \
If $\mathcal C(\Phi)$ is nonempty then
$$
\ker H_{\breve{\Phi}_+}^*\subseteq \ker H_{\Phi_-^*}^*.
$$
In particular,
$$
\hbox{nc}\{\Phi_+\}\le \hbox{nc}\{\widetilde{\Phi_-}\}.
$$
\end{thm}

\begin{proof}
Suppose $\mathcal C(\Phi)\ne \emptyset$. \ Then there exists a
function  $K\in H^\infty (\mathcal B(E))$ such that $\Phi-K\Phi^*\in
H^\infty(\mathcal B(E))$, then
 $H_{\Phi}=T_{\widetilde K}^* H_{\Phi^*}$, which
implies that $\ker H_{\Phi^*}\subseteq \ker H_{\Phi}$. \ But since
$\Phi\equiv \breve{\Phi}_-+\Phi_+\in L^\infty(\mathcal B(E))$, it
follows that
$$
H_{\Phi^*}=H_{\Phi_+^*}
=H_{\breve{\Phi}_+}^*
\quad \hbox{and} \quad
H_{\Phi}=H_{\breve{\Phi}_-}=H_{\Phi_-^*}^*.
$$
We thus have 
$$
\ker H_{\breve{\Phi}_+}^*\subseteq \ker H_{\Phi_-^*}^*.
$$
On the other hand, it follows from Lemma \ref{kerhadjoint} that
\begin{equation}\label{llllxmf}
\Omega H^2_{E^\prime}=\ker H_{\breve{\Phi}_+}^*\subseteq \ker
H_{\Phi_-^*}^*=\ker H_{\breve{\widetilde{\Phi}}_-}^* =\Delta
H^2_{E^{\prime\prime}}
\end{equation}
for some inner functions $\Omega$ and $\Delta$ with values in
$\mathcal B(E^\prime, E)$ and $\mathcal B(E^{\prime\prime},E)$,
respectively. \ Thus $\Delta$ is a left inner divisor of $\Omega$,
so that we have $\dim E^\prime\le \dim E^{\prime\prime}$, which
implies, by Theorem \ref{thm7566}, that $\hbox{nc}\{\Phi_+\}\le
\hbox{nc}\{\widetilde{\Phi_-}\}$.
\end{proof}

\bigskip

\begin{cor}
Let $\Phi\equiv \breve{\Phi}_-+\Phi_+\in L^\infty(\mathcal B(E))$
and $\mathcal C(\Phi)\ne \emptyset$. \ If $\breve\Phi_+$ is of
bounded type, then $\Phi_-^*$ is of bounded type.
\end{cor}

\begin{proof} Suppose that $\Phi\equiv \breve{\Phi}_-+\Phi_+\in L^\infty(\mathcal
B(E))$. \ Then by Lemma \ref{corfgghh2.9}, $\Phi^*=
(\breve{\Phi}_-)^*+(\Phi_+)^*\in L^\infty(\mathcal B(E))$. \ Thus
$(\breve{\Phi}_-)^*$  is a strong $L^2$-function and so is
$\Phi_-^*$. \ Assume that $\mathcal C(\Phi)\ne \emptyset$ and
$\breve\Phi_+$ is of bounded type. \ Then it follows from Theorem
\ref{thm41} and Lemma \ref{kerhadjoint} that
\begin{equation}\label{llllxmf2}
\Omega H^2_{E}=\ker H_{\breve{\Phi}_+}^*\subseteq \ker
H_{\Phi_-^*}^* =\Delta H^2_{E^{\prime\prime}}
\end{equation}
for some two-sided inner function $\Omega$ with values in $\mathcal
B(E)$ and an inner function $\Delta$ with values in $\mathcal
B(E^{\prime\prime},E)$. \ Thus, $\Delta$ is a left inner divisor of
$\Omega$ and hence, by Lemma \ref{rem.sdcfcfc}, $\Delta$ is
two-sided inner, so that $\Phi_-^*$ is of bounded type.
\end{proof}

\newpage

%
%
%
%
%

\chapter{Some unsolved problems} \

In this paper we have explored the Beurling-Lax-Halmos Theorem and
have tried to answer several outstanding questions. \ In this process, we have gotten
interesting results on a canonical decomposition of strong
$L^2$-functions, a connection between the Beurling degree and the
spectral multiplicity, and the multiplicity-free model operators. \
However there are still open questions in which we are interested. \
In this chapter, we pose several unsolved problems.

\bigskip

\noindent {\bf \S\ 9.1. The Beurling degree of an inner matrix
functions} \

\bigskip
\noindent The theory of spectral multiplicity for operators of the
class $C_0$ has been well developed (see \cite[Appendix 1]{Ni1},
\cite{SFBK}). \ For an inner matrix function $\Delta \in
H^\infty_{M_N}$ and $k=0,1,\cdots, N$, let
\begin{equation}\label{jcencisfirs}
\delta_k:=\hbox{g.c.d.}\,\{\hbox{all inner parts of the minors of
order $N-k$ of $\Delta$}\}.
\end{equation}
Then it is well-known that if $T\in C_0$ with characteristic
function $\Delta \in  H^\infty_{M_N}$, then
\begin{equation}\label{computeBeurling}
\mu_T=\min\, \bigl\{k:\ \delta_k =\delta_{k+1}\bigr\}.
\end{equation}
In fact, the proof for ``$\ge$" in (\ref{computeBeurling}) is not
difficult. \  But the proof for ``$\le$" is so complicated. \
However, Theorem \ref{maintheorem_sm} gives a simple proof for
``$\le$" in (\ref{computeBeurling}) with the aid of the
Moore-Nordgren Theorem. \ To see this, we recall that  for an inner
function $\Delta_k$ ($k=1,2$) with values in $M_N$, $\Delta_1$ and
$\Delta_2$ are called quasi-equivalent if there exist functions
$X,Y\in H^\infty_{M_N}$ such that $X\Delta_1=\Delta_2Y$ and such
that the inner parts $(\hbox{det}\, X)^i$ and $(\hbox{det}\,Y)^i$ of
the corresponding determinants are coprime to
$(\hbox{det}\,\Delta_k)^i$ ($k=1,2$).

The following theorem shows that the spectral multiplicity of
$C_0$-operators with square-inner characteristic functions can be
computed by studying diagonal characteristic functions (cf.
\cite{No}, \cite{MN}, \cite{Ni1}):

\medskip

\noindent {\bf Nordgren-Moore Theorem}. \
\begin{itemize}
\item[(a)]
Let $\Delta_k$ ($k=1,2$) be an inner function with values in $M_N$
and let $T_k:=P_{\mathcal H(\Delta)}S_{\mathbb C^N}\vert_{\mathcal
H(\Delta_k)}$ ($k=1,2$). \ If $\Delta_1$ and $\Delta_2$ are
quasi-equivalent then $\mu_{T_1}=\mu_{T_2}$.
\item[(b)]
Let $\Delta$ be an inner function with values in $M_N$. \ Then
$\Delta$ is quasi-equivalent to a unique diagonal inner function
$$
\hbox{diag}\,(\delta_0/\delta_1, \delta_1/\delta_2,
\cdots, \delta_{N-1}/\delta_N).
$$
\end{itemize}

\bigskip

By the Nordgren-Moore Theorem (a) and Theorem  \ref{maintheorem_sm},
we can see that
if $\Delta_1$ and $\Delta_2$ are quasi-equivalent square inner matrix functions
then
\begin{equation}\label{quasidegree}
\hbox{deg}_B(\widetilde{\Delta}_1)=\hbox{deg}_B(\widetilde{\Delta}_2).
\end{equation}

\medskip

We now have:

\begin{prop}\label{prop6.1}
If $\Delta$ is an $N\times N$ square-inner matrix function then
\begin{equation}\label{concludein}
\deg_B(\Delta) \le \min\, \bigl\{k:\ \delta_k =\delta_{k+1}\bigr\}.
\end{equation}
\end{prop}

\begin{proof} Let
$m:=\min\, \bigl\{k:\ \delta_k =\delta_{k+1}\bigr\}$. \ Then by the
Nordgren-Moore Theorem, $\Delta$ is quasi-equivalent  to
$\Theta\equiv \hbox{diag}\,(\delta_0/\delta_1, \cdots,
\delta_{m-1}/\delta_m, 1, \cdots, 1)$. \ We now take
$$
\Phi:=\begin{bmatrix}
\delta_0/\delta_1&0&\cdots&0\\
0&\delta_1/\delta_2&&\vdots\\
\vdots&&\ddots&0\\
0&\cdots&0&\delta_{m-1}/\delta_m\\
0&\cdots&0&1\\
&\vdots&&\vdots\\
0&\cdots&0&1
\end{bmatrix}\in H^{\infty}_{M_{N\times m}}.
$$
Then a direct calculation shows that
$$
\hbox{ker}\, H_{\Phi^*}=\left(\sum_{k=1}^m \bigoplus
(\delta_{k-1}/\delta_k)H^2\right)\bigoplus H^2_{\mathbb C^{N-m}}=
\Theta H^2_{\mathbb C^n}.
$$
It thus follows from (\ref{main2k}) and (\ref{quasidegree}) that
$\deg_B(\Delta)=\deg_B(\Theta)\le m$.
\end{proof}

\bigskip

\begin{cor}
If $\Theta$ is a diagonal inner matrix function of the form
$\Theta:=\hbox{diag}(\theta_1,\cdots, \theta_N)$ (where each
$\theta_i$ is a scalar inner function) then
$$
\hbox{deg}_B(\Theta)=\max\hbox{card}
\Bigl\{\sigma: \ \sigma\subseteq\{1,\cdots,N\}, \ \hbox{g.c.d.}
\{\theta_i: i\in\sigma\} \ne 1\Bigr\}.
$$
\end{cor}

\begin{proof}
This follows at once from (\ref{computeBeurling}) and Theorem
\ref{maintheorem_sm}.
\end{proof}

\bigskip

Now Proposition \ref{prop6.1} together with Theorem
\ref{maintheorem_sm} gives a simple proof for ``$\le$"  in
(\ref{computeBeurling}). \ Consequently, in (\ref{concludein}), we
may take ``$=$" in place of ``$\le$". \ However we were unable to
derive a similar formula to (\ref{concludein}) for {\it non-square}
inner matrix function. \ Thus we would like to pose:

\medskip

\begin{prob}\label{prob111}
If $\Delta$ is an $n\times m$ inner matrix function, describe
$\hbox{deg}_B(\Delta)$ in terms of its entries (e.g., minors).
\end{prob}

\vskip 1cm

%
%
%
%

\noindent {\bf \S\ 9.2. Spectra of model operators} \

\bigskip
\noindent
We recall that
if $\theta$ is a scalar inner function, then we may write
$$
\theta(\zeta)=B(\zeta) \hbox{exp}\left(-\int_{\mathbb T}
\frac{z+\zeta}{z-\zeta} d\mu(z)\right),
$$
where $B$ is a Blaschke product and $\mu$ is a singular measure on
$\mathbb T$ and that the spectrum, $\sigma(\theta)$, of $\theta$ is defined by
$$
\sigma(\theta):=\Bigl\{\lambda\in\hbox{cl}\,\mathbb D:\
\frac{1}{\theta}\ \hbox{can be continued analytically into a neighborhood of
$\lambda$}\Bigr\}.
$$
Then it was (\cite[p.63]{Ni1}) known that the
spectrum $\sigma(\theta)$ of $\theta$ is given
by
\begin{equation}\label{spect-measure}
\sigma(\theta)=\hbox{cl}\, \theta^{-1}(0)\, \bigcup \,
\hbox{supp}\,\mu.
\end{equation}
It was also (cf. \cite[p.72]{Ni1}) known that if $T \equiv
P_{\mathcal H(\Delta)}S_E\vert_{\mathcal H(\Delta)}\in C_0$, then
\begin{equation}\label{spect_C0}
\sigma(T)=\sigma(m_{\Delta}).
\end{equation}
In view of (\ref{spect_C0}), we may ask what is the spectrum of the
model operator $S_{E}^*|_{\mathcal H(\Delta)}$ ? \ Here is an
answer.

\medskip

\begin{prop}\label{propA}
Let $T:=S_{E}^*|_{\mathcal H(\Delta)}$ for an inner function
$\Delta$ with values in $\mathcal B(D,E)$. \ If $\Delta$ has a
meromorphic pseudo-continuation of bounded type in $\mathbb{D}^e$ and
$\omega_\Delta$ is the pseudo-characteristic scalar inner function
of $\Delta$, then
\begin{equation}\label{spect_C0.}
\overline{\sigma(\omega_{\Delta})} \subseteq \sigma(T).
\end{equation}
\end{prop}

\begin{proof}
If $\Delta_c$ is the complementary factor, with values in $\mathcal
B(D^\prime, E)$, of $\Delta$, then by the proof of Lemma
\ref{thnfhjjjfhbbbbskl}, $[\Delta, \Delta_c]$ is two-sided inner and
has a meromorphic pseudo-continuation of bounded type in $\mathbb{D}^e$. \
Thus, by  Proposition \ref{C0condition}, $S_{E}^*|_{\mathcal
H([\Delta, \Delta_c])}$ belongs to $C_0$. \ Then by the Model
Theorem, we have
$$
S_{E}^*|_{\mathcal H([\Delta, \Delta_c])}\cong P_{\mathcal
H(\widetilde{[\Delta, \Delta_c]})}S_E|_{\mathcal
H(\widetilde{[\Delta, \Delta_c]})}.
$$
It thus follows from Lemma \ref{thnfhjjjfhbbbbskl} and
(\ref{spect_C0}) that
\begin{equation}\label{nhnroklkvfld}
\sigma(S_{E}^*|_{\mathcal H([\Delta,
\Delta_c])})=\sigma(m_{\widetilde{[\Delta,
\Delta_c]}})=\sigma(\widetilde{\omega_{\Delta}})=\overline{\sigma(\omega_{\Delta})}.
\end{equation}
On the other hand, observe
$$
[\Delta, \Delta_c] H^2_{D\oplus D^\prime}=\Delta H^2_D\oplus
\Delta_c H^2_{D^\prime},
$$
and hence
$$
\mathcal H(\Delta)=\mathcal H([\Delta, \Delta_c])\oplus \Delta_c
H^2_{D^\prime}.
$$
Thus we may write
\begin{equation}\label{Tmat}
T=\begin{bmatrix} T_1&\ast\\ 0&T_2\end{bmatrix}:
\begin{bmatrix}\mathcal H([\Delta, \Delta_c])\\ \Delta_c H^2_{D^\prime}\end{bmatrix}
\to \begin{bmatrix}\mathcal H([\Delta, \Delta_c])\\ \Delta_c
H^2_{D^\prime}\end{bmatrix}.
\end{equation}
Note that $T_1=S_{E}^*|_{\mathcal H([\Delta, \Delta_c])}$. \ Since
by (\ref{spect-measure}) and (\ref{nhnroklkvfld}), $\sigma (T_1)$
has no interior points, so that $\sigma(T_1)\cap \sigma(T_2)$ has no
interior points. \ Thus we have $\sigma(T)=\sigma(T_1)\cup
\sigma(T_2)$ because in the Banach space
setting, the passage from $\sigma\left(\begin{smallmatrix} A&C\\
0&B\end{smallmatrix}\right)$ to $\sigma(A)\cup\sigma(B)$ is the
filling in certain holes in $\sigma\left(\begin{smallmatrix} A&C\\
0&B\end{smallmatrix}\right)$, occurring in $\sigma(A)\cap \sigma(B)$
(cf. \cite{HLL}). \ Therefore, by (\ref{nhnroklkvfld}), we have $
\overline{\sigma(\omega_{\Delta})} \subseteq \sigma(T)$.
\end{proof}

\bigskip

We would like to pose:

\medskip

\begin{prob}\label{9222}
If $T:=S_{E}^*|_{\mathcal H(\Delta)}$ for an inner function
$\Delta$ having a
meromorphic pseudo-continuation of bounded type in $\mathbb{D}^e$, describe
the spectrum of $T$ in terms of the pseudo-characteristic scalar
inner function of $\Delta$.
\end{prob}

\vskip 1cm

%
%
%
%
%

\noindent {\bf \S\ 9.3. The spectral multiplicity of model operators} \

\bigskip
\noindent
It was known (cf. \cite[p. 41]{Ni1}) that if
$T:=S_{E}^*|_{\mathcal H(\Delta)}$
for an inner function $\Delta$ with values in $\mathcal B(E^\prime,E)$,
with $\dim\, E^\prime<\infty$, then
$\mu_T\le \dim\, E^\prime +1$.
Theorem \ref{hwysonghkkl} says that if $\Delta$
has a meromorphic
pseudo-continuation of bounded type in $\mathbb{D}^e$,
then
$$
\mu_T\le \dim\, E^\prime.
$$
However we were unable to find an example showing that Theorem
\ref{hwysonghkkl} may fail if the condition ``$\Delta$ has a
meromorphic pseudo-continuation of bounded type in $\mathbb{D}^e$" is dropped.
\ Thus we would like to pose:

\medskip

\begin{prob}\label{9333}
Find an example of the operator $T\equiv S_{E}^*|_{\mathcal
H(\Delta)}$ for an inner function $\Delta$ with values in $\mathcal
B(E^\prime,E)$, with $\dim\, E^\prime<\infty$, satisfying
$$
\mu_T= \dim\, E^\prime+1.
$$
\end{prob}

\vskip 1cm
%
%
%
%
\noindent
{\bf \S\ 9.4. The Model Theorem} \

\bigskip
\noindent
The Model Theorem (cf. p. \pageref{MT}) says that if $T\in C_{0\, \bigcdot}$,
i.e., $T\in \mathcal {B(H)}$ is a contraction such that
$\lim_{n\to\infty} T^n x=0$ for each $x\in H$, then
$T$ is unitarily equivalent to the truncated backward shift
$S_{E}^*|_{\mathcal H(\Delta)}$
with the characteristic function $\Delta$ with values in
$\mathcal B(E^\prime, E)$, where
\begin{equation}\label{94001}
E=\hbox{cl ran}\,(I-T^*T).
\end{equation}
However, we were unable to determine $E^\prime$ in terms of spectral
properties of $T$ as in (\ref{94001}). \ In Theorem \ref{thm822}, we
give a necessary condition for ``$\dim\, E^\prime<\infty$." \ Thus, we would
like to pose:
\medskip

\begin{prob}\label{9444}
Let $T\in C_{0\, \bigcdot}$ and $\Delta\in H^\infty(\mathcal B(E^\prime,
E))$ be the characteristic function of $T$. \ For which operator
$T$, we have $\dim\, E^\prime <\infty$\,?
\end{prob}

\vskip 1cm

%
%
%
%
%

\noindent
{\bf \S\ 9.5. Cowen's Theorem and Abrahamse's Theorem} \

\bigskip
\noindent
For $\Phi\in L^\infty(\mathcal B(E))$, write
$$
\mathcal{E}(\Phi):=\Bigl\{ K\in H^\infty(\mathcal B(E)): \
\Phi-K\Phi^*\in H^\infty(\mathcal B(E))\ \ \hbox{and}\ \ ||K||_\infty\le 1 \Bigr\},
$$
i.e., $\mathcal E(\Phi)=\{K\in\mathcal C(\Phi):\ ||K||_\infty\le
1\}$ (cf. p.\pageref{cphi}). \ If $\dim\,E=1$ and $\Phi\equiv
\varphi$ is a scalar-valued function then an elegant theorem of C.
Cowen (cf. \cite{Co1}, \cite{NT}, \cite{CL}) says that $\mathcal
E(\varphi)$ is nonempty if and only if $T_\varphi$ is hyponormal,
i.e., the self-commutator $[T_\varphi^*, T_\varphi]$ is positive
semi-definite. \ Cowen's Theorem is to recast the operator-theoretic
problem of hyponormality into the problem of finding a solution of
an interpolation problem. \ In \cite{GHR}, it was shown that the
Cowen's theorem still holds for a Toeplitz operator $T_\Phi$ with a
matrix-valued {\it normal} (i.e., $\Phi^*\Phi=\Phi\Phi^*$) symbol
$\Phi\in L^\infty_{M_n}$.

\medskip

We then have:

\medskip

\begin{prob}\label{9661}
Extend Cowen's theorem for a Toeplitz operator with
an operator-valued normal symbol $\Phi\in L^\infty(\mathcal B(E))$\,.
\end{prob}

\bigskip

We recall that an operator $T\in\mathcal {B(H)}$ is called {\it
subnormal} if $T$ has a normal extension, i.e., $T=N\vert_{\mathcal
H}$, where $N$ is a normal operator on some Hilbert space $\mathcal
K\supseteq \mathcal H$ such that $\mathcal H$ is invariant for $N$.
\ In 1979, P.R. Halmos posed the following problem, listed as
Problem 5 in his Lecture ``Ten problems in Hilbert space"
(\cite{Ha3}, \cite{Ha4}): Is every subnormal Toeplitz operator
$T_\varphi$ with symbol $\varphi\in L^\infty$ either normal or
analytic (i.e., $\varphi\in H^\infty$) ? \ In 1984, C. Cowen and J.
Long \cite{CoL} have answered this question in the negative. \ To date, a characterization of subnormality of Toeplitz
operators $T_{\varphi}$ in terms of the symbols $\varphi$ has not been found. \ The
best partial answer to Halmos' Problem 5 was given by
M.B. Abrahamse: If $\varphi\in L^\infty$ is such that $\varphi$ or
$\overline\varphi$ is of bounded type, then $T_\varphi$ is either
normal or analytic; this is called Abrahamse's Theorem. \ Very
recently, in \cite[Theorem 7.3]{CHL3}, Abrahamse's Theorem was
extended to the cases of Toeplitz operators $T_\Phi$ with
matrix-valued symbols $\Phi$ under some constraint on the symbols
$\Phi$; concretely, when ``$\Phi$ has a tensored-scalar singularity."

We would like to pose:

\medskip

\begin{prob}\label{9662}
Extend Abrahamse's Theorem to Toeplitz operators $T_\Phi$ with
operator-valued symbols $\Phi\in L^\infty(\mathcal B(E))$.
\end{prob}

\vskip 1cm

%
%
%
%

\newpage

%
%


\end{document}